\documentclass[11pt]{article}
\title{\textbf{\Large A Reflected Moving Boundary Problem Driven by Space-Time White Noise}}
\author{Ben Hambly\footnote{ \href{mailto:ben.hambly@maths.ox.ac.uk}{ben.hambly@maths.ox.ac.uk}}\hspace{2mm} and Jasdeep Kalsi \footnote{\href{mailto:jasdeep.kalsi@maths.ox.ac.uk}{jasdeep.kalsi@maths.ox.ac.uk}}.\\\\
Mathematical Institute, University of Oxford\\\\}
\date{\today}

\usepackage{amsmath, amssymb, amsthm,bbm}
\usepackage{mathrsfs}
\usepackage{enumerate}
\newtheorem{thm}{Theorem}[section]
\newtheorem{lem}[thm]{Lemma}
\newtheorem{cor}[thm]{Corollary}
\newtheorem{defn}[thm]{Definition}

\newtheorem{prop}[thm]{Proposition}
\newtheorem{rem}[thm]{Remark}

\usepackage[margin=1in]{geometry}
\usepackage{graphicx}
\usepackage{epstopdf}
\usepackage{hyperref}
\usepackage{amsmath}
\usepackage{amsfonts}
\usepackage{setspace}
\usepackage{subcaption}
\newlength{\bibitemsep}
\newlength{\bibparskip}
\let\oldthebibliography\thebibliography
\renewcommand\thebibliography[1]{%
  \oldthebibliography{#1}%
  \setlength{\parskip}{\bibitemsep}%
  \setlength{\itemsep}{\bibparskip}%
}
\usepackage{indentfirst}
\usepackage{eqnarray}
\usepackage{enumerate}
\usepackage{commath}
\usepackage{amssymb}
\usepackage{afterpage}
\usepackage{bbm}
\usepackage{algorithm}
\usepackage{algpseudocode}
\usepackage{float}

\numberwithin{equation}{section}
\makeatletter

\newcommand{\xRightarrow}[2][]{\ext@arrow 0359\Rightarrowfill@{#1}{#2}}
\makeatother
\allowdisplaybreaks[1]
\begin{document}
\maketitle

\begin{abstract}
\noindent
We study a system of two reflected SPDEs which share a moving boundary. The equations describe competition at an interface and are motivated by the 
modelling of the limit order book in financial markets. The derivative of the moving boundary is given by a function of the two SPDEs in their relative 
frames. We prove existence 
and uniqueness for the equations until blow-up, and show that the solution is global when the boundary speed is bounded. We also derive the expected 
H\"{older} continuity for the process and hence for the derivative of the moving boundary. Both the case when the spatial domains are given by fixed 
finite distances from the shared boundary, and when the spatial domains are the semi-infinite intervals on either side of the shared boundary are considered. 
In the second case, our results require us to further develop the known theory for reflected SPDEs on infinite spatial domains by extending the uniqueness 
theory and establishing the local H\"{o}lder continuity of the solutions.
\end{abstract}

\section{Introduction}

There are many models for the behaviour of interfaces that arise in physical, biological and financial problems. In this paper we explore interfaces in one
dimension determined by competition between two types, which could be thought of as particles, species or offers to buy or sell, depending on the application. 
We think of a type as occupying a region on one side of the interface and evolving according to a reflected stochastic partial differential equation driven 
by white noise with a Dirichlet condition on the interface. The interface itself moves as a function of the profiles of the two types. One motivation is the 
evolution of the limit order book in a financial market. In this setting orders to buy or sell arrive at a rate determined by their distance to the best price 
and prices will rise or fall according to the order flow imbalance between the bid and ask sides of the book.
It is also possible to build biological models in which two species interact at an interface and their behaviour is determined by the distance of individuals 
from the interface. We do not focus on the modelling aspects, instead our aim is to establish existence, uniqueness and some properties of solutions to equations 
of the form
\begin{equation}\label{mmovingboundaryoriginal}\begin{split}
& \frac{\partial u^1}{\partial t}= \Delta u^1 + f_1(p(t)-x, u^1(t,x))+ \sigma_1(p(t)-x,u^1(t,x))\dot{W} + \eta^1 \\ 
& \frac{\partial u^2}{\partial t}= \Delta u^2 + f_2(x-p(t),u^2(t,x))+ \sigma_2(x-p(t),u^2(t,x))\dot{W} + \eta^2,
\end{split}
\end{equation}
where $u^1$ and $u^2$ have spatial domains on either side of the point $p(t)$ at any given time, and $\eta^1$, $\eta^2$ are reflection measures which 
keep the profiles positive. We impose Dirichlet boundary conditions so that $u^1(t,p(t))=u^2(t,p(t))=0$, with the point $p(t)$ evolving according to the equation 
\begin{equation}\label{boundary equation}
p^{\prime}(t)= h(u^1(t,p(t)- \cdot), u^2(t,\cdot-p(t))),
\end{equation}
where $h$ is a Lipschitz function mapping pairs of continuous functions to real numbers. The driving noise $\dot{W}$ here is space-time white 
noise, whilst the drift and diffusion coefficients $f_i$ and $\sigma_i$ for $i=1,2$ depend on the spatial coordinate in the frame relative to the boundary, 
as well as the value of the solution itself at that point. 
%This equation represents two reflected SPDEs which share a moving boundary, whose derivative 
%at a given time is a function of the two profiles $u^1$ and $u^2$ at that time. 

Within the class of equations produced by this model are approximations of the Stefan problem, where the motion of the boundary would be given by  
\begin{equation}
p^{\prime}(t)= \frac{\partial u^1}{\partial x}(t,p(t))- \frac{\partial u^2}{\partial x}(t,p(t)).
\end{equation}
The combination of the space-time white noise, moving boundary and the reflection measure make it difficult to find conditions which ensure differentiability 
of the profiles at the boundary, and so to arrive at an equation with precisely these dynamics. However, by choosing $h$ to be a function which 
emphasises the mass close to the interface, so that the boundary is still being moved by the ``relative pressure" of the two sides, we will have existence
and uniqueness for a system which approximates the Stefan problem. 

\subsection{Reflected SPDE and SPDEs with Moving Boundaries in the Literature}

Reflected SPDEs of the type 
\begin{equation}\label{reflll}
\frac{\partial u}{\partial t} = \Delta u + f(x,u)+ \sigma(x,u) \dot{W} + \eta,
\end{equation}
where $\dot{W}$ is space-time white noise and $\eta$ is a reflection measure, were initially studied in \cite{NP}. In \cite{NP}, the domain for the equation is $[0,T] \times [0,1]$, with Dirichlet boundary conditions imposed on $u$. Existence and uniqueness are established for the equation in the case of constant volatility i.e. $\sigma \equiv 1$. Existence for the equation in the case when $\sigma= \sigma(x,u)$ satisfies Lipschitz and linear growth conditions in its second argument was then established in \cite{DMP} using a penalization method. 
%This uses a penalisation method, where the problem is approximated by a sequence of equations 
%\begin{equation}\label{penalise}
%\frac{\partial u^{\epsilon}}{\partial t} = \Delta u^{\epsilon} + f(x,u^{\epsilon})+ \sigma(x,u^{\epsilon}) \dot{W} + \frac{1}{\epsilon}(u^{\epsilon})^-.
%\end{equation}
%The last term here provides a strong upward drift which acts when the solution becomes negative. It is shown that the solutions to (\ref{penalise}) increase to a solution of (\ref{reflll}) as $\epsilon \downarrow 0$. 
The penalization approach was adapted in \cite{Otobe}, in which the corresponding result is proved in the case when the domain for equation (\ref{reflll}) is $[0,T] \times \mathbb{R}$. Uniqueness for varying volatility $\sigma= \sigma(x,u)$ on compact spatial domains was then proved by in \cite{Xu}. The authors decouple the obstacle problem and SPDE components of the problem. This allows them to prove existence via a two-step Picard iteration, as well as uniqueness. 

Similarly, there has been much recent work on moving boundary problems for SPDEs. In \cite{KK}, existence and uniqueness for solutions to a Stefan problem for an SPDE driven by spatially coloured noise is proved. The corresponding problem in the case when the SPDE is driven by space-time white noise was then studied in \cite{Zheng} under the condition that the volatility vanishes quickly enough at the moving interface. More recent work on such problems include the models in \cite{Muller} and in \cite{KM}. In these papers the focus is on essentially the same equations as \eqref{mmovingboundaryoriginal} but without reflection and with coloured noise. Different boundary conditions are imposed at the interface in the two papers and in particular they are able to include a Brownian motion in the dynamics for the motion of the boundary. When thinking of the equations in \cite{Muller, KM} as models for the 
limit order book, the incorporation of a Brownian term ensures that the resulting price process is a semi-martingale.

\subsection{Main Results and Contributions}

In this paper we combine aspects of the models discussed above and consider the system of two reflected SPDEs sharing a moving boundary, 
(\ref{mmovingboundaryoriginal}). We examine the problem in two cases- firstly when the spatial domain is restricted to a fixed distance from the 
moving boundary and secondly when the spatial domain is infinite in both directions. Existence and uniqueness for the system is proved 
in both of these cases. Our approach is similar to that of \cite{Xu}. As in \cite{Xu}, we decouple the problem into studying a deterministic obstacle 
problem and applying SPDE estimates. The non-Lipschitz term created by the moving boundary is controlled by a suitably truncated version of the 
problem, for which existence and uniqueness is proved by a Picard argument. Consistency among the solutions of the truncated problems allows us 
to piece these together to obtain a solution to our original problem which exists until some blow-up time. 

In the case when the spatial domain is infinite, our analysis extends the known uniqueness theory for reflected SPDEs on infinite spatial domains. 
In \cite{Otobe}, uniqueness for the reflected stochastic heat equation on $\mathbb{R}$ is proved in the case when $\sigma \equiv 1$. We obtain 
uniqueness for a class of volatility coefficients which are allowed to depend on space and the solution itself at that particular point in space-time. The 
main condition here is that the coefficient $\sigma$ is Lipschitz in its second argument, with a Lipschitz constant which decays exponentially fast in 
the spatial variable. 

The local H\"{o}lder continuity of the solutions to our equations is also established, in both the case when the spatial domain is finite and when it is 
infinite. As one might expect, solutions can be shown to be up to $\frac{1}{4}$-H\"{o}lder continuous in time and up to $\frac{1}{2}$- H\"{o}lder 
continuous in space. In the case of the infinite spatial domain, this is a new result even for static reflected SPDEs. We argue by suitably adapting the 
proof in \cite{Dalang2}, in which the corresponding result was proved for solutions to static reflected SPDEs on compact spatial domains. As a corollary, 
we can also show that the derivative of the boundary is up to $\frac{1}{4}$-H\"{o}lder continuous in time, with this regularity inherited from the solution 
and a Lipschitz condition on the function $h$ in (\ref{boundary equation}).

An outline of the paper is as follows. In Section~2 we will discuss the deterministic obstacle problem, a key ingredient in working with reflected processes. 
In Section~3 we look at the case where the SPDEs are supported on finite intervals in the moving frame. The case of semi-infinite intervals in the moving 
frame is the topic of Section~4. The heat kernel estimates necessary for obtaining the SPDE estimates as well as the proofs for Section~2 are in the appendix.

\subsection{An Application: Limit Order Books}

The majority of modern trading takes place in limit order markets. In a limit order market, all traders are able to place orders of three types. Limit 
orders are offers to buy/sell which do not lead to an immediate transaction; they only result in a transaction when they are matched with incoming 
market orders. Market orders are offers to buy/sell the asset which match with an existing limit order and so result in immediate transaction. Finally, 
traders are able to cancel limit orders which they previously placed. The order book itself at a given time is simply the record of unexecuted, 
uncancelled limit orders at that time. There has been much interest in trying to model the dynamics of the book, particularly in a high frequency setting. 

As in \cite{Muller} and \cite{KM}, we can think of our equations (\ref{mmovingboundaryoriginal}) as being a model for the limit order book. In this context, we would think of the spatial variable as representing the price or the log-price. The random fields $u^1(t,x)$ and $u^2(t,x)$ would then be the density of limit orders existing at price $x$ and time $t$ on the bid and ask sides of the book respectively, and we can choose the function $h$ to represent an approximation of the local imbalance at the mid price. The volatility terms in our equations can naturally be thought of as the presence of high frequency trading in the model, which together with the drift term models the arrivals and cancellations of limit orders. The price process is then driven by the relative pressure of the existing orders on the two sides of the book. A potential advantage of our model over similar models in the literature is the presence of the reflection measure, which ensures that order volumes remain positive without the need for a multiplicative term in front of the noise. 

\begin{figure}\label{Price}
\begin{center}
\includegraphics[scale=0.6]{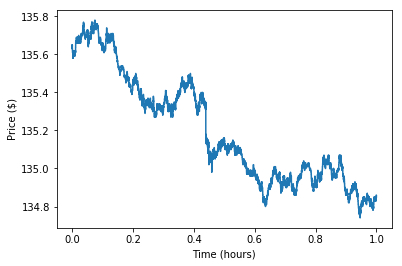} 
\caption{Bid price process for the SPDR Trust Series I on June 21 2012 between 09:30:00.000 and 10:30:00.000 EST}
\end{center}
\end{figure}

\begin{figure}\label{Price2}
\begin{center}
\includegraphics[scale=0.6]{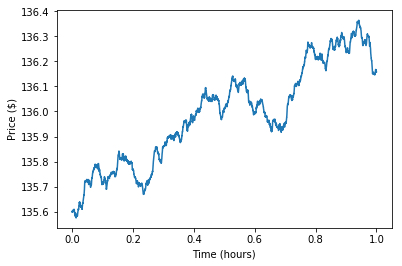} 
\caption{Simulated Price Process}
\end{center}
\end{figure}

Figures 1 and 2 are included here for illustrative purposes. We make use of data provided by the LOBSTER (Limit Order Book System, The Efficient Reconstructor) database. Figure 1 shows the evolution of the bid price process for the SPDR Trust Series I on June 21 2012 between 09:30:00.000 and 10:30:00.000 EST. Figure 2 displays a simulated price process obtained from our moving boundary model on the compact spatial domain $[0,1]$. Note that space here is on the scale of dollars i.e. a spatial interval of size $1$ represents a price interval of size $\$ 1$.  We use a simple forward Euler finite difference scheme in order to produce the simulation. Using the data of incoming market orders, limit orders and cancellations of the SPDR Trust Series I at the different relative prices over the same time period as in Figure 1, we fit the drift and volatility coefficients (which we assume to depend on the relative price only) of our SPDEs. When fitting these, we assume that these parameters are symmetric for the two sides of the order book, which is the case for this particular dataset up to a small error. Note that, for smooth functions $k$ and large $\lambda >0$, the functional $g_{\lambda}$ given by
\begin{equation}
g_{\lambda}(k):= \int_0^1 \lambda^2e^{-\lambda x} k(x) \; \textrm{d}x
\end{equation}
places most emphasis on the mass of $k$ near zero. It is also the case that $g_{\lambda}(k) \rightarrow k^{\prime}(0)$ as $\lambda \rightarrow \infty$. Recall that the functional $h$ in our equations determines the derivative of the boundary as a function of the two sides of the book, as in (\ref{boundary equation}). For our simulation, we make the choice
\begin{equation}
h(u^1,u^2)=  \alpha g_{\lambda}(u^1-u^2),
\end{equation}
with $\alpha =5$ and $\lambda = 100$. The boundary movement is therefore driven primarily by the local imbalance of offers to buy and sell 
at the mid, and approximates a Stefan condition. We also remark here that the Laplacian terms in our simulation were scaled down by a factor 
of $0.2$ in order to ensure that the order book profiles obtained by the simulation are on the correct scale.  

We note that the boundary motion for our equations can be shown to be $C^{1,\alpha}$ for $\alpha < 1/4$. However, by choosing $h$ to approximate 
the derivative at the boundary and looking at the price process over sufficiently long timescales, we can generate price processes that are rough at the 
appropriate scale as can be seen in Figure 2. 

\section{A Deterministic Parabolic Obstacle Problem}

In this section we will define and state some simple results for deterministic parabolic obstacle problems on both the compact spatial interval $[0,1]$ and the infinite interval $[0,\infty)$. The intuition for these equations is that the solutions solve the heat equation with a constraint that they must lie above some predetermined continuous function of time and space. It is important to note that both the obstacle and the solutions will be continuous functions here. In addition, being able to control differences in the solution by differences in the obstacles will be key in allowing us to introduce the reflection component in our SPDEs later. 

\subsection{The Deterministic Obstacle Problem on $[0,1]$}

The obstacle problem on the compact interval $[0,1]$ was originally discussed in \cite{NP}, in which the proofs can be found. 

\begin{defn}
Let $v \in C([0,T] ; C_0((0,1)))$ with $v(0,\cdot) \leq 0$. We say that the pair $(z, \eta)$ satisfies the heat equation with obstacle $v$ if:

\begin{enumerate}[(i)]
\item $z \in C([0,T] \times [0,1])$, $z(t,0)=z(t,1)=0$, $z(0,x) \equiv 0$ and $z \geq v$.
\item $\eta$ is a measure on $(0,1) \times [0,T]$.
\item $z$ weakly solves the PDE 
\begin{equation}
\frac{\partial z}{\partial t}= \frac{\partial^2 z}{\partial x^2} + \eta
\end{equation}
That is, for every $t \in [0,T]$ and every $\phi \in C^2((0,1)) \cap C_0((0,1))$, $$\int_0^1 z(t,x) \phi(x) \textrm{d}x= \int_0^t \int_0^1 z(s,x)\phi^{\prime \prime}(x) \textrm{d}x\textrm{d}s + \int_0^t \int_0^1 \phi(x) \eta(\textrm{d}x,\textrm{d}s).$$

\item $\int_0^t \int_0^1 (z(s,x)-v(s,x)) \; \eta(\textrm{d}x,\textrm{d}s)=0$.

\end{enumerate}
\end{defn}

The following result from \cite{NP} gives us existence and uniqueness for solutions to this problem. It is also proved that we can bound the difference between two solutions in the $L^{\infty}$-norm by the difference in the $L^{\infty}$-norm of the obstacles. This will be very helpful when proving estimates later, as it will allow us to control our reflected SPDE by an unreflected SPDE.

\begin{thm}[\cite{NP}, Theorem 1.4]
For every $v \in C([0,T] ;  C_0((0,1)))$ with $v(0, \cdot) \leq 0$, there exists a unique solution to the above obstacle problem. In addition, if $(z_1, \eta_1)$ and $(z_2, \eta_2)$ solve the obstacle problems with obstacles $v_1$ and $v_2$ respectively, then for $t \in [0,T]$, we have the estimate 
\begin{equation}\label{Obstacle[0,1]}
\|z_1-z_2\|_{\infty,t} \leq \|v_1-v_2\|_{\infty,t}\end{equation} where $\|.\|_{\infty,t}$ is defined for $u \in C([0,T] \times [0,1])$ by $$\|u\|_{\infty,t}:=\sup\limits_{s \in [0,t]} \sup\limits_{x \in [0,1]} |u(s,x)|.$$
\end{thm}

\subsection{The Deterministic Obstacle Problem on $[0,\infty)$}

Before discussing the obstacle problem in this section, we first introduce the relevant function spaces which we will be working on.

\begin{defn}
For $r \in \mathbb{R}$, we say that $u:[0,\infty) \rightarrow \mathbb{R}$ is in the space $\mathscr{L}_r$ if
\begin{equation}
\|u\|_{\mathscr{L}_r}:= \sup\limits_{x \geq 0} e^{-rx}|u(x)| < \infty.
\end{equation}
We say that $u \in \mathscr{C}_r$ if $u \in \mathscr{L}_r$ and $u$ is continuous. $\mathscr{C}_r$ is equipped with the same norm as $\mathscr{L}_r$
\end{defn}

\begin{defn}
We say that $u:[0,T] \times [0,\infty) \rightarrow \mathbb{R}$ is in the space $\mathscr{C}_r^T$ if $u$ is continuous and 
\begin{equation}
\sup\limits_{t \in [0,T]} \sup\limits_{x \geq 0} e^{-rx}|u(t,x)|
\end{equation}
\end{defn}

We are now in position to define the obstacle problem in this setting.

\begin{defn}
Fix some $r \in \mathbb{R}$. Let $v \in \mathscr{C}_r^T$ such that $v(t,0)=0$ for every $t \in [0,T]$. We say that the pair $(z, \eta)$ satisfies the heat equation with obstacle $v$ and exponential growth $r$ on $[0,\infty)$ if:

\begin{enumerate}[(i)]
\item $z \in \mathscr{C}_r^T$, $z(t,0)=0$, $z(0,x)=0$ and $z \geq v$.
\item $\eta$ is a measure on $(0,\infty) \times [0,T]$.
\item $z$ weakly solves the PDE
\begin{equation}
\frac{\partial z}{\partial t}= \frac{\partial^2 z}{\partial x^2} + \eta.
\end{equation}
That is, for every $t \in [0,T]$ and every $\phi \in C_c^{\infty}([0,\infty))$ with $\phi(0)=0$, $$\int_0^{\infty} z(t,x) \phi(x) \textrm{d}x= \int_0^t \int_0^{\infty} z(s,x)\phi^{\prime \prime}(x) \textrm{d}x\textrm{d}s + \int_0^t \int_0^{\infty} \phi(x) \eta(\textrm{d}x,\textrm{d}s).$$
\item $\int_0^t \int_0^{\infty} (z(s,x)- v(s,x)) \; \eta(\textrm{d}x,\textrm{d}s)=0$.
\end{enumerate}
\end{defn}

We note that the deterministic obstacle problem on the spatial domain $\mathbb{R}$ is considered in \cite{Otobe}. Here, we pose the problem in $\mathscr{C}_r^T$ for any $r \in \mathbb{R}$, and we work on the spatial domain $[0,\infty)$ rather than $\mathbb{R}$. A proof of the following result is provided in the appendix.

\begin{thm}\label{Obstacle INfinite}
For every $r \in \mathbb{R}$ and every $v \in \mathscr{C}_r^T$ such that $v(t,0)=0$, there exists a unique solution $(z, \eta)$ to the heat equation on $[0, \infty)$ with Dirichlet condition and obstacle $v$. Furthermore, if $v_1, v_2 \in \mathscr{C}_r^T$, we have that 
\begin{equation}
\|z_1-z_2\|_{\mathscr{C}_r^T} \leq C_{r,T} \|v_1-v_2\|_{\mathscr{C}_r^T},
\end{equation}
where $z_i$ is the solution to the obstacle problem corresponding to $v_i$.
\end{thm}

\section{The Moving Boundary Problem on Finite Intervals in the Relative Frame}

We are interested in the following reflected moving boundary problem:

\begin{equation}\label{movingboundaryoriginal}\begin{split}
& \frac{\partial u^1}{\partial t}= \Delta u^1 + f_1(p(t)-x, u^1(t,x))+ \sigma_1(p(t)-x,u^1(t,x))\dot{W} + \eta^1 \\ 
& \frac{\partial u^2}{\partial t}= \Delta u^2 + f_2(x-p(t),u^2(t,x))+ \sigma_2(x-p(t),u^2(t,x))\dot{W} + \eta^2,
\end{split}
\end{equation}
where $u^1$ and $u^2$ satisfy Dirichlet boundary conditions enforcing that they are zero at $p(t)$, with the point $p(t)$ evolving according to the equation 
\begin{equation*}
p^{\prime}(t)= h(u^1(t,p(t)- \cdot), u^2(t,p(t)+ \cdot)).
\end{equation*}
Here, $\dot{W}$ is a space-time white noise and $h: C([0,1])^2 \mapsto \mathbb{R}$, satisfies certain conditions which we will introduce later. $(\eta^1,\eta^2)$ are reflection measures for the functions $u^1$ and $u^2$ respectively, keeping the profiles positive and satifying the conditions 
\begin{enumerate}[(i)]
\item $\int_0^{\infty} \int_{\mathbb{R}} u^1(t,x) \;  \eta^1(\textrm{d}t,\textrm{d}x)=0,$ and 
\item $\int_0^{\infty} \int_{\mathbb{R}} u^2(t,x) \;  \eta^2(\textrm{d}t,\textrm{d}x)=0.$
\end{enumerate} 
In this section, we consider the case when the functions $u^1$ and $u^2$ are supported in the sets $\left\{ (t,x)\in [0,\infty) \times \mathbb{R} \; | \; x \in [p(t)-1,p(t)] \right\}$ and $\left\{ (t,x)\in [0,\infty) \times \mathbb{R} \; | \; x \in [p(t),p(t)+1] \right\}$ respectively.

\subsection{Formulation of the Moving Boundary Problem}

We would like to formalise what we mean by (\ref{movingboundaryoriginal}) in the compact case. Before doing so, we define what we mean for a space-time white noise to respect a given filtration. This will be useful in some of the measurability arguments which follow.

\begin{defn}
Let $(\Omega, \mathscr{F}, \mathscr{F}_t, \mathbb{P})$ be a complete filtered probability space. Suppose that $\dot{W}$ is a space-time white noise defined on this space. Define for $A \in \mathscr{B}(\mathbb{R})$, $$W_t(A):= \dot{W}([0,t] \times A).$$ We say that $\dot{W}$ respects the filtration $\mathscr{F}_t$ if $(W_t(A))_{t \geq 0, A \in \mathscr{B}(\mathbb{R})}$ is an $\mathscr{F}_t$- martingale measure i.e.  if for every $A \in \mathscr{B}(\mathbb{R})$, $(W_t(A))_{t \geq 0}$ is an $\mathscr{F}_t$-martingale.
\end{defn}

Let $(\Omega, \mathscr{F}, \mathbb{P})$ be a complete filtered probability space, and $\dot{W}$ space-time white noise on $\mathbb{R}^+ \times \mathbb{R}$. Let $\mathscr{F}_t$ be the filtration generated by the white noise, so that $\mathscr{F}_t = \sigma(\left\{ W(s,x) \; | \; x \in \mathbb{R}, s \leq t \right\} ).$ Suppose that $(u^1,\eta^1, u^2, \eta^2, p)$ is an $\mathscr{F}_t$-adapted process solving (\ref{movingboundaryoriginal}). Then $p:\mathbb{R}^+ \times \Omega \mapsto \mathbb{R}$ is a $\mathscr{F}_t$-adapted process such that the paths of $p(t)$ are almost surely $C^1$ (note that, in particular, $p$ is $\mathscr{F}_t$-predictable).  Let $\varphi \in C_c^{\infty}([0,\infty) \times (0,1))$, and define the function $\phi$ by setting $\phi(t,x)= \varphi(t, p(t)+x)$. By multiplying the equation for $u^1$ in (\ref{movingboundaryoriginal}) by such a $\phi$ and integrating over space and time, interpreting the derivatives in the usual weak sense, we obtain the expression
\begin{equation*}\begin{split}
\int_{\mathbb{R}} u^1(t,x) \phi(t,x) \textrm{d}x= & \int_{\mathbb{R}} u^1(0,x) \phi(0,x) \textrm{d}x +  \int_0^t \int_{\mathbb{R}} u^1(s,x) \frac{\partial \phi}{\partial t}(s,x) \textrm{d}x\textrm{d}s \\ & + \int_0^t \int_{\mathbb{R}}  u^1(s,x) \frac{\partial^2 \phi}{\partial x^2}(s,x) \textrm{d}x\textrm{d}s  + \int_0^t \int_{\mathbb{R}}  f_1(p(s)-x,u^1(t,x)) \phi(s,x) \textrm{d}x\textrm{d}s \\ & + \int_0^t \int_{\mathbb{R}} \sigma_1(p(s)-x,u^1(t,x)) \phi(s,x)W(\textrm{d}x,\textrm{d}s)\\  & + \int_0^t \int_{\mathbb{R}} \phi(s,x) \eta^1(\textrm{d}s,\textrm{d}x). 
\end{split}
\end{equation*}
We now introduce a change in the spatial variable in order to associate our problem with a fixed boundary problem. Setting $v^1(t,x)= u^1(t, p(t)-x)$, the above equation becomes
\begin{equation*}
\begin{split}
\int_0^1 v^1(t,x) \phi(t,p(t)-x) \textrm{d}x= & \int_0^1 v^1(0,x) \phi(0,p(0)-x) \textrm{d}x +  \int_0^t \int_0^1 v^1(s,x) \frac{\partial \phi}{\partial t}(s,p(s)-x) \textrm{d}x\textrm{d}s \\ &+ \int_0^t \int_0^1  v^1(s,x) \frac{\partial^2 \phi}{\partial x^2}(s,x) \textrm{d}x\textrm{d}s
+ \int_0^t \int_0^1 f_1(x,v^1(s,x)) \phi(s,p(s)-x) \textrm{d}x\textrm{d}s \\ &+ \int_0^t \int_0^1 \sigma_1(x,v^1(s,x)) \phi(s,p(s)-x)W_p(\textrm{d}x,\textrm{d}s)\\ &+ \int_0^t \int_0^1  \phi(s,p(s)-x) \eta_p^1(\textrm{d}x,\textrm{d}s). 
\end{split}
\end{equation*}
Here, $\dot{W}_p$ and $\eta_p^1$ are obtained by from $W$ and $\eta$ by shifting by $p(t)$. That is, for $t \in \mathbb{R}^+$ and $A \in \mathscr{B}( \mathbb{R})$,
\begin{equation}
\dot{W}_p([0,t] \times A)= \int_0^t \int_{A+p(s)} W(\textrm{d}s,\textrm{d}y), \; \; \; \; \; \; \eta_p^1([0,t] \times A)= \int_0^t \int_{A+p(s)} \eta^1(\textrm{d}y,\textrm{d}s), 
\end{equation}
Note that, since the process $p(t)$ is $\mathscr{F}_t$-predictable, $\dot{W}_p$ is then also a space time white noise which respects the filtration $\mathscr{F}_t$. Also, $\eta_p^1$ is a reflection measure for $v$, so that
\begin{equation}
\int_0^T \int_0^1 v^1(t,x) \;  \eta_p^1(\textrm{d}t,\textrm{d}x)=0.
\end{equation} 
We can calculate 
\begin{equation}\label{changeofvar}
\frac{\partial \phi}{\partial t}(s,x)= \frac{\partial \varphi}{\partial t}(t,p(t)+x) + p^{\prime}(t)\frac{\partial \varphi}{\partial x}(t,p(t)+x).\end{equation}
It therefore follows that 
\begin{equation*}
\begin{split}
\int_0^1 v^1(t,x) \varphi(t,x) \textrm{d}x= & \int_0^1 v^1(0,x) \varphi(0,x) \textrm{d}x +  \int_0^t \int_0^1 v^1(s,x) \frac{\partial \varphi}{\partial t}(s,x) \textrm{d}x\textrm{d}s \\ &+ \int_0^t \int_0^1 v^1(s,x)p^{\prime}(s) \frac{\partial \varphi}{\partial x}(s,x) \textrm{d}x\textrm{d}s+ \int_0^1 \int_0^t v^1(s,x) \frac{\partial^2 \varphi}{\partial x^2}(s,x) \textrm{d}x\textrm{d}s
\\ & + \int_0^1 \int_0^t f_1(x,v^1(s,x)) \varphi(s,x) \textrm{d}x\textrm{d}s+ \int_0^1 \int_0^t \sigma_1(x,v^1(s,x)) \varphi(s,x)W_p(\textrm{d}x,\textrm{d}s) \\ &+ \int_0^1 \int_0^t \varphi(s,x) \eta_p^1(ds,dx).
\end{split}
\end{equation*}
We can perform similar manipulations to obtain a weak form for $v^2(t,x):= u^2(t,p(t)+x)$. This yields that for test functions $\varphi \in C_c^{\infty}([0,\infty) \times (0,1))$, we should have that 
\begin{equation*}
\begin{split}
\int_0^1 v^2(t,x) \varphi(t,x) \textrm{d}x= & \int_0^1 v^2(0,x) \varphi(0,x) \textrm{d}x +  \int_0^t \int_0^1 v^2(s,x) \frac{\partial \varphi}{\partial t}(s,x) \textrm{d}x\textrm{d}s \\ & - \int_0^t \int_0^1 v^2(s,x)p'(s) \frac{\partial \varphi}{\partial x}(s,x) \textrm{d}x\textrm{d}s \\ & + \int_0^1 \int_0^t v^2(s,x) \frac{\partial^2 \varphi}{\partial x^2}(s,x) \textrm{d}x\textrm{d}s
+ \int_0^1 \int_0^t f_2(x,v^2(s,x)) \varphi(s,x) \textrm{d}x\textrm{d}s \\ & + \int_0^1 \int_0^t \sigma_2(x,v^2(s,x)) \varphi(s,x)W_p^-(\textrm{d}x,\textrm{d}s)+ \int_0^1 \int_0^t \varphi(s,x) \eta_p^2(ds,dx),
\end{split}
\end{equation*}
where $\dot{W}_p^-$ is given by
 \begin{equation*}
\dot{W}_p^-([0,t] \times A)= \dot{W}_p([0,t] \times (-A)).
\end{equation*}
Note also that, since $(u^1,\eta^1,u^2,\eta^2,p)$ is $\mathscr{F}_t$-adapted, we know that $(v^1,\eta^1_p, v^2, \eta^1_p, p)$ is also $\mathscr{F}_t$-adapted. 

\begin{rem}
By noting that for $i=1,2$, $\dot{W}^i_p$ respect the filtration $\mathscr{F}_t$, we ensure that our solutions cannot ``see the future'' of the space-changed driving noises $\dot{W}_p^i$. It of course makes sense intuitively that this should be the case, since they are measurable with respect to the filtration generated by $\dot{W}$. We would expect that the solution is in fact adapted to the filtration generated by the noises $\dot{W}_p^i$. We see this indirectly, when we later prove that in any filtered space with a space-time white noise which respects the filtration, there exists a unique solution to the problem. As we can choose to take the filtration to be generated by the noise, the unique solution in an enlarged space must be adapted to the noise.
\end{rem}

\begin{rem}
The above formulation would need to be adjusted if we were anticipating rough paths for $p(t)$. For example, if $p(t)$ were a semimartingale with a non-zero diffusion component, we would have to apply $It\hat{o}$'s formula for the change of variables (\ref{changeofvar}), which would change our weak form. 
\end{rem}

We now define what we mean by a solution to a particular class of reflected SPDEs. The preceding calculation will allow us to connect the solutions to these SPDEs to our moving boundary problem.

\begin{defn}

Let $(\Omega, \mathscr{F}, \mathscr{F}_t,\mathbb{P})$ be a complete filtered probability space. Let $\dot{W}$ be a space time white noise on this space which respects the filtration $\mathscr{F}_t$. Suppose that $\tilde{v}$ is a continuous $\mathscr{F}_t$-adapted process taking values in $C_0(0,1)$. Let $h: C_0(0,1) \times C_0(0,1)\rightarrow C_0(0,1)$ and $F: C_0(0,1) \rightarrow C_0(0,1)$ be Lipschitz functions. For the $\mathscr{F}_t$-stopping time $\tau$, we say that the pair $(v, \eta)$ is a local solution to the reflected SPDE
\begin{equation*}
\frac{\partial v}{\partial t}= \Delta v + h(v, \tilde{v}) \frac{\partial F(v)}{\partial x} + f(x, v) + \sigma(x,v) \dot{W} + \eta
\end{equation*}
with Dirichlet boundary conditions $v(t,0)=v(t,1)=0$ and initial data $v_0 \in C_0(0,1)^+$, until time $\tau$, if
\begin{enumerate}[(i)]
\item For every $x \in [0,1]$ and every $t \in [0,\infty)$, $v(t,x)$ is $\mathscr{F}_t$-measurable.
\item $v \geq 0$ almost surely.
\item $v \big|_{[0,\tau) \times [0,1]} \in C([0, \tau) \times [0,1])$ almost surely.
\item $v(t,x)= \infty$ for every $t \geq \tau$ almost surely.
\item $\eta$ is a measure on $(0,1) \times [0,\infty)$ such that
\begin{enumerate}
\item For every measurable map $\psi: [0,1] \times [0,\infty) \rightarrow \mathbb{R}$, 
\begin{equation}
\int_0^t \int_0^1 \psi(x,s) \; \eta(\textrm{d}x,\textrm{d}s)
\end{equation}
is $\mathscr{F}_t$-measurable.
\item $\int_0^{\infty} \int_0^1 v(t,x) \; \eta(\textrm{dx,dt})=0$.
\end{enumerate} 
\item There exists a localising sequence of stopping times $\tau_n \uparrow \tau$ almost surely, such that for every  $\varphi \in C^{1,2}_c([0,\infty) \times [0,1])$ such that $\varphi(s,0)=\varphi(s,1)=0$ for every $s \geq 0$, 
\begin{equation}\label{weak}
\begin{split}
\int_0^1 v(t \wedge \tau_n,x) \varphi(t\wedge \tau_n,x) \textrm{d}x= & \int_0^1 v(0,x) \varphi(0,x) \textrm{d}x+  \int_0^{t \wedge \tau_n} \int_0^1 v(s,x) \frac{\partial \varphi}{\partial t}(s,x)\textrm{d}x\textrm{d}s \\ & +  \int_0^{t \wedge \tau_n} \int_0^1  v(s,x) \frac{\partial^2 \varphi}{\partial x^2}(s,x) \textrm{d}x\textrm{d}s \\ & - \int_0^{t \wedge \tau_n} \int_0^1 F(v(s,\cdot))(x)h(v(s,\cdot),\tilde{v}(s,\cdot)) \frac{\partial \varphi}{\partial x}(s,x) \textrm{d}x\textrm{d}s\\ & + \int_0^{t \wedge \tau_n}\int_0^1 f(x,v(s,x)) \varphi(s,x) \textrm{d}x\textrm{d}s \\ & + \int_0^{t \wedge \tau_n} \int_0^1  \sigma(x,v(s,x)) \varphi(s,x)W(\textrm{d}x,\textrm{d}s) \\ & + \int_0^{t \wedge \tau_n} \int_0^1 \varphi(s,x) \; \eta(ds,dx).
\end{split}
\end{equation}
for every $t \geq 0$ almost surely.
\end{enumerate}
\end{defn}
We say that a local solution is \emph{maximal} if it cannot be extended to a solution on a larger stochastic interval. We say that a local solution is \emph{global} if we can take $\tau_n= \infty$ in (\ref{weak}).

\begin{rem}
Condition (iv) above is included for the purposes of discussing uniqueness only.
\end{rem}

Before stating the formal definition for our moving boundary problem we introduce some notation which will allow us to easily write down the profiles in suitable relativised coordinates. 
\begin{defn}\label{frameshift}
For $p_0 \in \mathbb{R}$, we define $\Theta_{p_0}^1: \mathbb{R} \rightarrow \mathbb{R}$  such that 
$$\Theta_{p_0}^1(x) = p_0 -x.$$
For a function $p: [0,\infty) \rightarrow \mathbb{R}$ we then define $\theta^1_p: [0,\infty) \times \mathbb{R} \rightarrow [0,\infty) \times \mathbb{R}$ such that 
$$\theta^1_p(t,x):= (t,\Theta_{p(t)}^1(x)).$$
We similarly define $\Theta^2_{p_0} : \mathbb{R} \rightarrow \mathbb{R}$ such that $$\Theta_{p_0}^2(x) = x- p_0,$$ and $\theta^2_p : [0,\infty) \times \mathbb{R} \rightarrow [0,\infty) \times \mathbb{R}$ such that 
$$\theta^2_p(t,x):= (t,\Theta_{p(t)}^2(x)).$$
\end{defn}
\begin{defn}
For a space time white noise $\dot{W}$, we denote by $\dot{W}^-$ the space time white noise such that $\dot{W}^-([0,t] \times A) = \dot{W}([0,t] \times (-A))$.
\end{defn}
\begin{defn}
Let $(\Omega, \mathscr{F}, \mathscr{F}_t,\mathbb{P})$ be a complete filtered probability space. Let $\dot{W}$ be a space time white noise on this space which respect the filtration $\mathscr{F}_t$. We say that the quintuple $(u^1, \eta^1, u^2, \eta^2, p)$ is a local solution to the moving boundary problem with initial data $(u^1_0,u^2_0,p_0)$, where $(u_0^1 \circ (\Theta_{p_0}^1)^{-1}, u_0^2 \circ (\Theta_{p_0}^2)^{-1}) \in C_0((0,1))^+ \times C_0((0,1))^+$, up to the $\mathscr{F}_t$-stopping time $\tau$ if 
\begin{enumerate}[(i)]
\item $p(0) = p_0$ and $p^{\prime}(t)$ is $\mathscr{F}_t$-adapted, with $p^{\prime}(t)=h(v^1(t,\cdot),v^2(t,\cdot))$.
\item $(v^1, \tilde{\eta}^1):=(u^1 \circ (\theta^1_p)^{-1}, \eta^1 \circ (\theta^1_p)^{-1})$ solves the reflected SPDE
\begin{equation}
\frac{\partial v^1}{\partial t}= \Delta v^1 - p^\prime(t)\frac{\partial v^1}{\partial x} + f_1(x,v^1) + \sigma_1(x,v^1)\dot{W} + \tilde{\eta}^1 
\end{equation}
with Dirichlet boundary conditions $v^1(0)=v^1(1)=0$ and initial data $v^1_0 = u^1_0 \circ (\Theta_{p_0}^1)^{-1}$ until time $\tau$.
\item $(v^2, \tilde{\eta}^2):=(u^2 \circ (\theta^2_p)^{-1}, \eta^2 \circ (\theta^2_p)^{-1})$ solves the reflected SPDE
\begin{equation}
\frac{\partial v^2}{\partial t}= \Delta v^2 + p^\prime(t)\frac{\partial v^2}{\partial x} + f_2(x,v^2) + \sigma_2(x,v^2)\dot{W}^- + \tilde{\eta}^2 
\end{equation}
with Dirichlet boundary conditions $v^2(0)=v^2(1)=0$ and initial data $v^2_0 = u^2_0 \circ (\Theta_{p_0}^2)^{-1}$ until time $\tau$.
\end{enumerate}
We refer to $(v^1, \tilde{\eta}^1, v^2, \tilde{\eta}^2)$ as the solution to the moving boundary problem in the relative frame.
\end{defn}

We now introduce the precise conditions on the coefficients. We suppose that for $i=1,2$, $f_i , \sigma_i$ are measurable mappings $f_i, \sigma_i : [0,1] \times \mathbb{R}^+ \rightarrow \mathbb{R}$ and that $h: C_0((0,1))^2 \rightarrow \mathbb{R}$ is a measurable function such that 
\begin{enumerate}[(I)]
\item For every $x \in [0,1]$, $u,v \in \mathbb{R}^+$, 
\begin{equation*}
|f_i(x,u)-f_i(x,v)| + |\sigma_i(x,u)-\sigma_i(x,v)| \leq C |u-v|
\end{equation*}
for some constant $C$.
\item $|f_i(x,u)|+ |\sigma_i(x,u)| \leq R(1+ |u|)$ for some constant $R$.
\item $h$ is bounded on bounded sets.
\item $h$ is Lipschitz, so there exists a constant $K$ such that for every $u_1,u_2,v_1,v_2 \in C_0((0,1))$, $|h(u_1,v_1)-h(u_2,v_2)| \leq K(\|u_1-u_2\|_{\infty}+\|v_1-v_2\|_{\infty})$.
\end{enumerate}

\subsection{Existence and Uniqueness}

\begin{thm}\label{Main}
Let $(\Omega, \mathscr{F}, \mathscr{F}_t,\mathbb{P})$ be a complete filtered probability space. Let $\dot{W}$ be a space time white noise on this space which respects the filtration $\mathscr{F}_t$. Suppose that $f, \sigma$ and $h$ satisfy the conditions (I)-(IV). Then there exists a unique maximal solution $(u^1, \eta^1, u^2, \eta^2, p)$ to the moving boundary problem. The blow-up time, $\tau$, is given by
\begin{equation*}
\tau:= \sup\limits_{M >0}\left[ \inf \left\{ t \geq 0 \; | \; \|u^1\|_{\infty,t} + \|u^2\|_{\infty,t} \geq M \right\} \right],
\end{equation*}
with $\tau >0$ almost surely.
\end{thm}
The following notation for the Dirichlet heat kernel will be used throughout the rest of the paper.

\begin{defn}\label{DirichletHK}
We define $H(t,x,y)$ to be the Dirichlet heat kernel on $[0,1]$, so that 
\begin{equation}
H(t,x,y):= \frac{1}{\sqrt{4 \pi t}}\sum\limits_{n =- \infty}^{\infty} \left[ \exp\left( - \frac{(x-y+2n)^2}{4t} \right) - \exp\left(- \frac{(x+y+2n)^2}{4t} \right) \right].
\end{equation}
\end{defn}

We will prove that we have global existence to the problem where the moving boundary term is truncated. Before doing so, we present the following result which will be applied in the argument.

\begin{prop}\label{finite white noise bound}
Fix $T>0$ and let $v\in L^p(\Omega; L^{\infty}([0,T]\times [0,1]))$, where $p>10$. Define, for $t \in [0,T]$ and $x \in [0,1]$,
\begin{equation*}
w(t,x):= \int_0^t \int_0^1 H(t-s,x,y) v(s,y) \textrm{W}(\textrm{d}y, \textrm{d}s).
\end{equation*}
Then we have that $w$ is continuous, and for $t \in [0,T]$,
\begin{equation*}
\mathbb{E}\left[ \|w(t,x)\|_{\infty,t}^p \right] \leq C_{p,T} \mathbb{E}\left[ \int_0^t  \|v(s,x)\|_{\infty,s}^p \textrm{d}s \right]
\end{equation*}
\end{prop}
\begin{proof}
Let $t \in [0,T]$. We have that, for $\tau, s \in [0,t]$ and $x, y \in [0,1]$
\begin{equation}\begin{split}
\mathbb{E}& \left[ |w(\tau,x)- w(s,y)|^p \right]\\  \leq & C_p  \mathbb{E}\left[ \left|\int_s^{\tau} \int_0^1 H(\tau-r,x,z) v(r,z) \textrm{W}(\textrm{d}z,\textrm{d}r) \right|^p \right] \\ & + C_p  \mathbb{E}\left[ \left| \int_0^s \int_0^1 \left[H(\tau-r,x,z)-H(\tau-r,y,z) \right] v(r,z) \textrm{W}(\textrm{d}z,\textrm{d}r) \right|^p \right] \\& + C_p \mathbb{E}\left[ \left| \int_0^s \int_0^1 \left[ H(\tau-r,y,z)- H(s-r,y,z) \right] v(r,z) \textrm{W}(\textrm{d}z,\textrm{d}r) \right|^p \right].
\end{split}
\end{equation}
For the first term, we have by the Burkholder's inequality that it is at most 
\begin{equation}
C_p \mathbb{E}\left[\left|\int_s^{\tau} \int_0^1 H(\tau-r,x,z)^2 v(r,z)^2 \textrm{d}z\textrm{d}r \right|^{p/2} \right] \leq \mathbb{E}\left[\left|\int_s^{\tau} \left( \int_0^1 H(\tau-r,x,z)^2 \textrm{d}z \right)\|v\|_{\infty,r}^2 \; \textrm{d}r \right|^{p/2} \right]. 
\end{equation}
An application of H\"{o}lder's inequality then gives that this is at most 
\begin{equation}
\left( \int_s^{\tau} \left[ \int_0^1 H(\tau-r,x,z)^2\textrm{d}r \right]^{p/(p-2)} \textrm{d}z \right)^{(p-2)/2} \times \int_s^{\tau} \mathbb{E}\left[\|v\|_{\infty,r}^p \right] \textrm{d}r.
\end{equation}
By the estimate (2) of Proposition \ref{HeatKernelEstimatesCompact}, we have that this is at most
\begin{equation}
C_{p,T} |\tau-s|^{(p-4)/4}\int_s^{\tau} \mathbb{E}\left[\|v\|_{\infty,r}^p \right] \textrm{d}r.
\end{equation}
By arguing similarly and making use of estimates (1) and (3) from Proposition \ref{HeatKernelEstimatesCompact}, we obtain that for $\tau , s \in [0,t]$ and $x, y \in [0,1]$,
\begin{equation}\begin{split}
\mathbb{E} \left[ |w(\tau,x)- w(s,y)|^p \right] \leq  & C_{p,T} \left( |\tau-s|^{1/2} + |x-y| \right)^{(p-4)/2} \times \int_0^t \mathbb{E}\left[\|v\|_{\infty,r}^p \right] \textrm{d}r \\ = & C_{p,T} \left( |\tau-s|^{1/2} + |x-y| \right)^{3+ \frac{p-10}{2}} \times \int_0^t \mathbb{E}\left[\|v\|_{\infty,r}^p \right] \textrm{d}r .
\end{split}
\end{equation}
The result then follows by Corollary A.3 in \cite{Dalang}. 
\end{proof}

\begin{prop}\label{Truncated problem Compact}
Let $(\Omega, \mathscr{F}, \mathscr{F}_t,\mathbb{P})$ be a complete filtered probability space.  Let $\dot{W}$ be  a space time white noise on this space which respects the filtration $\mathscr{F}_t$. Let $f_i, \sigma_i$ and $h$ satisfy the conditions (i)-(iv) and suppose that $v_0^1$, $v_0^2 \in C_0((0,1))^+$. Define for $M \geq 0$ the function $h_M  :  C_0((0,1)) \rightarrow \mathbb{R}$ such that 
$$h_M(v_1,v_2):= h(v_1 \wedge M, v_2 \wedge M).$$
Then for every $M \geq 0$, there exists a unique pair of $C_0((0,1))$-valued processes $v^1,v^2$ together with $\eta^1, \eta^2$ such that

\begin{enumerate}

\item $(v^1,\eta^1)$ is a global solution to the reflected SPDE 

\begin{equation}\label{truncatedc b}
\frac{\partial v^1}{\partial t}= \Delta v^1 - h_M(v^1,v^2) \frac{\partial}{\partial x}(v^1 \wedge M) + f_1(x, v^1) + \sigma_1(x,v^1) \dot{W} + \eta^1
\end{equation}
with Dirichlet boundary conditions $v^1(t,0)=v^1(t,1)=0$ and initial data $v_0^1.$

\item $(v^2,\eta^2)$ is a global solution to the reflected SPDE 

\begin{equation}\label{truncatedc a}
\frac{\partial v^2}{\partial t}= \Delta v^2 + h_M(v^1,v^2) \frac{\partial}{\partial x}(v^2 \wedge M) + f_2(x, v^2) + \sigma_2(x,v^2) \dot{W}^- + \eta^2
\end{equation}
with Dirichlet boundary conditions $v^2(t,0)=v^2(t,1)=0$ and initial data $v_0^2.$
\end{enumerate}
We then call $(v^1, \eta^1, v^2, \eta^2)$ the solution to the $M$-truncated problem in the relative frame.
\end{prop}

\begin{proof}
Note that by a concatenation argument, it is sufficient to prove existence and uniqueness on the time interval $[0,T]$ for any $T>0$. Fix $T>0$.  We perform a Picard iteration in order to obtain existence. The first approximations are given by $v^1(t,x)=v^1_0(x)$ and $v^2(t,x)=v^2_0(x)$ for all time. For $n \geq 1$, we let $w_{n+1}^1$ solve the SPDE
\begin{equation}
\frac{\partial w_{n+1}^1}{\partial t} = \Delta w_{n+1}^1 - h_M(v^1_n,v^2_n) \frac{\partial}{\partial x}(v^1_n \wedge M) + f_1(x, v^1_n) + \sigma_1(x,v^1_n) \dot{W}
\end{equation}
with Dirichlet boundary conditions $w_{n+1}^1(t,0)=w_{n+1}^1(t,1)=0$ and initial data $v_0^1$. We then set $v^1_{n+1}:= w_{n+1}^1+ z_{n+1}^1$, where $z_{n+1}^1$ solves the obstacle problem with obstacle $-w_{n+1}^1$. We similarly define $v^2_{n+1}$ in terms of $v^1_n$ and $v^2_n$, via functions $w_{n+1}^2$ and $z_{n+1}^2$. Writing the equation for $w_{n+1}^1$  in mild form gives the expression 
\begin{equation}\label{w}\begin{split}
w_{n+1}^1(t,x)= & \int_0^1 H(t,x,y)v_0^1(y) \textrm{d}y \\ & - \int_0^t \int_0^1 \frac{\partial H}{\partial y}(t-s,x,y) h_M(v^1_n(s,\cdot), v^2_n(s,\cdot)) \left[v^1_n(s,y) \wedge M \right] \textrm{d}y \textrm{d}s \\ & + \int_0^t \int_0^1 H(t-s,x,y)f_1(y,v^1_n(s,y)) \textrm{d}x\textrm{d}s \\ & + \int_0^t \int_0^1 H(t-s,x,y) \sigma_1(y,v^1_n(s,y)) \textrm{W}(\textrm{d}y,\textrm{d}s),
\end{split}
\end{equation}
where $H$ is the Dirichlet heat kernel as in Definition \ref{DirichletHK}. Recall that by Theorem \ref{Obstacle[0,1]}, we have that $\|v^1_{n+1} - v^1_n \|_{\infty,t} \leq 2\|w_{n+1}^1-w_n^1 \|_{\infty,t}$ almost surely. Therefore,
\begin{equation*}\begin{split}
\mathbb{E} & \left[ \|v^1_{n+1}-v^1_n \|_{\infty, t}^p \right] \leq 2^p \mathbb{E}\left[ \|w_{n+1}^1-w_n^1 \|_{\infty, t}^p \right] \\ & \leq \mathbb{E}\left[ \sup\limits_{s \in [0,t]}\sup\limits_{x \in [0,1]} \left| \int_0^s \int_0^1 H(s-r,x,y)\left[ f_1(y,v^1_n(r,y))-f_1(y,v^1_{n-1}(r,y) \right] \textrm{d}y \textrm{d}r \right|^p \right] \\ & + \mathbb{E}\left[\sup\limits_{s \in [0,t]}\sup\limits_{x \in [0,1]} \left| \int_0^s \int_0^1 H(s-r,x,y) \left[ \sigma_1(y,v^1_n(r,y))- \sigma_1(y,v^1_{n-1}(r,y)) \right] \textrm{W(dy,dr)} \right|^p \right] \\ & + \mathbb{E}\left[\sup\limits_{s \in [0,t]}\sup\limits_{x \in [0,1]} \left| \int_0^s \int_0^1 \frac{\partial H}{\partial y}(s-r,x,y)\left[ h_M(v^1_n(r,\cdot),v_n^2(r,\cdot))(v^1_n(r,y) \wedge M ) \right. \right. \right. \\ & \; \; \; \; \; \; \; \; \; \; \left. \left. \left. - h_M(v^1_{n-1}(r,\cdot),v_{n-1}^2(r,\cdot))(v^1_{n-1}(r,y) \wedge M )  \right] \textrm{d}y \textrm{d}r \right|^p \right] 
\end{split}
\end{equation*}
We deal with the three terms separately. For the first, we apply H\"{o}lder's inequality to see that it is at most  
\begin{equation}\begin{split}
\mathbb{E} & \left[ \sup\limits_{s \in [0,t]}\sup\limits_{x \in [0,1]} \left| \int_0^s \int_0^1 H(s-r,x,y)\left[ f_1(y,v^1_n(r,y))-f_1(y,v^1_{n-1}(r,y) \right]^p \textrm{d}y \textrm{d}r \right| \right. \\ & \times \left. \left|\int_0^s \int_0^1 H(s-r,x,y) \textrm{d}y \textrm{d}r \right|^{p/q} \right] \\ & \leq  C_T \mathbb{E}\left[ \sup\limits_{s \in [0,t]}\sup\limits_{x \in [0,1]} \left| \int_0^s \int_0^1 H(s-r,x,y)\left[ f_1(y,v^1_n(r,y))-f_1(y,v^1_{n-1}(r,y) \right]^p \textrm{d}y \textrm{d}r \right| \right]. 
\end{split}
\end{equation}
Making use of the Lipschitz property of the function $f_1$, this is at most
\begin{equation}\begin{split}
C_T & \mathbb{E}\left[ \sup\limits_{s \in [0,t]}\sup\limits_{x \in [0,1]} \left| \int_0^s \int_0^1 H(s-r,x,y)\left[ v^1_n(r,y)-v^1_{n-1}(r,y) \right]^p \textrm{d}y \textrm{d}r \right| \right] \\ \leq & C_T \mathbb{E}\left[ \sup\limits_{s \in [0,t]}\sup\limits_{x \in [0,1]} \left| \int_0^s \left( \int_0^1 H(s-r,x,y) \textrm{d}y \right) \| v^1_n-v^1_{n-1} \|_{\infty,r}^p \textrm{d}s \right| \right] \\ \leq & C_T \mathbb{E}\left[ \sup\limits_{s \in [0,t]}\sup\limits_{x \in [0,1]} \left| \int_0^s \| v^1_n-v^1_{n-1} \|_{\infty,r}^p \textrm{d}r \right| \right].  
\end{split}
\end{equation}
For the second term, we apply Proposition \ref{finite white noise bound} and the Lipschitz property of $\sigma_1$ to deduce that, for $p>10$ and $t \in [0,T]$,
\begin{multline*}
\mathbb{E}\left[\sup\limits_{s \in [0,t]}\sup\limits_{x \in [0,1]} \left| \int_0^s \int_0^1 H(s-r,x,y) \left[ \sigma_1(y,v^1_n(r,y))- \sigma_1(y,v^1_{n-1}(r,y)) \right] \textrm{W(dy,dr)} \right|^p \right] \\ \leq C_{p,T}\int_0^t \mathbb{E} \left[ \| v^1_n -v^1_{n-1} \|_{\infty,r}^p \right] \textrm{d}r.
\end{multline*}
Finally, we deal with the third term. Using the Lipschitz property of $h$, we see that the third term is at most
\begin{equation}\begin{split}
C_M & \mathbb{E}\left[ \sup\limits_{s \in [0,t]}\sup\limits_{x \in [0,1]} \left| \int_0^s \int_0^1 \left| \frac{\partial H}{\partial y}(s-r,x,y)\right|\left[  \| v^1_n - v^1_{n-1} \|_{\infty,r} + \|v^2_n - v^2_{n-1} \|_{\infty,r} \right] \textrm{d}y \textrm{d}r \right|^p \right] \\ = & C_M  \mathbb{E}\left[ \sup\limits_{s \in [0,t]}\sup\limits_{x \in [0,1]} \left| \int_0^s \left( \int_0^1 \left| \frac{\partial H}{\partial y}(s-r,x,y)\right| \textrm{d}y \right)\left[  \| v^1_n - v^1_{n-1} \|_{\infty,r} + \|v^2_n - v^2_{n-1} \|_{\infty,r} \right] \textrm{d}r \right|^p \right].
\end{split}
\end{equation}
Applying H\"{o}lder's inequality and Proposition \ref{Compact Heat Kernel Derivative} then gives that this is at most
\begin{equation}
\begin{split}
C_M & \mathbb{E}\left[ \left( \sup\limits_{x \in [0,1]}\int_0^t \left[ \int_0^1 \left| \frac{\partial H}{\partial y}(s-r,x,y)\right| \textrm{d}y \right]^q \textrm{d}r \right)^{p/q} \times \int_0^t \left[\| v^1_n - v^1_{n-1} \|_{\infty,r} +\| v^2_n - v^2_{n-1} \|_{\infty,r}  \right]^p \textrm{d}r \right] \\ \leq & C_M \times \mathbb{E} \left[ \left( \int_0^t \frac{1}{(t-r)^{q/2}} \textrm{d}r \right)^{p/q} \times \int_0^t \left[\| v^1_n - v^1_{n-1} \|_{\infty,r} +\| v^2_n - v^2_{n-1} \|_{\infty,r}  \right]^p \textrm{d}r \right] \\ = & C_M \times \mathbb{E} \left[\left( \int_0^t \frac{1}{r^{q/2}} \textrm{d}r \right)^{p/q} \times \int_0^t \left[\| v^1_n - v^1_{n-1} \|_{\infty,r} +\| v^2_n - v^2_{n-1} \|_{\infty,r}  \right]^p \textrm{d}r \right].
\end{split} 
\end{equation}
For $p>10$ and corresponding $q \in (1, \frac{10}{9})$, this is at most 
\begin{equation}
C_{M,p,T} \times \int_0^t \mathbb{E}[\left[\| v^1_n - v^1_{n-1} \|_{\infty,r}+\| v^2_n - v^2_{n-1} \|_{\infty,r} \right]^p] \textrm{d}r. 
\end{equation}
Putting this all together, we have shown that for any $t \in [0,T]$,
\begin{equation*}
\mathbb{E}\left[ \|v_{n+1}^1-v_n^1\|_{\infty,t}^p \right] \leq C_{M,p,T} \int_0^t \mathbb{E} \left[ \|v_n^1- v_{n-1}^1 \|_{\infty,s}^p + \|v_n^2-v_{n-1}^2\|_{\infty,s}^p \right] \textrm{d}s.
\end{equation*}
We can repeat these arguments to obtain similar bounds for $v^2$. Together, this gives
$$\mathbb{E}\left[\|v_{n+1}^1-v_n^1\|_{\infty,T}^p+ \|v_{n+1}^2-v_n^2\|_{\infty,T}^p \right] \leq C_{M,p,T} \int_0^T \mathbb{E}\left[\|v_{n+1}^1-v_n^1\|_{\infty,s}^p+ \|v_{n+1}^2-v_n^2\|_{\infty,s}^p \right] \textrm{d}s.$$ 
We can then argue that 
\begin{equation}\begin{split}
\mathbb{E}& \left[\|v_{n+1}^1-v_n^1\|_{\infty,T}^p+ \|v_{n+1}^2-v_n^2\|_{\infty,T}^p \right] \\ \leq & C_{M,p,T} \int_0^T \mathbb{E}\left[\|v_{n+1}^1-v_n^1\|_{\infty,s}^p+ \|v_{n+1}^2-v_n^2\|_{\infty,s}^p \right] \textrm{d}s \\ \leq & C_{M,p,T}^2 \int_0^T \int_0^s \mathbb{E}\left[\|v_{n-1}^1-v_{n-2}^1\|_{\infty,u}^p+ \|v_{n-1}^2-v_{n-2}^2\|_{\infty,u}^p \right] \textrm{d}u \;  \textrm{d}s \\ = & C_{M,p,T}^2 \int_0^T \int_u^T \mathbb{E}\left[\|v_{n-1}^1-v_{n-2}^1\|_{\infty,u}^p+ \|v_{n-1}^2-v_{n-2}^2\|_{\infty,u}^p \right] \textrm{d}s \;  \textrm{d}u \\ = & C_{M,p,T}^2 \int_0^T  \mathbb{E}\left[\|v_{n-1}^1-v_{n-2}^1\|_{\infty,u}^p+ \|v_{n-1}^2-v_{n-2}^2\|_{\infty,u}^p \right](T-u) \;  \textrm{d}u.
\end{split}
\end{equation}
Iterating this, we obtain
\begin{equation}\begin{split}
\mathbb{E}& \left[\|v_{n+1}^1-v_n^1\|_{\infty,T}^p+ \|v_{n+1}^2-v_n^2\|_{\infty,T}^p \right] \\  & \leq C_{M,p,T}^n \int_0^T \mathbb{E}\left[ \|v_1^1-v_0^1\|_{\infty,s}^p + \|v_1^2-v_0^2\|_{\infty,s}^p \right] \frac{(T-s)^{n-1}}{(n-1)!} \textrm{d}s \\ & \leq C_{M,p,T}^n \times \mathbb{E} \left[ \|v_1^1-v_0^1\|_{\infty,T}^p + \|v_1^2 - v_0^2\|_{\infty,T}^p \right] \frac{T^{n}}{n!}.
\end{split}
\end{equation}
Therefore, for $m > n \geq 1$ we have that 
\begin{equation}
\mathbb{E}\left[\|v_{n+1}^1-v_n^1\|_{\infty,T}^p+ \|v_{n+1}^2-v_n^2\|_{\infty,T}^p \right] \leq \sum\limits_{k=n}^{m-1} \left[ \frac{\tilde{C}_{M,p,T}^kT^k}{k!} \right] \mathbb{E}\left[ \|v_1^1-v_0^1\|_{\infty,T}^p + \|v_1^2-v_0^2\|_{\infty,T}^p \right] \rightarrow 0.
\end{equation}
as $m,n \rightarrow \infty$. Hence, the sequence $(v_n^1,v_n^2)$ is Cauchy in the space $L^p(\Omega ;C([0,T] \times [0,1]))^2$ and so converges to some pair $(v^1,v^2)$. We now verify that this is indeed a solution to our evolution equation. Let $\tilde{w}^1$ be given by
\begin{equation}\begin{split}
\tilde{w}^1(t,x)= &\int_0^1 H(t,x,y)v_0^1(y) \textrm{d}y \\ & - \int_0^t \int_0^1 \frac{\partial H}{\partial y}(t-s,x,y) h_M(v^1(s,\cdot), v^2(s,\cdot)) \left[v^1(s,y) \wedge M \right] \textrm{d}y \textrm{d}s \\ & + \int_0^t \int_0^1 H(t-s,x,y)f_1(y,v^1(s,y)) \textrm{d}x\textrm{d}s + \int_0^t \int_0^1 H(t-s,x,y) \sigma_1(y,v^1(s,y)) \textrm{W}(\textrm{d}y,\textrm{d}s).
\end{split}
\end{equation}
Define $\tilde{v}^1= \tilde{w}^1 + \tilde{z}^1$, where $\tilde{z}^1$, together with a measure $\tilde{\eta}^1$, solves our obstacle problem with obstacle $-\tilde{w}^1$. Then, by arguing as before, we see that 
\begin{equation}\begin{split}
\mathbb{E} \left[ \|\tilde{v}^1- v_n^1\|_{\infty,T}^p \right] \leq & C_{M,p,T} \int_0^T \mathbb{E} \left[ \|v^1- v_{n-1}^1\|_{\infty,t}^p+ \|v^2- v_{n-1}^2\|_{\infty,t}^p \right] \textrm{d}t \\ \leq & \tilde{C}_{M,p,T} \mathbb{E} \left[ \|v^1-v^1_{n-1}\|_{\infty,T}^p + \|v^2- v_{n-1}^2\|_{\infty,T}^p\right] \rightarrow 0.
\end{split}
\end{equation} 
It follows that $\tilde{v}^1= v^1$ in $L^p(\Omega; C([0,T] \times [0,1]))$. The same applies to $v^2$, so it follows that the pair $(v^1,v^2)$, together with the reflection measures $(\tilde{\eta}^1, \tilde{\eta}^2)$, do indeed satisfy our problem. 

Uniqueness follows by essentially the same argument. Given two solutions with the same initial data, $(v^1_1,v^2_1)$ and $(v_2^1,v_2^2)$ (together with their reflection measures), we argue as before to obtain that, for $t \in [0,T]$, 
\begin{equation}
\mathbb{E}\left[\|v^1_1-v_2^1\|_{\infty,t}^p+ \|v^2_1-v^2_2\|_{\infty,t}^p \right] \leq \int_0^t \mathbb{E}\left[\|v^1_1-v_2^1\|_{\infty,s}^p+ \|v^2_1-v^2_2\|_{\infty,s}^p \right] \textrm{d}s.
\end{equation}
The equivalence then follows by Gronwall's inequality.
\end{proof}

We are now in position to prove Theorem \ref{Main}. This essentially amounts to showing that the solutions to our truncated problems coincide for different $M$. We use this to  define a candidate function, and then check the conditions for this candidate.  

\begin{proof}[Proof of Theorem \ref{Main}]
For every $M >0$, let $(v^1_M, \eta_M^1, v^2_M,\eta_M^2)$ be the solution to the $M$-truncated problem. Suppose $M_1 \leq M_2$. Then we clearly have that $(v^1_{M_2}, \eta_{M_2}^1, v_{M_2}^2, \eta_{M_2}^2)$ solves the $M_1$-truncated problem in the relative frame, until the stopping time
\begin{equation}
\tilde{\tau}= \inf \left\{ t \geq 0 \;  | \; \|(v^1_{M_2}, v^2_{M_2})\|_{\infty, t} \geq M_1 \right\}.
\end{equation}
We can then argue as in the proof of Proposition \ref{Truncated problem Compact} to deduce that 
\begin{equation}
\mathbb{E} \left[ \|(w^1_{M_1}, w^2_{M_1})-(w_{M_2}^1, w_{M_2}^2)\|_{\infty,\tilde{\tau}}^p \right] =0,
\end{equation}
where, for $i=1,2$, $(w_{M_i}^1,w^2_{M_i})$ are the solutions to the unreflected equations which correspond to $(v^1_{M_1}, \eta_{M_1}^1, v^2_{M_1},\eta_{M_1}^2)$ and $(v^1_{M_2}, \eta_{M_2}^1, v^2_{M_2},\eta_{M_2}^2)$ respectively, as in (\ref{w}). Therefore, by uniqueness of solutions to the obstacle problem, $(v^1_{M_1}, \eta_{M_1}^1, v^2_{M_1},\eta_{M_1}^2)$ and $(v^1_{M_2}, \eta_{M_2}^1, v^2_{M_2},\eta_{M_2}^2)$ agree until the random time $\tilde{\tau}$. This consistency allows us to define our candidate solution in the relative frame, $(v^1,\eta^1, v^2,\eta^2)$, by setting $(v^1,\eta^1, v^2,\eta^2)=(v_M^1,\eta_M^1,v_M^2,\eta_M^2)$ on $[0,\tau_M]$, where 
\begin{equation}\label{localising}
\tau_M = \inf \left\{ t \geq 0 \; | \; \|(v_M^1,v_M^2)\|_{\infty,t} \geq M \right\}. 
\end{equation}
This defines $(v^1,\eta^1, v^2,\eta^2)$ on the interval $[0, \tau)$, where $\tau= \sup\limits_{M >0} \tau_M.$ As a convention, we set $v^1(t,x)=v^2(t,x)=\infty$ for $x \in [0,1]$ and $t \geq \tau$. It is clear from the definition that in fact
\begin{equation}
\tau= \sup\limits_{M >0} \left[ \inf\left\{ t \geq 0 \; | \; \|(v^1,v^2)\|_{\infty,t} \geq M \right\}\right].
\end{equation}
Clearly, $(v^1,\eta^1, v^2,\eta^2)$ is then a maximal solution to the moving boundary problem in the relative frame until the explosion time $\tau$, with localising sequence $\tau_M$. In addition $\tau >0$ almost surely as, by construction, $v^i \in C([0,\tau) \times [0,1])$ for $i=1,2$ almost surely. We now prove uniqueness. If $(v_1^1, \eta^1_1,v_1^2,\eta^2_1)$ and $(v_2^1,\eta_2^1,v_2^2,\eta_2^2)$ are both maximal solutions, they both satisfy the $M$-truncated problem until they exceed $M$ in the infinity norm, and so we can once again argue as in Proposition \ref{Truncated problem Compact} to obtain that they both agree with the unique solution of the $M$-truncated problem until these times. Since this holds for every $M$, it follows that they agree until a common explosion time. We therefore have the result.
\end{proof}

\begin{prop}\label{global existence}
Suppose that $h$ is a bounded function. Then the solution to the moving boundary problem is global.
\end{prop}

\begin{proof}
Fix $T>0$. Let $(v^1,\eta^1, v^2,\eta^2)$ be the unique maximal solution to the moving boundary problem in the relative frame, and let $\tau$ be the blow-up time for this solution. We consider the solutions to the corresponding truncated solutions, $(v_M^1,v_M^2)$, for $M>0$, with the same initial data i.e. $(v_M^1(0,x),v_M^2(0,x))=(v^1_0(x),v^2_0(x))$.  Let $w^1_M$ solve the SPDE
\begin{equation}
\frac{\partial w^1}{\partial t} = \Delta w^1 - h_M(v^1,v^2) \frac{\partial}{\partial x}(v^1 \wedge M) + f_1(x, v^1) + \sigma_1(x,v^1) \dot{W}
\end{equation}
with Dirichlet boundary conditions $w^1_M(t,0)=w^1_M(t,1)=0$ and initial data $w^1_M(0,x)=v^1_0(x)$. Define $w_M^2$ similarly. Noting that $\|v^1_M\|_{\infty,T} \leq 2\|w_M^1\|_{\infty,T}$ and making use of the mild form for $w_M^1$, we have that
 \begin{equation*}\begin{split}
\mathbb{E}\left[ \|v^1_M\|_{\infty, T}^p \right] & \leq  C_p \sup\limits_{t \in [0,T]} \sup\limits_{x \in [0,1]} \left|\int_0^1 H(t,x,y)v^1_0(y) \textrm{d}y \right|^p \\ & +  C_p\mathbb{E}\left[ \sup\limits_{t \in [0,T]}\sup\limits_{x \in [0,1]} \left| \int_0^t \int_0^1 H(t-s,x,y)f_1(y,v^1_M(s,y)) \textrm{d}y \textrm{d}s \right|^p \right] \\ & + C_p\mathbb{E}\left[\sup\limits_{t \in [0,T]}\sup\limits_{x \in [0,1]} \left| \int_0^t \int_0^1 H(t-s,x,y)\sigma_1(y,v^1_M(s,y))\textrm{W}(\textrm{d}y, \textrm{d}s) \right|^p \right] \\ & + C_p\|h\|_{\infty}^p\mathbb{E}\left[ \sup\limits_{t \in [0,T]}\sup\limits_{x \in [0,1]} \left( \int_0^t \int_0^1 \left| \frac{\partial H}{\partial y}(t-s,x,y) \right| \|v^1_M\|_{\infty,s} \textrm{d}y \textrm{d}s \right)^p \right] 
\end{split}
\end{equation*}
By arguing as in Theorem \ref{Main}, we obtain that for $t \in [0,T]$
\begin{equation*}
\mathbb{E}\left[ \|v^1_M\|_{\infty, t}^p \right] \leq C_{p,T, \|h\|_{\infty}}\left(\|v^1_0\|_{\infty}+ \int_0^t \mathbb{E}\left[ \|v^1_M\|_{\infty, s}^p \right]\textrm{d}s \right).
\end{equation*}
By noting that $C_{p,T,\|h\|_{\infty}}$ and $\|v_0^1\|$ do not depend on $M$ here, we can apply Gronwall's Lemma to obtain that $$\sup\limits_{M>0} \mathbb{E}[\|v^1_M\|_{\infty,T}^p] < \infty.$$ It follows that
\begin{equation*}
\mathbb{E}[\|v^1\|_{\infty, \tau \wedge T}^p]\leq \mathbb{E}\left[\lim\limits_{M \rightarrow \infty} \|v^1_M\|_{\infty, T}^p \right] \leq \liminf\limits_{M \rightarrow \infty}\mathbb{E} \left[ \|v^1_M\|_{\infty, T}^p \right] < \infty.
\end{equation*}
Similarly,$$\mathbb{E}\left[ \|v^2\|_{\infty, \tau \wedge T}^p \right] < \infty.$$ This can only hold if there is almost surely no blow-up before time $T$ i.e. $\tau >T$ almost surely. Since this holds for every $T>0$, we must have that $\tau = \infty$ almost surely. We then also have that 
\begin{equation}
\mathbb{E}\left[ \|v^i\|_{\infty, t}^p \right] < \infty
\end{equation}
for $i=1,2$ and every $t\geq 0$. This allows us to take limits in the localising sequence, so we can obtain that the solution is indeed global.
\end{proof}

\subsection{H\"{o}lder Continuity of the Solutions}

We now prove that, as in the case of the static reflected SPDE, our equations enjoy the expected H\"{o}lder continuity- up to $1/4$-H\"{o}lder in time and up to $1/2$-H\"{o}lder in space. The details of the proof here are a simplification of those used in \cite{Dalang2}, where H\"{o}lder continuity is proved for the equations when there is no moving boundary term. 

The following result is Lemmas 3.1 in \cite{Dalang2}.

\begin{lem}\label{PDE BOund}
Let $V \in C^{1,2}_b([0,T] \times [0,1])$ and $\psi,F \in C([0,T] \times [0,1])$ with $\psi \leq 0$. Suppose that 
\begin{equation}
\frac{\partial V}{\partial t}= \frac{1}{2}V^{\prime \prime} + \psi V + \psi F
\end{equation}
with Dirichlet or Neumann boundary conditions at zero, and zero initial data. Then
\begin{equation}
\|V\|_{T,\infty} \leq \|F\|_{T,\infty}.
\end{equation}
\end{lem}

We now present a slight adaptation of Lemma 3.2 in \cite{Dalang2}.

\begin{lem}\label{smoothing}
Suppose that $D= (0,1)$ or $D= (0,\infty)$. Let $f: [0,T] \times D \rightarrow \mathbb{R}$ be such that $f \equiv 0$ on $\partial D$ and for every $t,s \in [0,T]$ and every $x,y \in \bar{D}$
\begin{equation}
|f(t,x)-f(s,y)| \leq K(|t-s|^{\alpha} + |x-y|^{\beta}).
\end{equation}
Then there exists a smooth function $f_{p,q}:[0,T] \times D \rightarrow \mathbb{R}$ such that $f_{p,q} \equiv 0$ on $\partial D$ and
\begin{enumerate}[(i)]
\item $\|f_{p,q}\|_{\infty} \leq \|f\|_{\infty}$.
\item $\|f_{p,q}-f\|_{\infty} \leq C_{\alpha,\beta}K(p^{\alpha}+q^{\beta})$.
\item $\left\| \frac{\partial f_{p,q}}{\partial t} \right\|_{\infty} \leq C_{\alpha, \beta}K p^{\alpha -1}$.
\item $\left\| \frac{\partial f_{p,q}}{\partial x} \right\|_{\infty} \leq C_{\alpha,\beta}K q^{\beta -1}.$ 
\end{enumerate}
\end{lem}
\begin{proof}
The proof is as in \cite{Dalang2}, replacing the use of the heat kernel on $\mathbb{R}$ to smooth $f$ with the Dirichlet heat kernel on $D$.
\end{proof}

We now present the result regarding the H\"{o}lder continuity of our solutions. In addition to allowing for the extra term in the equation, corresponding to the moving boundary, our proof here slightly differs from the approach used in \cite{Dalang2} in another way. In \cite{Dalang2}, the solution to the obstacle problem 
\begin{equation}\label{staticobstacle}
\frac{\partial u}{\partial t}= \Delta u + f(x,u) + \sigma(x,u) \dot{W} + \eta
\end{equation}
is approximated by the solutions to the solutions of the penalised SPDEs

\begin{equation}\label{staticpenalise}
\frac{\partial u_{\epsilon}}{\partial t}= \Delta u_{\epsilon} + f(x,u_{\epsilon}) + \sigma(x,u_{\epsilon}) \dot{W} + g_{\epsilon}(u_{\epsilon}),
\end{equation}
where $g_{\epsilon}(x)= \frac{1}{\epsilon} \arctan( [x \wedge 0]^2)$. H\"{o}lder continuity of the solution to (\ref{staticobstacle}) is then shown by uniformly controlling the H\"{o}lder continuity of the equations (\ref{staticpenalise}). Here, we instead approximate $u$ by the solutions to the equations 

\begin{equation}\label{staticpenalise2}
\frac{\partial u_{\epsilon}}{\partial t}= \Delta u_{\epsilon} + f(x,u) + \sigma(x,u) \dot{W} + g_{\epsilon}(u_{\epsilon}).
\end{equation}
By using $u$ in the coefficients of our approximating SPDEs $f$ and $\sigma$ here, we limit the problem of uniformly controlling the H\"{o}lder coefficients to studying the deterministic obstacle problem.

\begin{thm}\label{Holder Continuity}
For $i=1,2$, let $u_0^i$ be such that $u_0^i \circ (\Theta_{p_0}^i)^{-1} \in C_0((0,1))^+ \cap C^{\gamma/2}([0,1])$ for every $\gamma \in (0,1)$. Then, for every $T>0$, $M>0$ and every $\gamma \in (0,1)$ the solution $(v^1_M,v^2_M)$ to the $M$-truncated problem in the relative frame with initial data $(u_0^1 \circ (\Theta_{p_0}^1)^{-1}, u_0^2 \circ (\Theta_{p_0}^2)^{-1})$, described by equations (\ref{truncatedc b}) and (\ref{truncatedc a}), is $\gamma/4$-H\"{o}lder in time and $\gamma/2$-H\"{o}lder in space on $[0,T] \times [0,1]$. In particular, if $(u^1,\eta^1,u^2,\eta^2,p)$ is the solution to our moving boundary problem with initial data $(u_0^1, u_0^2, p_0)$, then $(u^1,u^2)$ enjoys the same H\"{o}lder regularity locally until the blow-up time, $\tau$.
\end{thm} 
\begin{proof}
We consider $v^1_M$ only, since the argument for $v^2_M$ is identical. Define $w^1_M$ to be the $C_0((0,1))$-valued process given by
\begin{equation}\begin{split}
w^1_M(t,x)= &\int_0^1 H(t,x,y)u_0^1((\Theta_{p_0}^1)^{-1}(y)) \textrm{d}y \\ &  - \int_0^t \int_0^1 \frac{\partial H}{\partial y}(t-s,x,y)h_M(v^1_M(s,.), v^2_M(s,.))(v^1_M(s,y) \wedge M) \textrm{d}y \textrm{d}s \\ & + \int_0^t \int_0^1 H(t-s,x,y)f_1(s,v^1_M(s,y)) \textrm{d}y \textrm{d}s \\ &  + \int_0^t \int_0^1 H(t-s,x,y)\sigma_1(s,v^1_M(s,y)) \textrm{W}(\textrm{d}y, \textrm{d}s).
\end{split}
\end{equation}
Let $r=2\gamma+12$, so that $\gamma=r/2 -6$. By applying the inequalities from Propositions \ref{HeatKernelEstimatesCompact} and \ref{HeatDerivativeBoundsCompact} together with Burkholder's inequality, we see that 
\begin{equation}
\mathbb{E}\left[|w^1_M(t,x)-w^1_M(s,y)|^r \right] \leq C(|t-s|^{1/2}+|x-y| )^{\frac{r}{2}-2}.
\end{equation}
We note that it is Proposition \ref{HeatDerivativeBoundsCompact} which allows us to control the extra term arising due to the moving boundary. It then follows by Corollary A.3 in \cite{Dalang} that there exists a random variable $X \in L^r$ such that
\begin{equation}
|v^1(t,x)-v^1(s,y)|^r \leq X(|t-s|^{1/2}+|x-y| )^{\frac{r}{2}-6}= X(|t-s|^{1/2}+|x-y| )^{\gamma}
\end{equation}
almost surely. From here, the argument follows the steps from Theorem 3.3 in \cite{Dalang2}, so we give an outline only and refer the reader to the proof of Theorem 3.3 in \cite{Dalang2} for further details. For each $\epsilon >0$, let $z^{\epsilon}$ solve the PDE 
\begin{equation}\label{obstacle}
\frac{\partial z^{\epsilon}}{\partial t}= \Delta z^{\epsilon}+ g_{\epsilon}(z^{\epsilon} + v^1)
\end{equation}
on $[0,1]$ with Dirichlet boundary conditions and zero initial data, where we once again define $$g_{\epsilon}(x):= \frac{1}{\epsilon} \arctan([x \wedge 0]^2).$$  We then have (see \cite{NP}) that $z^{\epsilon}+v$ increases to $u$, the solution of the reflected SPDE on $[0,1]$. Let $(v^1)_{p,q}$ be a smoothing of $v^1$ as in Proposition \ref{smoothing}, with respect to the random variable $X$, the H\"{o}lder coefficients $\gamma/2$ and $\gamma/4$, and the constants $p$, $q$ which are yet to be determined. Define $z^{\epsilon}_{p,q}$ to be the solution of the PDE
\begin{equation}\label{obstaclesmooth}
\frac{\partial z^{\epsilon}_{p,q}}{\partial t}= \Delta z^{\epsilon}_{p,q}+ g_{\epsilon}(z^{\epsilon}_{p,q} + (v^1)_{p,q})
\end{equation}
with Dirichlet boundary condition at zero and and zero initial data. We then have that (see the proof of Theorem 1.4 in \cite{NP} for details)
\begin{equation}
\|z^{\epsilon}\|_{T, \infty} \leq \|v^1\|_{T, \infty},
\end{equation}
and
\begin{equation}
\|z^{\epsilon}_{p,q}\|_{T, \infty} \leq \|(v^1)_{p,q}\|_{T, \infty}.
\end{equation}
Define $\alpha_{p,q}^{\epsilon} := \frac{\partial w_{p,q}^{\epsilon}}{\partial t}$ and $\beta_{p,q}^{\epsilon} := \frac{\partial w_{p,q}^{\epsilon}}{\partial x}$. By differentiating the equation (\ref{obstaclesmooth}) in time we obtain 
\begin{equation}\label{111}
\frac{\partial \alpha_{p,q}^{\epsilon}}{\partial t}= \Delta \alpha_{p,q}^{\epsilon} + g^{\prime}_{\epsilon}(z^{\epsilon,a,b} + (v^1)_{p,q})\left[\alpha_{p,q}^{\epsilon} + \frac{\partial(v^1)_{p,q}}{\partial t}  \right], 
\end{equation} 
with zero initial data and Dirichlet boundary conditions $\alpha_{p,q}^{\epsilon}(t,0)=\alpha_{p,q}^{\epsilon}(t,1)=0$. Similarly, if we differentiate (\ref{obstaclesmooth}) in space 
\begin{equation}\label{222}
\frac{\partial \beta_{p,q}^{\epsilon}}{\partial t}= \Delta \beta_{p,q}^{\epsilon} + g^{\prime}_{\epsilon}(z^{\epsilon}_{p,q} + (v^1)_{p,q})\left[\beta_{p,q}^{\epsilon} + \frac{\partial(v^1)_{p,q}}{\partial x}  \right], 
\end{equation} 
with initial data $z_0^{\prime}=0$ and Neumann boundary conditions $\frac{\partial  \beta_{p,q}^{\epsilon}}{\partial x}(t,0)=\frac{\partial  \beta_{p,q}^{\epsilon}}{\partial x}(t,1)=0$ . Applying Lemma \ref{PDE BOund} to equations (\ref{111}) and (\ref{222}) controls the infinity norms of $\alpha^{\epsilon}_{p,q}$ and $\beta^{\epsilon}_{p,q}$ by the infinity norms of $\frac{\partial (v^1)_{p,q}}{\partial t}$ and $\frac{\partial (v^1)_{p,q}}{\partial x}$ respectively, uniformly over $\epsilon$. By using the bounds from Lemma \ref{smoothing} and choosing $p=|t-s|$, $q=|x-y|$, we can then obtain uniform control the $\gamma/4$-H\"{o}lder norm in time and the $\gamma/2$-H\"{o}lder norm in space. Letting $\epsilon \downarrow 0$ then allows us to conclude.
\end{proof}

\begin{cor}\label{HolderPrice}
For every $\gamma \in (0,1)$, the derivative of the boundary is locally $\gamma/4$-H\"{o}lder continuous on $[0, \tau)$, where $\tau$ is the blow-up time.
\end{cor}
\begin{proof}
Fix $\gamma \in (0,1)$. Recall that $p^{\prime}(t)= h(v^1(t, \cdot), v^2(t,\cdot))$, where $v^1(t,x)= u^1(t,p(t)- \cdot)$ and $v^2(t,x)=u^2(t,p(t)+ \cdot)$. 
For $M > 0$, define 
\begin{equation}
\tau_M = \inf \left\{ t \geq 0 \; | \; \|(v^1,v^2)\|_{\infty,t} \geq M \right\}. 
\end{equation}
Note that $\tau_M \uparrow \tau$ as $M \rightarrow \infty$. Let $t, s \in [0,\tau_M]$.
We have, by the Lipschitz property of $h$, that
\begin{equation}\begin{split}
|p^{\prime}(t)- p^{\prime}(s)| & \leq K \left( \|v^1(t, \cdot) - v^1(s, \cdot)\|_{\infty} + \|v^2(t, \cdot) - v^2(s, \cdot)\|_{\infty} \right) \\ & =K \left( \|v^1_M(t, \cdot) - v^1_M(s, \cdot)\|_{\infty} + \|v^2_M(t, \cdot) - v^2_M(s, \cdot)\|_{\infty} \right). 
\end{split}
\end{equation}
The result then follows by Theorem \ref{Holder Continuity}.
\end{proof}

\section{The Moving Boundary Problem on Semi-Infinite Intervals in the Relative Frame}

We now consider the analogous obstacle problem, where the two sides of the equation satisfy SPDEs on the infinite halflines $(-\infty,p(t)]$ and $[p(t),\infty)$ respectively. That is 

\begin{equation*}
\frac{\partial u^1}{\partial t}= \Delta u^1 + f_1(p(t)-x, u^1(t,x))+ \sigma_1(p(t)-x, u^1(t,x))\dot{W} + \eta^1
\end{equation*}
on $[0,\infty) \times (- \infty,p(t)]$, and 
\begin{equation*}
\frac{\partial u^2}{\partial t}= \Delta u^2 + f_2(x-p(t),u^2(t,x))+ \sigma_2(x-p(t),u^2(t,x))\dot{W} + \eta^2,
\end{equation*}
on $[0,\infty) \times [p(t), \infty)$. We once again have Dirichlet conditions at the mid, $p(t)$, so that $u^1(t,p(t))=u^2(t,p(t))=0$, with the point $p(t)$ evolving according to the equation 
\begin{equation*}
p'(t)= h(u^1(t,p(t)-\cdot), u^2(t,p(t)+ \cdot )).
\end{equation*}
Here, $W$ is a space-time white noise and $h$ is a function of the two profiles of the equation on either side of the shared boundary. As before, $\eta^1$ and $\eta^2$ are reflection measures for the functions $u^1$ and $u^2$ respectively, keeping the profiles positive and satisfying the conditions 
\begin{enumerate}[(i)]
\item $\textrm{supp}(\eta^1) \subset \left\{ (t,x) \; | \;  x \in (\infty,p(t)) \right\}$,
\item $\textrm{supp}(\eta^2) \subset \left\{ (t,x) \; | \;  x \in (p(t), \infty) \right\}$,
\item $\int_0^{\infty} \int_{\mathbb{R}} u^1(t,x) \;  \eta^1(\textrm{d}t,\textrm{d}x)=0,$
\item $\int_0^{\infty} \int_{\mathbb{R}} u^2(t,x) \;  \eta^2(\textrm{d}t,\textrm{d}x)=0.$
\end{enumerate} 

\subsection{Formulation of the Problem}

We will be working in the spaces $\mathscr{C}_r$ and $\mathscr{C}_r^T$, defined in Section 2.2, throughout this section. This presents issues when handling both the non-Lipschitz term arising due to the moving boundary and the stochastic term. Truncating the boundary term requires more care, as we are now trying to control the $\mathscr{C}_r^T$-norm of the process. We are also unable to suitably control the supremum of the stochastic terms using our previous arguments, as they are not well suited to unbounded domains. For this reason, we introduce extra decay for the growth of the volatility relative to the growth of the drift term. Fixing $r \in \mathbb{R}$ we take, for $i=1,2$, $f_i$ and $\sigma_i$ to be measurable mappings from $[0,\infty) \times \mathbb{R}^+ \rightarrow \mathbb{R}$ and $h: \mathscr{C}_r \times \mathscr{C}_r \rightarrow \mathbb{R}$ to be a measurable function such that, for some $C, \delta>0$
\begin{enumerate}[(I)]
\item For every $x \in [0,\infty)$, $u,v \in \mathbb{R}$, 
\begin{equation*}
|f_i(x,u)-f_i(x,v)| \leq C |u-v|.
\end{equation*}
\item For every $x \in [0,\infty)$, $u \in \mathbb{R}$, $$|f_i(x,u)| \leq C(e^{rx}+ |u|).$$
\item For every $x \in [0,\infty)$, $u,v \in \mathbb{R}$, 
\begin{equation*}
|\sigma_i(x,u)-\sigma_i(x,v)| \leq C e^{-\delta x}|u-v|.
\end{equation*}
\item For every $x \in [0,\infty)$, $u \in \mathbb{R}$, $$|\sigma_i(x,u)| \leq Re^{-\delta x}(e^{rx}+|u|).$$
\item $h$ is bounded on bounded sets in $\mathscr{C}_r$.
\item For every $u_1,u_2,v_1,v_2 \in \mathscr{C}_r$, $$|h(u_1,v_1)-h(u_2,v_2)| \leq K( \|u_1- u_2\|_{\mathscr{C}_{r}}+ \|v_1-v_2\|_{\mathscr{C}_{r}}).$$
\end{enumerate}
Since our notion of solution here is motivated by the same ideas as in the compact case, we move straight to the definitions for solutions to non-linear SPDEs and moving boundary problems on $\mathbb{R}$.

\begin{defn}

Let $(\Omega, \mathscr{F}, \mathscr{F}_t,\mathbb{P})$ be a complete filtered probability space. Let $\dot{W}$ be a space time white noise on this space which respects the filtration $\mathscr{F}_t$. Suppose that $\tilde{v}$ is a continuous $\mathscr{F}_t$-adapted process taking values in $\mathscr{C}_r$. Let $h: \mathscr{C}_r \times \mathscr{C}_r \rightarrow \mathscr{C}_r$ and $F:\mathscr{C}_r \rightarrow \mathscr{C}_r$ be Lipschitz functions. For the $\mathscr{F}_t$-stopping time $\tau$, we say that the pair $(v, \eta)$ is a local $\mathscr{C}_r$-valued solution to the reflected SPDE 

\begin{equation*}
\frac{\partial v}{\partial t}= \Delta v + h(v,\tilde{v}) \frac{\partial F(v)}{\partial x} + f(x, v) + \sigma(x,v) \dot{W} + \eta
\end{equation*}
with Dirichlet boundary condition $v(t,0)=0$ and initial data $v_0 \in \mathscr{C}_r^+$ with $v_0(0)=0$, until time $\tau$, if
\begin{enumerate}[(i)]
\item For every $x \geq 0$ and every $t \geq 0$,  $v(x,t)$ is $\mathscr{F}_t$- measurable.
\item $v \geq 0$ almost surely.
\item $v \big|_{[0,t] \times [0,\infty)} \in \mathscr{C}_r^t$ for every $t < \tau$ almost surely.
\item $v(t,x)= \infty$ for every $t \geq \tau$ almost surely.
\item $\eta$ is a measure on $[0,\infty) \times [0,\infty)$ such that
\begin{enumerate}
\item For every measurable map $\psi: [0,\infty) \times [0,\infty) \rightarrow \mathbb{R}$, 
\begin{equation}
\int_0^{\infty} \int_0^{\infty} \psi(x,s) \; \eta(\textrm{d}x,\textrm{d}s)
\end{equation}
is $\mathscr{F}_t$-measurable.
\item $\int_0^{\infty} \int_0^{\infty} v(t,x) \; \eta(\textrm{dx,dt})=0$.
\end{enumerate} 
\item There exists a localising sequence of stopping times $\tau_n \uparrow \tau$ almost surely, such that for every  $\varphi \in C^{1,2}_c([0,\infty) \times [0,\infty))$ with $\varphi(s,0)=0$, and for every $t \geq 0$, 
\begin{equation}\label{weak2}
\begin{split}
\int_0^{\infty} v(t \wedge \tau_n ,x) \varphi(t \wedge \tau_n,x) \textrm{d}x= & \int_0^{\infty} v(0,x) \varphi(0,x) \textrm{d}x +  \int_0^{t \wedge \tau_n} \int_0^{\infty} v(s,x) \frac{\partial \varphi}{\partial t}(s,x)\textrm{d}x\textrm{d}s \\ & +  \int_0^{t \wedge \tau_n} \int_0^{\infty} v(s,x) \frac{\partial^2 \varphi}{\partial x^2}(s,x) \textrm{d}x\textrm{d}s \\ & - \int_0^{t \wedge \tau_n} \int_0^{\infty} F(v(s,\cdot))(x)h(v(s,\cdot),\tilde{v}(s,\cdot)) \frac{\partial \varphi}{\partial x}(s,x) \textrm{d}x\textrm{d}s\\ & + \int_0^{t \wedge \tau_n} \int_0^{\infty}  f(x,v(s,x)) \varphi(s,x) \textrm{d}x\textrm{d}s \\ & + \int_0^{t \wedge \tau_n} \int_0^{\infty}  \sigma(x,v(s,x)) \varphi(s,x)W(\textrm{d}x,\textrm{d}s) \\ & + \int_0^{t \wedge \tau_n} \int_0^{\infty}  \varphi(s,x) \; \eta(ds,dx).
\end{split}
\end{equation}
almost surely.
\end{enumerate}
\end{defn}
Similarly to as in Section 3, we say that a local $\mathscr{C}_r$-valued solution is \emph{maximal} if it cannot be extended to a $\mathscr{C}_r$-valued solution on a larger stochastic interval, and we say that a local solution is \emph{global} if we can take $\tau_n = \infty$ in (\ref{weak2}).

\begin{defn}
Let $(\Omega, \mathscr{F}, \mathscr{F}_t,\mathbb{P})$ be a complete filtered probability space. Let $\dot{W}$ be a space time white noise on this space which respects the filtration $\mathscr{F}_t$. We say that the quintuple $(u^1, \eta^1, u^2, \eta^2, p)$ is a local solution to the moving boundary problem on $\mathbb{R}$ with exponential growth $r$ and initial data $(u^1_0,u^2_0, p_0)$, where $(u_0^1 \circ (\Theta_{p_0}^1)^{-1}, u_0^2 \circ (\Theta_{p_0}^2)^{-1}) \in \mathscr{C}_r^+ \times \mathscr{C}_r^+$, up to the $\mathscr{F}_t$-stopping time $\tau$ if 
\begin{enumerate}[(i)]
\item $(v^1, \tilde{\eta}^1):=(u^1 \circ (\theta^1_p)^{-1}, \eta^1 \circ (\theta^1_p)^{-1})$ is a $\mathscr{C}_r$-valued solution to the non-linear SPDE
\begin{equation}
\frac{\partial v^1}{\partial t}= \Delta v^1 - p^\prime(t)\frac{\partial v^1}{\partial x} + f_1(x,v^1) + \sigma_1(x,v^1)\dot{W} + \tilde{\eta}^1 
\end{equation}
with Dirichlet boundary condition $v^1(t,0)=0$ and initial data $v^1_0= u^1_0 \circ (\Theta_{p_0}^1)^{-1} \in \mathscr{C}_r^+$, until time $\tau$.
\item $(v^2, \tilde{\eta}^2):=(u^2 \circ (\theta^2_p)^{-1}, \eta^2 \circ (\theta^2_p)^{-1})$ is a $\mathscr{C}_r$-valued solution to the non-linear SPDE
\begin{equation}
\frac{\partial v^2}{\partial t}= \Delta v^2 + p^\prime(t)\frac{\partial v^2}{\partial x} + f_2(x,v^2) + \sigma_2(x,v^2)\dot{W}^- + \tilde{\eta}^2 
\end{equation}
with Dirichlet boundary condition $v^2(t,0)=0$ and initial data $v^2_0= u^2_0 \circ (\Theta_{p_0}^2)^{-1} \in \mathscr{C}_r^+$, until time $\tau$.
\item $p(0)=p_0$ and $p'(t)=h(v^1(t,\cdot),v^2(t,\cdot))$.
\end{enumerate}
We refer to $(v^1, \tilde{\eta}^1, v^2, \tilde{\eta}^2)$ as the solution to the moving boundary problem in the relative frame.
\end{defn}

\subsection{Existence and Uniqueness}

As in the proof of Theorem \ref{Main}, we will use a Picard iteration in order to prove existence and uniqueness for a truncated version of this problem. There is some extra complexity introduced when trying to do this in the case of an infinite spatial domain. In particular, we should be more careful in how we truncate the problem.

\begin{thm}\label{Main2}
Let $(\Omega, \mathscr{F}, \mathscr{F}_t,\mathbb{P})$ be a complete filtered probability space. Let $\dot{W}$ be a space time white noise on this space which respects the filtration $\mathscr{F}_t$. There exists a unique maximal solution $(u^1,\eta^1, u^2, \eta^2, p)$ to the moving boundary problem on $\mathbb{R}$, with the blow-up time given by $$\tau := \sup\limits_{M >0} \left[ \inf\left\{ t \geq 0 \; \big| \; \|u^1\|_{\mathscr{C}_r} + \|u^2\|_{\mathscr{C}_r} \geq M \right\}\right],$$ with $\tau >0$ almost surely.
\end{thm}

The following notation will be used throughout the rest of the paper.

\begin{defn}
We define $G(t,x,y)$ to be the Dirichlet heat kernel on $[0, \infty)$, that is 
\begin{equation}
G(t,x,y):= \frac{1}{\sqrt{4\pi t}} \left[ \exp\left(- \frac{(x-y)^2}{4t} \right)- \exp\left(- \frac{(x+y)^2}{4t} \right) \right].
\end{equation}
For $r \in \mathbb{R}$, we also define the notation
\begin{equation}
G_r(t,x,y):=  e^{-r(x-y)}G(t,x,y).
\end{equation}
\end{defn}

Before proving Theorem \ref{Main2}, we present here some results which will be essential to the proof.

\begin{lem}\label{deterministic integral bound}
Let $r \in \mathbb{R}$. Suppose that $u \in L^{1}([0,T]; \mathscr{L}_r)$. Then we have that, for $t \in [0,T]$,
\begin{equation}
\sup\limits_{\tau \in [0,t]} \sup\limits_{x \geq 0} \left| e^{-rx} \int_0^{\tau} \int_0^{\infty} G(\tau-s,x,y)u(s,y) \textrm{d}y\textrm{d}s \right| \leq  C_{r,T} \int_0^t \|u\|_{s, \mathscr{L}_r} \textrm{d}s.
\end{equation}
\end{lem}
\begin{proof}
\begin{equation} \begin{split}
\left| e^{-rx}\int_0^{\tau} \int_0^{\infty} G(\tau-s,x,y)u(s,y) \textrm{d}y\textrm{d}s \right| = & \left|\int_0^{\tau} \int_0^{\infty} G_r(t-s,x,y)e^{-ry}u(s,y) \textrm{d}y\textrm{d}s\right| \\ \leq & \int_0^{\tau} \left( \int_0^{\infty} G_r(t,x,y) \textrm{d}y \right)\|u\|_{s,\mathscr{L}_r}\textrm{d}s \\  \leq & C_{r,T} \int_0^t \|u\|_{s,\mathscr{L}_r} \textrm{d}s.
\end{split}
\end{equation}
\end{proof}
We would like an analogous result which would allow us to control the noise term appearing in the mild formulation. The following lemmas will enable us to obtain such an estimate.

\begin{prop}\label{Infinite Holder}
Suppose that $r \in \mathbb{R}$. Let $u \in L^p(\Omega;L^{\infty}([0,T] ; \mathscr{L}_{r}))$ with $p>10$.
Define 
\begin{equation}
w(t,x):= e^{-rx} \int_0^t \int_0^{\infty} G(t-s,x,z) u(s,z) \textrm{W}(\textrm{d}z,\textrm{d}s).
\end{equation}
Then $w$ is continuous almost surely and for $x , y \in [0,\infty)$ and $0 \leq s \leq \tau \leq t \leq T$ we have that
\begin{equation}
\mathbb{E}\left[ |w(\tau,x)- w(s,y)|^p \right] \leq C_{p,T,r} \mathbb{E}\left[\int_0^t \|u\|_{s,\mathscr{L}_{r}}^p \textrm{d}s \right]  \left( |\tau-s|^{1/4} + |x-y|^{1/2} \right)^{p-4},
\end{equation}
\end{prop}
\begin{proof}
We have that 
\begin{equation}\begin{split}
\mathbb{E}& \left[ |w(\tau,x)- w(s,y)|^p \right]\\  \leq & C_p  \mathbb{E}\left[ \left|\int_s^{\tau} \int_0^{\infty} G_{r}(\tau-q,x,z) e^{-rz}u(q,z) \textrm{W}(\textrm{d}z,\textrm{d}q) \right|^p \right] \\ & + C_p  \mathbb{E}\left[ \left| \int_0^s \int_0^{\infty} \left[G_{r}(\tau-q,x,z)-G_{r}(\tau-q,y,z) \right] e^{-rz}u(q,z) \textrm{W}(\textrm{d}z,\textrm{d}q) \right|^p \right] \\& + C_p \mathbb{E}\left[ \left| \int_0^s \int_0^{\infty} \left[ G_{r}(\tau-q,y,z)- G_{r}(s-q,y,z) \right] e^{-rz}u(q,z) \textrm{W}(\textrm{d}z,\textrm{d}q) \right|^p \right].
\end{split}
\end{equation}
We bound the first term only, and note that the other terms follow similarly by the estimates from Proposition \ref{Exponential heat kernel estimates}. Burkholder's inequality gives
\begin{equation*} \begin{split}
\mathbb{E}& \left[ \left|\int_s^{\tau} \int_0^{\infty} G_{r}(\tau-q,x,z) e^{-rz}u(q,z) \textrm{W}(\textrm{d}z,\textrm{d}q) \right|^p \right] \\ & \leq C_p \mathbb{E}\left[ \left|\int_s^{\tau} \int_0^{\infty} G_{r}(\tau-q,x,z)^2 e^{-2rz}u(q,z)^2 \textrm{d}z\textrm{d}q \right|^{p/2} \right] \\ & \leq C_p \mathbb{E}\left[ \left|\int_s^{\tau} \left( \int_0^{\infty} G_{r}(\tau-q,x,z)^2 \textrm{d}z \right) \|u \|_{q,\mathscr{L}_r}^2 \textrm{d}q \right|^{p/2} \right] 
\end{split}
\end{equation*}
H\"{o}lder's inequality then gives
\begin{equation}\begin{split}
\mathbb{E} & \left[ \left|\int_s^{\tau} \left( \int_0^{\infty} G_{r}(\tau-q,x,z)^2 \textrm{d}z \right) \|u \|_{q,\mathscr{L}_r}^2 \textrm{d}q \right|^{p/2} \right]  \\ & \leq C_{p,T}\left( \int_s^{\tau} \left[ \int_0^{\infty} G_{r}(\tau-q,x,z)^2 \textrm{d}z \right]^{p/(p-2)} \textrm{d}q \right)^{(p-2)/2} \times \mathbb{E} \left[ \int_s^{\tau}  \|u\|^p_{q,\mathscr{L}_{r}} \textrm{d}q \right]. 
\end{split}
\end{equation}
Applying the first bound from Proposition \ref{Exponential heat kernel estimates} then gives that this is at most 
\begin{equation*}
C_{r,p, T} |\tau-s|^{(p-4)/4} \times \mathbb{E} \left[ \int_s^{\tau}  \|u\|^p_{q,\mathscr{L}_{r}} \textrm{d}q \right].
\end{equation*}
By making similar arguments, using the other bounds from Proposition \ref{Exponential heat kernel estimates}, we obtain that 
\begin{equation*}
\mathbb{E} \left[ |w(\tau,x)- w(s,y)|^p \right] \leq C_{r,p, T} \mathbb{E} \left[ \int_0^t  \|u\|^p_{q,\mathscr{L}_{r}} \textrm{d}q \right] \times \left( |\tau-s|^{1/4}+ |x-y|^{1/2} \right)^{p-4} .
\end{equation*} 
Continuity of $w$ then follows by Corollary A.3 in \cite{Dalang}. 
\end{proof}
The following result is a reformulation of Lemma 3.4 in \cite{Otobe}.

\begin{lem}\label{GRR}
Let $p, K , \delta >0$. Suppose that $w: [0,T] \times [0,\infty) \rightarrow \mathbb{R}$ is a random field such that for every $s,t \in [0,T]$, every $n$ and every $x,y \in [n,n+1]$, 
\begin{equation}
\mathbb{E}\left[ |w(t,x)-w(s,y)|^p \right] \leq K ( |t-s| + |x-y|)^{2+ \epsilon}.
\end{equation}
Then for every $\delta >0$, there exists a constant $C$ depending only on $p$, $\epsilon$, $T$ and $\delta$, and a non-negative random variable $Y$ such that
\begin{equation}
\|w\|_{\mathscr{C}_{\delta}^T} \leq C( |w(0,0)| +Y),
\end{equation}
almost surely, where $\mathbb{E}\left[ Y^p \right] \leq CK$.
\end{lem} 

\begin{prop}\label{white noise integral bound}
Let $r \in \mathbb{R}$. Suppose that $u \in L^p(\Omega; L^{\infty}([0,T]; \mathscr{L}_{r-\epsilon}))$.  Then we have that, for $t \in [0,T]$ and $p> 12$, 
\begin{equation}
\mathbb{E}\left[ \sup\limits_{\tau \in [0,t]} \sup\limits_{x \geq 0} \left| e^{-rx} \int_0^{\tau} \int_0^{\infty} G(\tau-s,x,y)u(s,y) \textrm{W}( \textrm{d}y, \textrm{d}s) \right|^p \right] \leq C_{p,r,T,\epsilon} \mathbb{E}\left[\int_0^t \|u\|_{s,\mathscr{L}_{r-\epsilon}}^p \textrm{d}s \right].
\end{equation}
\end{prop}
\begin{proof}
Define
 \begin{equation}
w(t,x):= e^{-(r-\epsilon)x} \int_0^t \int_0^{\infty} G(t-s,x,z) u(s,z) \textrm{W}(\textrm{d}y,\textrm{d}s).
\end{equation}
Then Proposition \ref{Infinite Holder} gives that for $\tau, s \in [0,t]$ and $x,y \in [0,1]$, 
\begin{equation}
\mathbb{E}\left[ |w(\tau,x)- w(s,y)|^p \right] \leq C_{p,T,r,\epsilon} \mathbb{E}\left[\int_0^t \|u\|^p_{q,\mathscr{L}_{(r-\epsilon)}} \textrm{d}q \right]  \left( |\tau-s|^{1/4} + |x-y|^{1/2} \right)^{p-4}.
\end{equation}
It then follows by Lemma \ref{GRR} that, for $p>12$,
\begin{equation}
\mathbb{E}\left[\|w\|_{t,\mathscr{C}_{\epsilon}}^p \right] \leq C_{p,T,r, \epsilon} \mathbb{E}\left[\int_0^t \|u\|^p_{s,\mathscr{L}_{(r-\epsilon)}} \textrm{d}s \right].
\end{equation}
So we have the result.
\end{proof}

We are now in position to prove Theorem \ref{Main2} with a Picard iteration. Since the ideas for the remaining arguments are similar to those in the proof of Theorem \ref{Main}, we give an outline of the strategy only.

\begin{proof}[Proof of Theorem \ref{Main2}]
Our strategy is as follows:
\begin{enumerate}
\item We note that, by the definition, it is sufficient to prove existence and uniqueness for maximal solutions to the coupled SPDEs 
\begin{enumerate}
\item $(v^1, \tilde{\eta^1}):=(u^1 \circ (\theta^1_p)^{-1}, \eta^1 \circ (\theta^1_p)^{-1})$, a $\mathscr{C}_r$-valued solution to the non-linear SPDE
\begin{equation}
\frac{\partial v^1}{\partial t}= \Delta v^1 - h(v^1(t,\cdot),v^2(t,\cdot))\frac{\partial v^1}{\partial x} + f_1(x,v^1) + \sigma_1(x,v^1)\dot{W} + \eta 
\end{equation}
with Dirichlet boundary condition $v^1(0)=0$ and initial data $v^1_0= u^1_0 \circ (\Theta_{p_0}^1)^{-1} \in \mathscr{C}_r^+$.

\item $(v^2, \tilde{\eta^2}):=(u^2 \circ (\theta^2_p)^{-1}, \eta^2 \circ (\theta^2_p)^{-1})$, a $\mathscr{C}_r$-valued solution to the non-linear SPDE
\begin{equation}
\frac{\partial v^2}{\partial t}= \Delta v^2 + h(v^1(t,\cdot),v^2(t,\cdot))\frac{\partial v^2}{\partial x} + f_2(x,v^2) + \sigma_2(x,v^2)\dot{W}^- + \eta 
\end{equation}
with Dirichlet boundary condition $v^2(0)=0$ and initial data $v^2_0= u^2_0 \circ (\Theta_{p_0}^2)^{-1} \in \mathscr{C}_r^+$.
\end{enumerate}

\item We once again consider a truncated version of the problem. That is, we find $(v^1_M, \eta_M^1, v^2_M, \eta^2_M)$ such that 
\begin{enumerate}
\item $(v^1_M,\eta^1_M)$ solves the reflected SPDE 
\begin{equation}\label{truncated b}
\frac{\partial v^1_M}{\partial t}= \Delta v^1_M - h_{M,r}(v^1_M,v^2_M) \frac{\partial}{\partial x}(F_{M,r}(v^1_M)) + f_1(x, v^1_M) + \sigma_1(x,v^1_M) \dot{W} + \eta^1_M
\end{equation}
with Dirichlet boundary condition $v^1(0)=0$ and initial data $v^1_0= u^1_0 \circ (\Theta_{p_0}^1)^{-1} \in \mathscr{C}_r^+$.
\item $(v^2_M,\eta^2_M)$ solves the reflected SPDE 
\begin{equation}\label{truncated a}
\frac{\partial v^2}{\partial t}= \Delta v^2_M + h_{M,r}(v^1_M,v^2_M) \frac{\partial}{\partial x}(F_{M,r}(v^2_M)) + f_2(x, v^2_M) + \sigma_2(x,v^2_M) \dot{W}^- + \eta^2
\end{equation}
with Dirichlet boundary condition $v^2(0)=0$ and initial data $v^2_0= u^2_0 \circ (\Theta_{p_0}^2)^{-1} \in \mathscr{C}_r^+$.
\end{enumerate}

As in Proposition \ref{Truncated problem Compact}, $h_{M,r}$ is defined by applying $h$ to suitably truncated inputs, with the truncation function here being $F_{M,r}$. Formalising this, we define $h_{M,r}: \mathscr{C}_r \times \mathscr{C}_r \rightarrow \mathbb{R}$ by
\begin{equation}
h_{M,r}(v_1,v_2):= h(F_{M,r}(v_1),F_{M,r}(v_2)),
\end{equation}
with $F_{M,r}: \mathscr{C}_r \rightarrow \mathscr{C}_r$ is given by $$F_{M,r}(u)(x):= e^{rx} \min(e^{-rx}u(x),M).$$ 
\item Use a Picard argument to prove global existence and uniqueness for the solution to the truncated problem on the finite time interval $[0,T]$. The first approximations are given by $v^1_{M,n}(t,x)=v^1_0(x)$ and $v^2_{M,n}(t,x)=v^2_0(x)$ for all time. For $n \geq 1$, we let $w^1_{M,n+1}$ solve the SPDE
\begin{equation}
\frac{\partial w_{M,n+1}^1}{\partial t} = \Delta w_{M,n+1}^1 - h_{M,r}(v^1_{M,n},v^2_{M,n}) \frac{\partial}{\partial x}(v^1_{M,n} \wedge M) + f_1(x, v^1_{M,n}) + \sigma_1(x,v^1_{M,n}) \dot{W}
\end{equation}
with Dirichlet boundary condition $w^1_{M,n+1}(t,0)=0$ and initial data $w_{M,n+1}^1(0,x)=v_0^1(x)$. Writing this in mild form gives the expression 
\begin{equation}\begin{split}
w^1_{M,n+1}(t,x)= & \int_0^{\infty} G(t,x,y)v_0^1(y) \textrm{d}y \\ & - \int_0^t \int_0^{\infty} \frac{\partial G}{\partial y}(t-s,x,y) h_{M,r}(v^1_{M,n}(s,\cdot), v^2_{M,n}(s,\cdot)) F_{M,r}(v^1_{M,n}(s,\cdot))(y) \textrm{d}y \textrm{d}s \\ & + \int_0^t \int_0^{\infty} G(t-s,x,y)f_1(y,v^1_{M,n}(s,y)) \textrm{d}x\textrm{d}s \\ & + \int_0^t \int_0^{\infty} G(t-s,x,y) \sigma_1(y,v^1_{M,n}(s,y)) \textrm{W}(\textrm{d}y,\textrm{d}s).
\end{split}
\end{equation}
We then set $v^1_{M,n+1}:= w^1_{M,n+1}+ z^1_{M,n+1}$, where $z^1_{M,n+1}$ solves the obstacle problem with obstacle $-w^1_{M,n+1}$. We similarly define $w^2_{M,n+1}$ and $v^2_{M,n+1}$. Our Lipschitz conditions on $h$, $f$ and $\sigma$, together with the estimates from Lemma \ref{deterministic integral bound}, Proposition \ref{white noise integral bound} and Proposition \ref{Infintie derivative bound} allow us to argue as in the proof of Propsition \ref{Truncated problem Compact} to obtain that, for $t \in [0,T]$,
\begin{multline}\label{pic1}
\mathbb{E}\left[\|w_{M,n+1}^1-w_{M,n}^1\|_{\mathscr{C}_r^t}^p+ \|w_{M,n+1}^2-w_{M,n}^2\|_{\mathscr{C}_r^t}^p \right] \\ \leq C_{M,p,T,r,\delta} \int_0^t \mathbb{E}\left[\|v_{M,n+1}^1-v_{M,n}^1\|_{\mathscr{C}_r^s}^p+ \|v_{M,n+1}^2-v_{M,n}^2\|_{\mathscr{C}_r^s}^p \right] \textrm{d}s.
\end{multline}
Theorem \ref{Obstacle INfinite} gives that, for $t \in [0,T]$ and $i=1,2$,
\begin{equation}
\|v_{M,n+1}^i-v_{M,n}^i \|_{\mathscr{C}_r^t} \leq C_{r,T} \|w_{M,n+1}^i - w_{M,n}^i \|_{\mathscr{C}_r^t}.
\end{equation}
Plugging this into (\ref{pic1}) then gives that, for $t \in [0,T]$,
\begin{multline}\label{pic2}
\mathbb{E}\left[\|v_{M,n+1}^1-v_{M,n}^1\|_{\mathscr{C}_r^t}^p+ \|v_{M,n+1}^2-v_{M,n}^2\|_{\mathscr{C}_r^t}^p \right] \\ \leq C_{M,p,T,r,\delta} \int_0^t \mathbb{E}\left[\|v_{M,n+1}^1-v_{M,n}^1\|_{\mathscr{C}_r^s}^p+ \|v_{M,n+1}^2-v_{M,n}^2\|_{\mathscr{C}_r^s}^p \right] \textrm{d}s.
\end{multline}
Arguing as in Proposition \ref{Truncated problem Compact}, we see that $(v_{M,n}^1,v_{M,n}^2)_{n \geq 1}$ is Cauchy in $L^p(\Omega; \mathscr{C}_r^T)^2$ for large enough $p$, and the limit solves the truncated problem given by equations (\ref{truncated b}) and (\ref{truncated a}). 
\item As in Proposition \ref{Truncated problem Compact},  uniqueness for the truncated problem can be shown by applying the same estimates as in the proof of existence and concluding with a Gronwall argument.
\item We note the consistency of the truncated problems different truncation values $M$ and use this to define a solution to the problem until the $\|.\|_{\mathscr{C}_r}$ norm blows up.
\item We observe that uniqueness of the truncated problems implies uniqueness for the original moving boundary problem.
\item To deduce that $\tau >0$ almost surely, consider $w_M^i$, the solution to the SPDE 
\begin{equation}
\frac{\partial w_{M}^1}{\partial t} = \Delta w_{M}^1 - h_{M,r}(v^1_{M,n},v^2_{M,n}) \frac{\partial}{\partial x}(v^1_{M,n} \wedge M) + f_1(x, v^1_{M,n}) + \sigma_1(x,v^1_{M,n}) \dot{W}.
\end{equation}
By Propositions A.1 and A.3, we have that $w_M ^i \in C([0,T]; \mathscr{C}_r)$ almost surely. It follows that, for $M$ large enough, $\rho_M >0$, where 
\begin{equation}
\rho_M = \inf \{ t \geq 0 \; | \; \|w_M^1\|_{\mathscr{C}_r^t}+ \|w_M^2\|_{\mathscr{C}_r^t} \geq M/2 \}.
\end{equation}
Since $\| v_M^i \|_{\mathscr{C}_r^t} \leq 2 \|w_M^i \|_{\mathscr{C}_r^t}$, we have that, for large enough $M$, $\tau_M >0$, where 
\begin{equation}
\tau_M = \inf \{t \geq 0 \; | \; \|v_M^1\|_{\mathscr{C}_r^t}+ \|v_M^2\|_{\mathscr{C}_r^t} \geq M \}.
\end{equation}
It follows that $\tau >0$.
\end{enumerate}
\end{proof}
\begin{prop}
Suppose that $h$ is bounded. Then the solution to the moving boundary problem on $\mathbb{R}$ is global.
\end{prop}
\begin{proof}
We argue as in Proposition \ref{global existence}, replacing bounds on $H$ with the corresponding bounds on $G_r(t,x,y)$ and $G_{r+\delta}(t,x,y)$.
\end{proof}

\begin{rem}
We note that our uniqueness result here extends the existing theory for uniqueness for reflected SPDEs on infinite spatial domains. Until now, uniqueness had only been shown for equations 
\begin{equation}
\frac{\partial u}{\partial t} = \Delta u + f(x,u) + \sigma(x,u)\dot{W} + \eta
\end{equation}
in the case when $\sigma$ is constant. This was proved in \cite{Otobe}, where the spatial domain was $\mathbb{R}$ (this makes no difference to the arguments here). Choosing $h=0$ in our equations i.e. a static boundary, we obtain uniqueness for solutions to these equations in the spaces $\mathscr{C}_r$, provided that the dependence of the volatility on the solution itself decays exponentially, as in conditions (iii) and (iv) in the formulation of the problem.
\end{rem}

\subsection{H\"{o}lder Continuity}

\begin{thm}\label{Holder Continuity2}
For $i=1,2$, let $u_0^i$ be such that $v_0^i :=u_0^i \circ (\Theta_{p_0}^i)^{-1} \in \mathscr{C}_r^+$, with 
\begin{equation*}
|e^{-rx}v_0^i(x) - e^{-ry}v_0^i(y)| \leq C_{\gamma} |x-y|^{\gamma/2}
\end{equation*}
for every $\gamma \in (0,1)$ and every $x,y \in [0,\infty)$.
Then, for every $M>0$ and every $\gamma \in (0,1)$ the solution $(v^1_M,v^2_M)$ to the $M$-truncated problem with initial data $(v_0^1,v_0^2)$, described by equations (\ref{truncated b}) and (\ref{truncated a}), is locally $\gamma/4$-H\"{o}lder in time and $\gamma/2$-H\"{o}lder in space. In particular, if $(u^1,\eta^1,u^2,\eta^2,p)$ is the solution to our moving boundary problem with initial data $(u_0^1,u_0^2,p_0)$, then $(u^1,u^2)$ enjoys the same H\"{o}lder regularity locally until the blow-up time, $\tau$.
\end{thm} 

The proof of this is similar in spirit to that of Theorem \ref{Holder Continuity}. There are, however, some intricate differences which arise due to the infinite spatial domain. We first introduce here a modified version of Lemma \ref{PDE BOund} which is suitable for this context.

\begin{prop}\label{PDE BOund1}
Let $r >0$. Let $V \in C^{1,2}_b([0,T] \times [0,\infty))$ and $\psi,F \in C([0,T] \times [0,\infty))$ with $\psi \leq 0$
 bounded and $|F(t,x)| \leq Ke^{Rx}$ for some $K,R >0$. Suppose that 
\begin{equation}
\frac{\partial V}{\partial t}= \frac{1}{2}V^{''} + \psi V + \psi F
\end{equation}
with Dirichlet or Neumann boundary conditions at zero, and zero initial data. Then there exists a constant $C_{r,T}$ such that 
\begin{equation}
\|V\|_{\mathscr{C}_r^T} \leq C_{r,T} \|F\|_{\mathscr{C}_r^T}.
\end{equation}
\end{prop}
\begin{proof}
We will prove the result for the Neumann boundary condition- the argument for the Dirichlet condition is essentially the same. Let $(B_t^x)_{t \geq 0}$ be a Brownian motion on $[0,\infty)$ with reflection at $0$, started at $x$. Then by arguing as in Lemma 3.6 in \cite{DalangMuller}, we have that 
\begin{equation}
V_t(x)= \mathbb{E} \left[ \int_0^t \exp \left(\int_0^s \psi_{t-r}(B_r^x) \textrm{d}r \right)\psi_{t-s}(B_s^x)F_{t-s}(B_s^x) \textrm{d}s \right].
\end{equation}
Therefore
\begin{equation}
e^{-rx}V_t(x)= \mathbb{E} \left[ \int_0^t \exp \left(\int_0^s \psi_{t-r}(B_r^x) \textrm{d}r \right)\psi_{t-s}(B_s^x)F_{t-s}(B_s^x)e^{-rB_s^x} (e^{rB_s^x-rx}) \textrm{d}s \right].
\end{equation}
Hence
\begin{equation}\label{BM} \begin{split}
|e^{-rx}V_t(x)| \leq & \|F\|_{\mathscr{C}_r^T} \times \mathbb{E} \left[ - \int_0^t \exp \left(\int_0^s \psi_{t-r}(B_r^x) \textrm{d}r \right)\psi_{t-s}(B_s^x) (e^{rB_s^x-rx}) \textrm{d}s \right] \\ \leq & \|F\|_{\mathscr{C}_r^T} \times \mathbb{E} \left[ - \int_0^t \exp \left(\int_0^s \psi_{t-r}(B_r^x) \textrm{d}r \right)\psi_{t-s}(B_s^x) \textrm{d}s \left(\sup\limits_{s \in [0,t]}(e^{rB_s^x-rx})\right) \right]  \\ \leq & \|F\|_{\mathscr{C}_r^T} \times \mathbb{E}\left[\left(1-\exp\left(\int_0^t \psi_{t-r}(B_r^x) \textrm{d}r\right) \right)\left(\sup\limits_{s \in [0,t]}(e^{rB_s^x-rx})\right) \right] \\ \leq & \|F\|_{\mathscr{C}_r^T} \times \mathbb{E}\left[\sup\limits_{s \in [0,t]}(e^{rB_s^x-rx}) \right] \leq \|F\|_{\mathscr{C}_r^T} \times \mathbb{E} \left[\sup\limits_{s \in [0,T]}(e^{rB_s^x-rx}) \right] \\ = & \|F\|_{\mathscr{C}_r^T} \times e^{-rx} \times \mathbb{E}\left[ \sup\limits_{s \in [0,T]} e^{rB_s^x} \right].
\end{split}
\end{equation}
We note that the law of $B^x$ is simply the law of $|W^x|$, where $W$ is a standard Brownian motion (no reflection) started from $x$. Therefore 
\begin{equation}
e^{-rx}\mathbb{E}\left[\sup\limits_{s \in [0,T]} e^{rB_s^x}\right] =e^{-rx} \mathbb{E}\left[\sup\limits_{s \in [0,T]} \left( e^{rW_s^x}\mathbbm{1}_{\left\{ W_s^x \geq 0 \right\}}+ e^{-rW_s^x}\mathbbm{1}_{\left\{ -W_s^x \geq 0 \right\}} \right) \right].
\end{equation}
By the symmetry of Brownian motion, this is at most 
\begin{equation}
2e^{-rx}\mathbb{E}\left[\sup\limits_{s \in [0,T]} e^{rW_s^x} \right] = 2\mathbb{E}\left[\sup\limits_{s \in [0,T]} e^{rW_s^0} \right] \leq 2 e^{Tr^2/2} \times \mathbb{E}\left[ \sup\limits_{s \in[0,T]} e^{rW_s^0-sr^2/2} \right].
\end{equation}
Since $e^{rW_t^0 - tr^2/2}$ is a square integrable martingale, we apply Cauchy-Schwarz and Doob's $L^2$ inequality to get 
\begin{equation}
2\mathbb{E}\left[\sup\limits_{s \in [0,T]} e^{rW_s^0-sr^2/2} \right] \leq 2\mathbb{E}\left[\sup\limits_{s \in [0,T]} e^{2rW_s^0-sr^2} \right]^{\frac{1}{2}} \leq 4 \mathbb{E}\left[ e^{2rW_T^0 - Tr^2} \right]^{\frac{1}{2}}= 4 e^{Tr^2/2}
\end{equation}
Therefore, we have that 
\begin{equation}
e^{-rx}\mathbb{E}\left[\sup\limits_{s \in [0,T]} e^{rB_s^x}\right] \leq 4e^{Tr^2}.
\end{equation}
Plugging this into (\ref{BM}), we obtain that
\begin{equation}
\|V\|_{\mathscr{C}_r^T} \leq 4e^{Tr^2} \|F\|_{\mathscr{C}_r^T}.
\end{equation}
So we have the result.
\end{proof}

\begin{proof}[Proof of Theorem \ref{Holder Continuity2}]
The argument broadly follows the steps in the proof of Theorem \ref{Holder Continuity} and consequently those in \cite{Dalang}. Fix some $T>0$. Let $(v_M^1, v_M^2)$ solve the $M$-truncated problem. Define $w_M^1$ so that 
\begin{equation}\begin{split}
w^1_{M}(t,x)= & \int_0^{\infty} G(t,x,y)v_0^1(y) \textrm{d}y \\ & - \int_0^t \int_0^{\infty} \frac{\partial G}{\partial y}(t-s,x,y) h_{M,r}(v^1_{M}(s,\cdot), v^2_{M}(s,\cdot)) F_{M,r}(v^1_{M}(s,\cdot))(y) \textrm{d}y \textrm{d}s \\ & + \int_0^t \int_0^{\infty} G(t-s,x,y)f_1(y,v^1_{M}(s,y)) \textrm{d}x\textrm{d}s \\ & + \int_0^t \int_0^{\infty} G(t-s,x,y) \sigma_1(y,v^1_{M}(s,y)) \textrm{W}(\textrm{d}y,\textrm{d}s).
\end{split}
\end{equation}
Let $\xi=2\gamma+12$, so that $\gamma=\xi/2 -6$. By applying the inequalities from Propositions \ref{Exponential heat kernel estimates} and \ref{HeatDerivativeEstimateInfinite}, together with Burkholder's inequality, we see that 
\begin{equation}
\mathbb{E}\left[|e^{-rx}w^1_M(t,x)-e^{-ry}w^1_M(s,y)|^{\xi} \right] \leq C(|t-s|^{1/2}+|x-y| )^{\frac{\xi}{2}-2}.
\end{equation}
Applying Corollary A.3 from \cite{Dalang}, we find random variables $\psi_m$ such that $\sup\limits_{m \in \mathbb{N}}\mathbb{E}[(\psi_m)^p] < \infty,$ and for $t,s \in[0,T]$, $x,y \in [m,m+1]$ 
\begin{equation}
|e^{-rx}w_M^1(t,x)-e^{-ry}w_M^1(s,y)| \leq \psi_m( |t-s|^{\frac{\gamma}{2}}+ |x-y|^{\gamma}).
\end{equation}
Let $\mu >0$. We note that, for $x,y \in [m,m+1]$ and $s,t \in [0,T]$, 
\begin{equation}\begin{split}
|e^{-(r+ \mu)x}w_M^1(t,x)-e^{-(r+\mu)y}w_M^1(s,y)| \leq & e^{-m\mu}|e^{-rx}w_M^1(t,x)-e^{-ry}w_M^1(s,y)|\\ & + |e^{-rx}w_M^1(t,x)||e^{-y\mu}-e^{-x\mu}|.
\end{split}
\end{equation}
By considering the derivative of $e^{-\mu x}$ we see that 
\begin{equation}
|e^{-y\mu}-e^{-x\mu}| \leq \mu|x-y|e^{-m\mu}.
\end{equation}
Therefore
\begin{equation}\begin{split}
|e^{-(r+\mu)x}w_M^1(t,x)-e^{-(r+\mu)y}w_M^1(s,y)| \leq & e^{-m\mu}|e^{-rx}w_M^1(t,x)-e^{-ry}w_M^1(s,y)|\\ & + \mu|x-y|e^{-m\mu}|e^{-rx}w_M^1(t,x)|.
\end{split}
\end{equation}
Define the random variable $$Z:= \mu \|w_M^1\|_{\mathscr{C}_r^T}.$$ Then we know that for $p \geq 1$, $\mathbb{E}\left[Z^p \right]< \infty.$ Setting $Y_m:=\psi_m + Z$ we have that $R:=\sup\limits_{m \geq 0} \mathbb{E}[Y_m^p] < \infty.$
Therefore, for any $t,s \in [0,T]$ and $x,y \geq 0$ we have that 
\begin{equation}\label{abc}\begin{split}
|e^{-(r+\mu) x}w_M^1(t,x)-e^{-(r+\mu) y}w_M^1(s,y)| \leq \left(\sum\limits_{m=0}^{\infty} Y_me^{-m \mu} \right) \times (|t-s|^{\frac{\gamma}{2}}+|x-y|^{\gamma}) .
\end{split}
\end{equation}
Define $$X:= \sum\limits_{m=0}^{\infty} Y_me^{-m\mu}.$$ Then 
\begin{equation}\begin{split}
\mathbb{E}\left[X^p \right]= \mathbb{E}\left[ \left(\sum\limits_{m=0}^{\infty} Y_me^{-m\mu}\right)^p \; \right] \leq & \mathbb{E}\left[ \sum\limits_{m=0}^{\infty} Y_m^pe^{-m\mu p/2} \right]\times \left( \sum\limits_{m=0}^{\infty} e^{-m\mu q/2} \right)^{\frac{p}{q}} \\ = & C_{r,p,\mu} \sum\limits_{m=0}^{\infty} \mathbb{E}[Y_m^p]e^{-mp\mu /2} = C_{r,p,\mu,R} < \infty. 
\end{split}
\end{equation}
Now let $z^{\epsilon}$ solve the PDE 
\begin{equation}\label{obstacleinf}
\frac{\partial z^{\epsilon}}{\partial t}= \Delta z^{\epsilon}+ g_{\epsilon}(z^{\epsilon} + w_M^1)
\end{equation}
on $[0,\infty)$ with Dirichlet boundary condition at zero, and zero initial data, where we define $$g_{\epsilon}(x):= \frac{1}{\epsilon} \arctan([x \wedge 0]^2).$$ Then by Proposition \ref{infiniteobstacleexist}, $z^{\epsilon}+w_M^1$ increases to $v_M^1$, the solution of the reflected SPDE on $[0,\infty)$. Let $(e^{-(r+\mu) x}w_M^1)^{\alpha, \beta}$ be a smoothing of $e^{-(r+\mu) x}w_M^1$ as in Proposition \ref{smoothing}, with respect to the H\"{o}lder coefficients $\gamma/2$ and $\gamma/4$. Define $z^{\epsilon, \alpha, \beta}$ to be the solution of the PDE

\begin{equation}\label{obstaclesmoothinf}
\frac{\partial z^{\epsilon,\alpha, \beta}}{\partial t}= \Delta z^{\epsilon,\alpha, \beta}+ g_{\epsilon}(z^{\epsilon,\alpha, \beta} + e^{(r+ \mu) x}(e^{-(r+\mu) x}w_M^1)^{\alpha, \beta})
\end{equation}
with Dirichlet boundary condition at zero and and zero initial data. By Proposition \ref{obstaclepenalisationcontrol}, we obtain that 
\begin{equation}
\|z^{\epsilon}\|_{\mathscr{C}_{r+\mu}^T} \leq C_{r,T}\|w_M^1\|_{ \mathscr{C}_{r+\mu}^T},
\end{equation}
and
\begin{equation}
\|z^{\epsilon,a,b}\|_{\mathscr{C}_{r+\mu}^T} \leq \|(e^{-(r+\mu)x}w_M^1)^{\alpha, \beta}\|_{T, \infty}.
\end{equation}
Differentiating (\ref{obstaclesmoothinf}) with respect to $t$, we see that $q^{\epsilon,\alpha, \beta}:=\frac{\partial z^{\epsilon,\alpha, \beta}}{\partial t}$ solves
\begin{equation}
\frac{\partial q^{\epsilon,\alpha, \beta}}{\partial t}= \Delta q^{\epsilon,\alpha, \beta} + g^{\prime}_{\epsilon}(z^{\epsilon,\alpha, \beta} + e^{(r+\mu)x}(e^{-(r+\mu)x}w_M^1)^{\alpha, \beta})\left[q^{\epsilon,\alpha, \beta} + e^{(r+\mu)x}\frac{\partial(e^{-(r+\mu)x}w_M^1)^{\alpha, \beta}}{\partial t}  \right].
\end{equation}    
The boundary condition is Dirichlet, since $z$ does not change at 0, which means the time derivative is zero there. The initial data is $z_0^{\prime \prime}=0$, since $z_0$ is identically zero.  Note that $g^{\prime}_{\epsilon}$ is negative, so we can use Proposition \ref{PDE BOund1} to deduce that 
\begin{equation}
\| q^{\epsilon,\alpha, \beta}\|_{\mathscr{C}_{r+\mu}^T} \leq C_{r+\mu,T} \left\| \frac{\partial (e^{-(r+\mu)x}w_M^1)^{\alpha, \beta}}{\partial t}\right\|_{\infty,T} .
\end{equation}
Differentiating (\ref{obstaclesmoothinf}) with respect to $x$, we obtain, that $y^{\epsilon,\alpha, \beta}:= \frac{\partial z^{\epsilon,\alpha, \beta}}{\partial x}$ satisfies
\begin{equation}
\begin{split}
\frac{\partial y^{\epsilon,\alpha, \beta}}{\partial x}= &  \Delta y^{\epsilon,\alpha, \beta}+ g^{\prime}_{\epsilon}(z^{\epsilon,\alpha, \beta} + e^{(r+\mu)x}(e^{-(r+\mu)x}w_M^1)^{\alpha, \beta}) \left[y^{\epsilon,\alpha, \beta} \right. \\ &  \left. + e^{(r+\mu)x}\frac{\partial (e^{-(r+\mu)x}w_M^1)^{\alpha, \beta}}{\partial x}+ (r+\mu)e^{(r+\mu)x} (e^{-(r+\mu)x}w_M^1)^{\alpha, \beta} \right].
\end{split}
\end{equation} 
with initial data $z_0^{\prime}=0$ and Neumann boundary condition at zero. Proposition \ref{PDE BOund1} then gives  
\begin{equation}
\|y^{\epsilon,\alpha, \beta}\|_{\mathscr{C}_{r+\mu}^T} \leq C_{r+\mu,T}\left( \left\|\frac{\partial (e^{-(r+\mu)x}w_M^1)^{\alpha, \beta}}{\partial x} \right\|_{\infty,T} +\|(e^{-(r+\mu)x}w_M^1)^{\alpha, \beta}\|_{\infty,T} \right).
\end{equation}
Another application of Proposition \ref{obstaclepenalisationcontrol} gives
\begin{equation}
\|z^{\epsilon}-z^{\epsilon,\alpha,\beta}\|_{\mathscr{C}_{r+\mu}^T} \leq \|e^{-(r+\mu)x}w_M^1-(e^{-(r+\mu)x}w_M^1)^{\alpha, \beta} \|_{\infty,T}.
\end{equation}
We clearly have that 
\begin{equation}\begin{split}
|e^{-(r+\mu)x}z^{\epsilon}(t,x)-e^{-(r+\mu)y}z^{\epsilon}(s,y)| \leq & 2\|e^{-(r+\mu)x}z^{\epsilon} - e^{-(r+\mu)x} z^{\epsilon,\alpha, \beta}\|_{\infty,T} \\ & + |e^{-(r+\mu)x}z^{\epsilon,\alpha, \beta}(t,x)-e^{-(r+\mu)x}z^{\epsilon,\alpha, \beta}(s,x)| \\ & + |e^{-(r+\mu)x}z^{\epsilon,\alpha, \beta}(s,x)-e^{-(r+\mu)y}z^{\epsilon,\alpha, \beta}(s,y)|.  
\end{split}
\end{equation}
By the estimates deduced above, we can bound this by
\begin{multline}
2C_{r+\mu,T}\|e^{-(r+\mu)x}w_M^1 - (e^{-(r+\mu)x} w_M^1)^{\alpha, \beta}\|_{\infty,T}  + |t-s|\left\| \partial z^{\epsilon,\alpha, \beta}/\partial t \right\|_{\mathscr{C}_{r+\mu}^T} \\ + |x-y| \left\| \partial (e^{-(r+\mu)x} z^{\epsilon,\alpha, \beta})/ \partial x \right\|_{\infty,T}.
\end{multline}
This is at most 
\begin{equation*}\begin{split}
2 & \|e^{-(r+\mu)x}w_M^1 - (e^{-(r+\mu)x} w_M^1)^{\alpha, \beta}\|_{\infty,T} + C_{r+\mu,T}|t-s| \left\| \frac{\partial (e^{-(r+\mu)x}w_M^1)^{\alpha, \beta}}{\partial t}\right\|_{\infty,T}\\ & + |x-y| \left( \left\|\frac{\partial z^{\epsilon,\alpha, \beta}}{\partial x} \right\|_{\mathscr{C}_{r+\mu}^T}+ \|(r+\mu)z^{\epsilon,\alpha, \beta}\|_{\mathscr{C}_{r+\mu}^T} \right) \\ \leq & 
2\|e^{-(r+\mu)x}w_M^1 - (e^{-(r+\mu)x} w_M^1)^{\alpha, \beta}\|_{\infty} + C_{r+\mu,T}|t-s|\left\| \frac{\partial (e^{-(r+\mu)x}w_M^1)^{\alpha, \beta}}{\partial t}\right\|_{\infty,T} \\ & + C_{r+\mu,T}|x-y| \left( \left\|\frac{\partial (e^{-(r+\mu)x}w_M^1)^{\alpha, \beta}}{\partial x} \right\|_{\infty,T}+ \|w_M^1\|_{\mathscr{C}_{r+\mu}^T} \right).
\end{split}
\end{equation*}
Making the choices $a= |t-s|$ and $b=|x-y|$, this is at most 
\begin{equation}
C_{\gamma,T,r+\mu}\left(X|t-s|^{\gamma /2} + X|x-y|^{\gamma} + |x-y|\|w_M^1 \|_{\mathscr{C}_{r+\mu}^T}  \right).
\end{equation}
Since $X$ and $\|w_M^1 \|_{\mathscr{C}_{r+\mu}^T}$ are in $L^p$, they are finite almost surely. Letting $\epsilon \downarrow 0$ and noting the inequality (\ref{abc}) gives that, for every $x,y \geq 0$ and every $s,t \in [0,T]$
\begin{equation}\label{holder45}
|e^{-(r+\mu)x}v_M^1(t,x)-e^{-(r+\mu)y}v_M^1(s,y)| \leq C_{\gamma,T,r+\mu}\left(X|t-s|^{\gamma /2} + X|x-y|^{\gamma} + |x-y|\|w_M^1 \|_{\mathscr{C}_{r+\mu}^T}  \right)
\end{equation}
almost surely. Since we know that $\|v_M^1\|_{\mathscr{C}_{r+\mu}^T} < \infty$ almost surely, it follows that $v_M^1$ is locally $\gamma /4$ H\"{o}lder in time and $\gamma/2$ H\"{o}lder in space almost surely.
\end{proof}
\begin{cor}\label{HolderPrice2}
Suppose that there exists $\mu >0$ such that for every $u_1,u_2,v_1,v_2 \in \mathscr{C}_r$,
\begin{equation}
|h(u_1,v_1)-h(u_2,v_2)| \leq K( \|u_1- u_2\|_{\mathscr{C}_{r+ \mu}}+ \|v_1-v_2\|_{\mathscr{C}_{r+ \mu}}).
\end{equation}
Then, for every $\gamma \in (0,1)$, the derivative of the boundary is locally $\gamma/4$-H\"{o}lder continuous on $[0, \tau)$, where
\begin{equation}
\tau := \inf \left\{ t \geq 0 \; \big| \; \|u^1\|_{\mathscr{C}_r} + \|u^2\|_{\mathscr{C}_r} = \infty \right\}.
\end{equation}
\end{cor}
\begin{proof}
Fix $\gamma \in (0,1)$. Recall that $p^{\prime}(t)= h(v^1(t, \cdot), v^2(t,\cdot))$, where $v^1(t,x)= u^1(t,p(t)- \cdot)$ and $v^2(t,x)=u^2(t,p(t)+ \cdot)$. For $M >0$, define
\begin{equation}
\tau_M:=  \inf \left\{ t \geq 0 \; \big| \; \|u^1\|_{\mathscr{C}_r} + \|u^2\|_{\mathscr{C}_r} > M \right\} 
\end{equation}
Note that $\tau_M \uparrow \tau$ as $M \rightarrow \infty$. Let $t, s \in [0,\tau_M]$.
We have, by the Lipschitz property of $h$, that
\begin{equation}
|p^{\prime}(t)- p^{\prime}(s)| \leq K \left( \|v^1(t, \cdot) - v^1(s, \cdot)\|_{\mathscr{C}_{r+ \mu}} + \|v^2(t, \cdot) - v^2(s, \cdot)\|_{\mathscr{C}_{r+ \mu}} \right). 
\end{equation}
By equation (\ref{holder45}) in the proof of Theorem \ref{Holder Continuity2}, there exists an almost surely finite random variable $C_M$ such that this is at most $KC_M|t-s|^{\gamma/4}$, which gives the result.
\end{proof}

\begin{rem}
Choosing $h=0$, we see that in particular the solutions to the reflected SPDEs 
\begin{equation}
\frac{\partial u}{\partial t} = \Delta u + f(x,u) + \sigma(x,u)\dot{W} + \eta
\end{equation}
on $[0,\infty) \times [0,\infty)$ (or $[0,\infty) \times \mathbb{R}$ by the same arguments) are locally up to $1/4$-H\"{o}lder continuous in time and $1/2$-H\"{o}lder continuous in space.
\end{rem}

\textbf{Acknowledgements.} 
The research of J. Kalsi was supported by EPSRC (EP/L015811/1).

\appendix 

\section{Heat Kernel Estimates}
We present here some of the simple estimates for the heat kernels on $[0,1]$ and $[0,\infty)$ which were used throughout, and details of their proofs.

\subsection{Heat Kernel on $[0,\infty)$}
Recall that the Dirichlet heat kernels on $[0,\infty)$ is given by
\begin{equation}
G(t,x,y):= \frac{1}{\sqrt{4\pi t}} \left[ \exp\left(- \frac{(x-y)^2}{4t} \right)- \exp\left(- \frac{(x+y)^2}{4t} \right) \right],
\end{equation}
and that $G_r(t,x,y):= e^{-r(x-y)}G(t,x,y).$
We define the functions $$F_1(t,x,y):= \frac{1}{\sqrt{4 \pi t}}\exp\left(- \frac{(x-y)^2}{4t} \right),$$ and $$F_2(t,x,y):= \frac{1}{\sqrt{4 \pi t}}\exp\left(- \frac{(x+y)^2}{4t} \right).$$ 
In this section, the proofs focus on the $F_1$ component of $G$. The arguments for the $F_2$ components are similar. 

\begin{prop}\label{Exponential heat kernel estimates}
Fix $r \in \mathbb{R}$, $T>0$. Then, for $p>4$, we have that 
\begin{enumerate}
\item For every $t,s \in [0,T]$, $$\sup\limits_{x \geq 0} \left(\int_s^t \left[ \int_0^{\infty} G_r(t-u,x,z)^2 \textrm{d}z \right]^{p/(p-2)} \textrm{d}u \right)^{\frac{p-2}{2}} \leq C_{p,r,T} |t-s|^{(p-4)/4}.$$
\item For every $t,s \in[0,T]$, 
$$\sup\limits_{x \geq 0}\left(\int_0^s \left[ \int_0^{\infty} \left|G_r(t-u,x,z)-G_r(s-u,x,z) \right|^2 \textrm{d}z \right]^{p/(p-2)} \textrm{d}u \right)^{\frac{p-2}{2}} \leq C_{p,r,T} |t-s|^{(p-4)/4}.$$
\item For every $x,y \in [0,\infty)$
$$\sup\limits_{s \in [0,T]}\left(\int_0^s \left[ \int_0^{\infty} \left|G_r(s-u,x,z)-G_r(s-u,y,z) \right|^2 \textrm{d}z \right]^{p/(p-2)} \textrm{d}u \right)^{\frac{p-2}{2}} \leq C_{p,r,T} |x-y|^{(p-4)/2}.$$
\end{enumerate}
\end{prop}
\begin{proof}[Proof of 1]
Note that 
\begin{equation}\label{F1}
e^{-r(x-z)}F_1(t,x,z)= e^{r^2 t} F_1(t,x+2 r t,z)= e^{r^2 t} F_1(t,x,z- 2rt).
\end{equation}
Therefore, 
\begin{equation}\begin{split}
\int_s^t \left[ \int_0^{\infty} \left|e^{-r(x-z)}F_1(t-u,x,z)\right|^2 \textrm{d}z \right]^{\frac{p}{p-2}} \textrm{d}u  & \leq C_{p,r,T} \int_s^t \left[ \int_{\mathbb{R}} \left|F_1(u,x,z) \right|^2 \textrm{d}z \right]^{p/(p-2)} \textrm{d}u \\ & =  C_{p,r,T} \int_s^t \left[ \int_{\mathbb{R}}  \frac{1}{4 \pi u}\exp\left( - \frac{(x-z)^2}{4u} \right)\textrm{d}z \right]^{p/(p-2)} \textrm{d}u \\ & \leq C_{p,r,T} \int_s^t u^{-\frac{p}{2(p-2)}} \textrm{d}u. 
\end{split}
\end{equation}
If $p>4$, we have that $-\frac{p}{2(p-2)}>-1$, and so this is equal to
\begin{equation}
C_{p,r,T} |t-s|^{\frac{p-4}{2(p-2)}}.
\end{equation}
The result follows.
\end{proof}
\begin{proof}[Proof of 2]
We again make use of equation (\ref{F1}). This gives that 
\begin{equation} \begin{split}
\left|e^{-r(x-z)}F_1(t,x,z)-e^{-r(x-z)}F_1(s,x,z) \right| \leq & e^{Tr^2}|F_1(t,x,z-2rt)-F_1(s,x,z-2rt)| \\ & + F_1(s,x,z-2rt)|e^{r^2t}-e^{r^2 s} | \\ \leq & e^{Tr^2}|F_1(t,x,z-2rt)-F_1(s,x,z-2rt)| \\ & + e^{Tr^2} F_1(s,x,z-2rt)|t-s| .
\end{split}
\end{equation}
Therefore,
\begin{equation}\begin{split}
& \left(  \int_0^s \left[ \int_0^{\infty} e^{-r(x-z)}|F_1(t-u,x,z)-F_1(s-u,x,z)|^2 \textrm{d}z \right]^{p/(p-2)} \textrm{d}u \right)^{\frac{p-2}{2}} \\ &  \leq C_{p,r,T}\left( \int_0^s \left[ \int_0^{\infty} |F_1(t-u,x,z)-F_1(s-u,x,z)|^2 \textrm{d}z \right]^{p/(p-2)} \textrm{d}u \right)^{\frac{p-2}{2}} \\ & \; \; \; + C_{p,r,T}|t-s|^p \left(\int_0^s \left[ \int_0^{\infty} |F_1(s-u,x,z-2rt)|^2 \textrm{d}z \right]^{p/(p-2)} \textrm{d}u \right)^{\frac{p-2}{2}} \\ &  \leq C_{p,r,T}\left( \int_0^s \left[ \int_0^{\infty} |F_1(t-u,x,z)-F_1(s-u,x,z)|^2 \textrm{d}z \right]^{p/(p-2)} \textrm{d}u \right)^{\frac{p-2}{2}} \\ & \; \; \; + C_{p,r,T}|t-s|^p \left(\int_0^s \left[ \int_{\mathbb{R}} |F_1(s-u,x,z)|^2 \textrm{d}z \right]^{p/(p-2)} \textrm{d}u \right)^{\frac{p-2}{2}}.
\end{split}
\end{equation}
The integral in the second term is integrable if $p>4$. Therefore, the second term is equal to $C_{p,r,T}|t-s|^p$ for $p>4$. For the first term, we have that it is equal to 
\begin{equation}
\begin{split}
& C_{p,r,T} \int_0^s  \left[ \int_0^{\infty} \left| \frac{1}{\sqrt{(t-s)+(s-u)}} e^{-(x-z)^2/4((t-s)+(s-u))} - \frac{1}{\sqrt{s-u}} e^{-(x-z)^2/4(s-u)} \right|^2 \textrm{d}z \right]^{p/(p-2)} \textrm{d}u \\ & = C_{p,r,T} \int_0^s  \left[ \int_0^{\infty} \left| \frac{1}{\sqrt{(t-s)+u}} e^{-(x-z)^2/4((t-s)+u)} - \frac{1}{\sqrt{u}} e^{-(x-z)^2/4u} \right|^2 \textrm{d}z \right]^{p/(p-2)} \textrm{d}u.
\end{split}
\end{equation}
By making the substitution $v= \frac{u}{|t-s|}$, we see that this integral is at most
\begin{equation}
|t-s| \int_0^{\infty} \left[ \int_{\mathbb{R}} \left| \frac{1}{\sqrt{(t-s)(1+v)}} e^{-z^2/4((t-s)(1+v))} - \frac{1}{\sqrt{(t-s)v}} e^{-z^2/4(t-s)v} \right|^2 \textrm{d}z \right]^{p/(p-2)} \textrm{d}v.
\end{equation}
The substitution $\tilde{z}= \frac{z}{\sqrt{t-s}}$ then gives that this is equal to 
\begin{equation}\label{A7}
|t-s|^{\frac{p-4}{2(p-2)}} \int_0^{\infty} \left[ \int_{\mathbb{R}} \left| \frac{1}{\sqrt{1+v}} e^{-\tilde{z}^2/4(1+v)} - \frac{1}{\sqrt{v}}e^{-\tilde{z}^2/4v} \right|^2 \textrm{d}\tilde{z} \right]^{p/(p-2)} \textrm{d}v.
\end{equation}
The integral in (\ref{A7}) converges provided that $p>4$, giving the result.
\end{proof}

\begin{proof}[Proof of 3]
By (\ref{F1}), we have that 
\begin{equation} \begin{split}
& \int_0^s \left[ \int_0^{\infty} | e^{-r(x-z)}F_1(s-u,x,z)-e^{-r(y-z)}F_1(s-u,y,z)|^2 \textrm{d}z \right]^{p/(p-2)} \textrm{d}u \\ & \leq C_{p,r,T} \int_0^s \left[ \int_0^{\infty} |F_1(s-u,x,z-2rt)-F_1(s-u,y,z-2rt)|^2 \textrm{d}z \right]^{p/(p-2)} \textrm{d}u  \\ & \leq C_{p,r,T} \int_0^s \left[  \int_{\mathbb{R}} |F_1(s-u,x,z)-F_1(s-u,y,z)|^2 \textrm{d}z \right]^{p/(p-2)} \textrm{d}u  \\ & = C_{p,r,T} \int_0^s \left[ \int_{\mathbb{R}} \frac{1}{u} \left[ \exp \left(- \frac{((x-y)-(z-y))^2}{4u} \right) - \exp \left( - \frac{(z-y)^2}{4u} \right) \right]^2 \textrm{d}z \right]^{p/(p-2)} \textrm{d}u \\ & \leq C_{p,r,T} \int_0^{\infty} \left[ \int_{\mathbb{R}} \frac{1}{u} \left[ \exp \left(- \frac{((x-y)-h)^2}{4u} \right) - \exp \left( - \frac{h^2}{4u} \right) \right]^2 \textrm{d}h \right]^{p/(p-2)} \textrm{d}u.
\end{split}
\end{equation}
Making the change of variables $w=\frac{h}{|x-y|}$ and $v=\frac{u}{(x-y)^2}$, we see that this is equal to 
\begin{equation}
C_{p.r,T}|x-y|^{\frac{p-4}{p-2}} \int_0^{\infty} \left[ \int_{\mathbb{R}} \frac{1}{ v}\left| e^{-(1+w)^2/4v} - e^{-w^2/4v} \right|^2 \textrm{d}w \right]^{p/(p-2)}\textrm{d}v.
\end{equation}
The integral here converges provided that $p>4$, and only depends on $p$. We therefore have the result.
\end{proof}

\begin{prop}\label{Infintie derivative bound}
Let $r \in \mathbb{R}$. Then for $t \in [0,T]$,  
\begin{equation}
\sup\limits_{ x \geq 0} \left( \int_0^{\infty} e^{-r(x-y)}\left|\frac{\partial G}{\partial y}(t,x,y)\right| \textrm{d}y\right) \leq \frac{C_{r,T}}{\sqrt{t}}.
\end{equation}
\end{prop}
\begin{proof}
It is sufficient to prove the result for $F_1$ and $F_2$ separately. Calculating gives that  
\begin{equation}\begin{split}
e^{-r(x-y)}\frac{\partial F_1}{\partial y}(t,x,y)& = \frac{(x-y)}{8 t^{3/2}}\exp\left(-\frac{(x-y+2r t)^2}{4t} \right) \exp \left( r^2 t \right) \\ & \leq \frac{(x-y)}{8 t^{3/2}}\exp\left(-\frac{(x-y+2r t)^2}{4t} \right) \exp \left( r^2 T \right).
\end{split}
\end{equation}
Therefore
\begin{equation}\begin{split}
\int_0^{\infty} e^{-r(x-y)}\left| \frac{\partial F_1}{\partial y}(t,x,y) \right| \textrm{d}y \leq &  C_{r,T} \int_0^{\infty} \frac{(x-y)}{t^{3/2}}\exp\left(-\frac{(x-y+2r t)^2}{4t} \right)  \textrm{d}y \\ \leq & C_{r,T} \int_0^{\infty} \frac{(x-y+2r t)}{t^{3/2}}\exp\left(-\frac{(x-y+2r t)^2}{4t} \right)  \textrm{d}y \\ \leq & C_{r,T} \int_{\mathbb{R}} \frac{y}{t^{3/2}} \exp \left( -\frac{y^2}{4t} \right) \textrm{d}y = \frac{C_{r,T}}{\sqrt{t}}. 
\end{split}
\end{equation}
The proof for the $F_2$ part of the heat kernel is similar. By noting that 
\begin{equation}
e^{-r(x-y)} \frac{\partial F_2}{\partial y} (t,x,y)= \frac{x+y}{8 t^{3/2}} \exp \left( - \frac{(x+y-2r t)^2}{4t} \right) \exp \left(r^2 t -2rx \right) 
\end{equation}
and arguing as before, we obtain the result.
\end{proof}

\begin{prop}\label{HeatDerivativeEstimateInfinite}
Let $\delta \in \mathbb{R}$. Then for every $t,s \in [0,T]$ and every $x,y \in [0,\infty)$, and every $p>4$, we have that
\begin{enumerate}
\item $\left[ \int_0^s \left( \int_0^{\infty} \left|e^{-\delta(x-z)}\frac{\partial G}{\partial z}(r,x,z)- e^{-\delta(y-z)}\frac{\partial G}{\partial z}(r,y,z) \right| \textrm{d}z \right)^{p/(p-2)} \textrm{d}r \right]^{(p-2)/p} \leq C_{p,T,\delta} |x-y|^{(p-4)/2p},$
\item $\left[ \int_0^s \left( \int_0^{\infty} \left|e^{-\delta(x-z)}\frac{\partial G}{\partial z}(t-r,x,z)- e^{-\delta(x-z)}\frac{\partial G}{\partial z}(s-r,x,z) \right| \textrm{d}z \right)^{p/(p-2)} \textrm{d}r \right]^{(p-2)/p} \\ \leq C_{p,T,\delta}|t-s|^{(p-4)/4p},$
\item $\left[ \int_s^t \left( \int_0^{\infty} \left| e^{-\delta(x-z)}\frac{\partial G}{\partial z}(t-r,x,z)\right| \textrm{d}z \right)^{p/(p-2)} \textrm{d}r \right]^{(p-2)/p} \leq C_{p,T,\delta}|t-s|^{(p-4)/4p}.$
\end{enumerate}
\end{prop}
\begin{proof}[Proof of 1]
First note that
\begin{equation}\label{decomp}\begin{split}
e^{-\delta(x-z)}\frac{\partial F_1}{\partial z}(r,x,z)=  & \frac{(x-z+2\delta r)}{8 r^{3/2}}\exp\left(-\frac{(x-z+2\delta r)^2}{4r} \right) \exp \left( \delta^2 r \right) \\ & -  \frac{\delta}{4\sqrt{r}}\exp\left(-\frac{(x-z+2\delta r)^2}{4r} \right) \exp \left( \delta^2 r \right).
\end{split}
\end{equation}
We bound the terms corresponding to these two components separately. Let $q=p/(p-2)$. Then, by applying H\"{o}lder's inequality, we have that 
\begin{equation}
\begin{split}
\int_0^s &  \left( \int_0^{\infty} \left|\frac{(x-z+2\delta r)}{8 r^{3/2}}\exp\left(-\frac{(x-z+2\delta r)^2}{4r} \right) -\frac{(y-z+2\delta r)}{8 r^{3/2}}\exp\left(-\frac{(y-z+2\delta r)^2}{4r} \right) \right| \textrm{d}z \right)^{q} \textrm{d}r \\ \leq &  \int_0^s  \left( \int_{\mathbb{R}} \left|\frac{(x-y+h)}{8 r^{3/2}}\exp\left(-\frac{(x-y+h)^2}{4r} \right) -\frac{h}{8 r^{3/2}}\exp\left(-\frac{h^2}{4r} \right) \right| \textrm{d}z \right)^{q} \textrm{d}r  \\ \leq &  \int_0^{\infty}  \left( \int_{\mathbb{R}} \left|\frac{(x-y+h)}{8 r^{3/2}}\exp\left(-\frac{(x-y+h)^2}{4r} \right) -\frac{h}{8 r^{3/2}}\exp\left(-\frac{h^2}{4r} \right) \right| \textrm{d}z \right)^{q} \textrm{d}r.  
\end{split}
\end{equation}
We split the time integral into two parts- the integral on $[0, |x-y|]$ and the integral on $(|x-y|, \infty)$. For the first of these domains, we have
\begin{equation}
\int_0^{|x-y|}  \left( \int_{\mathbb{R}} \left|\frac{(x-y+h)}{8 r^{3/2}}\exp\left(-\frac{(x-y+h)^2}{4r} \right) -\frac{h}{8 r^{3/2}}\exp\left(-\frac{h^2}{4r} \right) \right| \textrm{d}z \right)^{q} \textrm{d}r.  
\end{equation}
Letting $h=|x-y|u$, we obtain 
\begin{equation}
\begin{split}
\int_0^{|x-y|} &   \left( \int_{\mathbb{R}} |x-y|^2 \left|\frac{(u+1)}{8 r^{3/2}}\exp\left(-\frac{(u+1)^2(x-y)^2}{4r} \right) -\frac{u}{8 r^{3/2}}\exp\left(-\frac{u^2(x-y)^2}{4r} \right) \right| \textrm{d}z \right)^{q} \textrm{d}r \\ & \leq C_p |x-y|^{p/(2p-4)} \int_0^{|x-y|}  \left( \frac{1}{\sqrt{r}(x-y)^2} \right)^{q} \textrm{d}r \\ & = C_p \int_0^{|x-y|} r^{-p/(2p-4)} \textrm{d}r = C_p |x-y|^{(p-4)/(2p-4)}.
\end{split}
\end{equation}
To bound on $(|x-y|, \infty)$, we note that 
\begin{equation}
\frac{\textrm{d}}{\textrm{d}x} ( xe^{-x^2/r}) = e^{-x^2/r} -\frac{2x^2}{r}e^{-x^2/r} \leq Ce^{-x^2/2r}.
\end{equation}
Therefore, outside the region $u \in [-1,0]$, we have that 
\begin{equation}
\begin{split}
\left|\frac{(u+1)}{8 r^{3/2}}\exp\left(-\frac{(u+1)^2(x-y)^2}{4r} \right) -\frac{u}{8 r^{3/2}}\exp\left(-\frac{u^2(x-y)^2}{4r} \right) \right| \leq \frac{C}{r^{3/2}} \max \{ e^{-u^2(x-y)^2/4r}, e^{-(u+1)^2(x-y)^2/4r} \}
\end{split}
\end{equation}
In the region $u \in [-1,0]$, we have that 
\begin{equation}
\left|\frac{(u+1)}{8 r^{3/2}}\exp\left(-\frac{(u+1)^2(x-y)^2}{4r} \right) -\frac{u}{8 r^{3/2}}\exp\left(-\frac{u^2(x-y)^2}{4r} \right) \right| \leq C \frac{1}{r^{3/2}}.
\end{equation}
It follows that 
\begin{equation}
\begin{split}
\int_{|x-y|}^{\infty} &   \left( \int_{\mathbb{R}} |x-y|^2 \left|\frac{(u+1)}{8 r^{3/2}}\exp\left(-\frac{(u+1)^2(x-y)^2}{4r} \right) -\frac{u}{8 r^{3/2}}\exp\left(-\frac{u^2(x-y)^2}{4r} \right) \right| \textrm{d}z \right)^{q} \textrm{d}r \\ & \leq C_p \int_{|x-y|}^{\infty} \left[ \frac{|x-y|^2}{r^{3/2}} \left( 1 + \int_{\mathbb{R}} e^{-u^2(x-y)^2/4r} \textrm{d}u \right) \right]^{q} \textrm{d} r \\ & \leq C_p |x-y|^{2q} \int_{|x-y|}^{\infty} r^{-3q/2} + (r|x-y|)^{-q} \textrm{d}r \\ & = C_p |x-y|^{2q}( |x-y|^{-(3q-2)/2} + |x-y|^{(1-2q)} ) 
\end{split}
\end{equation}
Since $|x-y| \leq 1$, this is at most $C_p|x-y| \leq C_p|x-y|^{(p-4)/(2p-4)}.$ We have therefore deduced inequality (1) for the first component on the right hand side of expression (\ref{decomp}). For the second component of (\ref{decomp}), we have that that
\begin{equation}
\begin{split}
\int_0^{\infty} & \left|\frac{1}{\sqrt{r}}\exp\left(-\frac{(x-z+2\delta r)^2}{4r} \right) \exp \left( \delta^2 r \right)-\frac{1}{\sqrt{r}}\exp\left(-\frac{(y-z+2\delta r)^2}{4r} \right) \exp \left( \delta^2 r \right) \right| \textrm{d}z \\ & \leq \int_{\mathbb{R}} \left|\frac{1}{\sqrt{r}}\exp\left(-\frac{(x-y+h)^2}{4r} \right) \exp \left( \delta^2 r \right)-\frac{1}{\sqrt{r}}\exp\left(-\frac{h^2}{4r} \right) \exp \left( \delta^2 r \right) \right| \textrm{d}h.
\end{split}
\end{equation}
\begin{equation}
\left|\frac{\textrm{d}}{\textrm{d}x}(e^{-x^2/4r}) \right| \leq \frac{C}{\sqrt{r}} e^{-x^2/8r}.
\end{equation}
We therefore have that, for $h \notin [-|x-y|,0]$,
\begin{equation}
\left|\exp\left(-\frac{(x-y+h)^2}{4r} \right)-\exp\left(-\frac{h^2}{4r} \right)  \right| \leq \frac{C}{\sqrt{r}}|x-y|\max \{e^{-h^2/8r},e^{-(x-y+h)^2/8r} \}.
\end{equation}
For $h \in [-|x-y,0]$, we use the simple bound
\begin{equation}
\left|\exp\left(-\frac{(x-y+h)^2}{4r} \right)-\exp\left(-\frac{h^2}{4r} \right)  \right|  \leq 2.
\end{equation}
Putting this together, we obtain that 
\begin{equation}
\begin{split}
 \int_{\mathbb{R}} & \left|\frac{1}{\sqrt{r}}\exp\left(-\frac{(x-y+h)^2}{4r} \right) \exp \left( \delta^2 r \right)-\frac{1}{\sqrt{r}}\exp\left(-\frac{h^2}{4r} \right) \exp \left( \delta^2 r \right) \right| \textrm{d}h \\ & \leq 2\frac{|x-y|}{\sqrt{r}} + C|x-y| \int_{\mathbb{R}} \frac{1}{r}e^{-h^2/4r} \textrm{d}r = C\frac{|x-y|}{\sqrt{r}}.
 \end{split}
\end{equation}
It follows that 
\begin{equation}
\begin{split}
\int_0^s & \left[ \int_0^{\infty}  \left|\frac{1}{\sqrt{r}}\exp\left(-\frac{(x-z+2\delta r)^2}{4r} \right) \exp \left( \delta^2 r \right)-\frac{1}{\sqrt{r}}\exp\left(-\frac{(y-z+2\delta r)^2}{4r} \right) \exp \left( \delta^2 r \right) \right| \textrm{d}z \right]^{p/(p-2)} \textrm{d}r \\ & \leq C_{T,\delta,p} |x-y|^{p/(p-2)}.
\end{split}
\end{equation}
Inequality (1) for the second component of (\ref{decomp}) is then a simple consequence of this.
The manipulations required to prove inequalities (2) and (3) are similar, and we therefore omit these lengthy calculations.
\end{proof} 

\subsection{Heat Kernel on $[0,1]$}

Recall that 
\begin{equation}
H(t,x,y):=  \frac{1}{\sqrt{4 \pi t}}\sum\limits_{n =- \infty}^{\infty} \left[ \exp\left( - \frac{(x-y+2n)^2}{4t} \right) - \exp\left(- \frac{(x+y+2n)^2}{4t} \right) \right].
\end{equation}
We make the observation here that this expression can be written as 
\begin{equation}
H(t,x,y)= G(t,x,y) - \frac{1}{\sqrt{4 \pi t}} \exp \left( - \frac{(x+y-2)^2}{4t} \right) +L(t,x,y),
\end{equation}
where $L$ is a smooth function of $t,x,y$ which vanishes at $t=0$. Consequently, we are able to prove the estimates for $H$ in this section in analogous ways to how we proved the corresponding results for $G$. We therefore omit the proofs here.
\begin{prop}\label{HeatKernelEstimatesCompact}
Fix $T>0$. Then, for $p>4$, we have that
\begin{enumerate}
\item For every $t,s \in [0,T]$, 
$$\sup\limits_{x \in [0,1]}\left(\int_s^t \left[\int_0^{1} H(t-r,x,z)^{2} \textrm{dr} \right]^{p/(p-2)} \textrm{dz} \right)^{\frac{p-2}{2}} \leq C_{p} |t-s|^{(p-4)/4}.$$
\item For every $t,s \in [0,T]$, 
$$\sup\limits_{x \in [0,1]}\left(\int_0^s  \left[ \int_0^{1} \left(H(t-r,x,z)-H(s-r,x,z) \right)^{2} \textrm{ds} \right]^{p/(p-2)} \textrm{dz} \right)^{\frac{p-2}{2}} \leq C_{p} |t-s|^{(p-4)/4}.$$
\item For every $x,y \in [0,1]$,
$$\left(\int_0^s \left[ \int_0^{1} \left(H(s-r,x,z)-H(s-r,y,z) \right)^2 \textrm{d}z\right]^{p/(p-2)} \textrm{d}r \right)^{\frac{p-2}{2}} \leq C_{p} |x-y|^{(p-4)/2}.$$
\end{enumerate}
\end{prop}
\begin{prop}\label{Compact Heat Kernel Derivative}
For $t \in [0,T]$,  
\begin{equation}
\sup\limits_{ x \geq 0} \left( \int_0^{\infty} \left|\frac{\partial H}{\partial y}(t,x,y)\right| \textrm{d}y\right) \leq \frac{C_{\delta,T}}{\sqrt{t}}.
\end{equation}
\end{prop}
\begin{prop}\label{HeatDerivativeBoundsCompact}
For $t,s \in [0,T]$, $x,y \in [0,1]$ and $p>4$, we have that 
\begin{enumerate}
\item $\left[ \int_0^s \left( \int_0^1 \left|\frac{\partial H}{\partial z}(t-r,x,z)- \frac{\partial H}{\partial z}(s-r,x,z) \right| \textrm{d}z \right)^{p/(p-2)} \textrm{d}r \right]^{(p-2)/p} \leq C |t-s|^{(p-4)/4p},$ 
\item $\left[ \int_0^s \left( \int_0^1 \left|\frac{\partial H}{\partial z}(s-r,x,z)- \frac{\partial H}{\partial z}(s-r,y,z) \right| \textrm{d}z \right)^{p/(p-2)} \textrm{d}r \right]^{(p-2)/p} \leq C |x-y|^{(p-4)/2p},$ 
\item $\left[ \int_s^t \left( \int_0^1 \left| \frac{\partial H}{\partial z}(t-r,x,z)\right| \textrm{d}z \right)^{p/(p-2)} \textrm{d}r \right]^{(p-2)/p} \leq C|t-s|^{(p-4)/4p}.$
\end{enumerate}
\end{prop}

\section{Proof of Theorem \ref{Obstacle INfinite}}

In this section, we will prove a series of Propositions which together constitute a proof of Theorem \ref{Obstacle INfinite}.

\begin{prop}\label{penalisationobstacleincreasing}
Let $r>0$ and $v \in C([0,T]; \mathscr{C}_r)$ such that $v(t,0)=0$. For $\epsilon >0$, let $z^{\epsilon}$ be the solution to the PDE
\begin{equation}\label{PDE}
\frac{\partial z^{\epsilon}}{\partial t} = \Delta z^{\epsilon} + \frac{1}{\epsilon} \arctan(((z^{\epsilon}+v) \wedge 0)^2).
\end{equation}
Then $z^{\epsilon}$ increases as $\epsilon \downarrow 0$.
\end{prop}
\begin{proof}
Let $0< \epsilon_2 \leq \epsilon_1$. Fix some $\delta >0$ and set $y(t,x):= e^{-\delta x}(z^{\epsilon_1}(t,x)-z^{\epsilon_2}(t,x))$. We have that $y$ satisfies 
\begin{equation}
\frac{\partial y}{\partial t} = \Delta y + 2\delta \frac{\partial y}{\partial x} + \delta^2 y + \frac{e^{-\delta x}}{\epsilon_1} \arctan(((z^{\epsilon_1}+v) \wedge 0)^2)- \frac{e^{- \delta x}}{\epsilon_2} \arctan(((z^{\epsilon_2}+v) \wedge 0)^2).
\end{equation}
We know that $$\|y\|_{T, \mathscr{C}_{-\delta}} = \|z^{\epsilon_1} - z^{\epsilon_2} \|_{T,\infty} < \infty.$$ 
Testing our equation for $y$ with the positive part of $y$, we obtain that 
\begin{equation}\begin{split}
\|y_T^+ \|_{L^2}^2 = & - \int_0^T \left\|\frac{\partial y_t^+}{\partial x} \right\|_{L^2}^2 \textrm{d}t+ \delta \int_0^T \int_0^{\infty} \frac{\partial ((y^+)2)}{\partial x}(s,x)  \textrm{d}x\textrm{d}s + \delta^2 \int_0^T \|y^+_t\|_{L^2}^2 \textrm{d}t + \textrm{Negative Part}.
\end{split} 
\end{equation}
We note that testing against the last two terms gives a negative contribution, since when $y \geq 0$, we have that $z_{\epsilon_1} \geq z_{\epsilon_2}$ and so $ ((z^{\epsilon_1}+v) \wedge 0)^2 \leq  ((z^{\epsilon_2}+v) \wedge 0)^2$, from which it follows that $$\frac{e^{-\delta x}}{\epsilon_1} \arctan( ((z^{\epsilon_1}+v) \wedge 0)^2)- \frac{e^{- \delta x}}{\epsilon_2} \arctan(((z^{\epsilon_2}+v) \wedge 0)^2 \leq 0.$$
Putting this together, we see that 
\begin{equation}
\|y_T^+\|_{L^2}^2 \leq \delta^2 \int_0^T \|y_t^+ \|_{L^2}^2 \textrm{d}t.
\end{equation}
Gronwall's inequality then gives that $y_T^+=0$ i.e. that $z_{\epsilon_1} \leq z_{\epsilon_2}$.
\end{proof}

The following bound will allow us to control solutions of our obstacle problems by the obstacles themselves. 

\begin{prop}\label{obstaclepenalisationcontrol}
Let $r \in \mathbb{R}$ and $v_1$,$v_2 \in C([0,T]; \mathscr{C}_r)$ such that $v_1(t,0)=v_2(t,0)=0$. Fix $\epsilon>0$. For $i=1,2$, let $z_i^{\epsilon}$ be the solution to the PDE
\begin{equation}\label{PDE12}
\frac{\partial z_i^{\epsilon}}{\partial t} = \Delta z_{i}^{\epsilon} + \frac{1}{\epsilon} \arctan(((z_{i}^{\epsilon}+v_i) \wedge 0)^2)
\end{equation}
with boundary condition $z_i(t,0)=0$ and zero initial data. Then there exists a constant $C_{r,T}$ such that 
\begin{equation}
\|z_1^{\epsilon}-z_2^{\epsilon}\|_{\mathscr{C}_r^T} \leq C_{r,T} \|v_1-v_2\|_{\mathscr{C}_r^T}.
\end{equation}
\end{prop}

\begin{proof}
Let $w$ be given by
\begin{equation}
w(t,x)=e^{rx+r^2t}\phi(t),
\end{equation} 
where we define $\phi(t):= \|v_1-v_2\|_{\mathscr{C}_r^t}$. Then we have that 
\begin{equation}
\frac{\partial w}{\partial t} = \Delta w + e^{rx+r^2t}\frac{ \textrm{d}\phi}{\textrm{d}t}.
\end{equation} 
We note here that $\phi$ is positive and increasing, and that we interpret $\frac{\textrm{d}\phi}{\textrm{d}t}$ in a weak sense in the equation above. From the definition of $w$ we see that $w \geq v_2 -v_1.$ Let $\delta >\max(0,-r)$ and define $$\tilde{z}(t,x):= e^{-(\delta + r) x}(z_1^{\epsilon}(t,x) - z_2^{\epsilon}(t,x) - w(t,x)).$$ Then $\tilde{z}$ solves the equation 
\begin{equation}\label{aa}\begin{split}
\frac{\partial \tilde{z}}{\partial t} = & \Delta \tilde{z} + 2(\delta + r) \frac{\partial \tilde{z}}{\partial x} + (\delta + r)^2 \tilde{z} + \frac{e^{-(\delta + r)x}}{\epsilon} \arctan(((z_1^{\epsilon}+v_1) \wedge 0)^2) \\ & -\frac{e^{-(\delta + r) x}}{\epsilon} \arctan(((z_{i}^{\epsilon}+v_i) \wedge 0)^2) -e^{-\delta x + r^2t} \frac{\textrm{d}\phi}{\textrm{d}t},
\end{split}
\end{equation}
with zero initial data and boundary condition $\tilde{z}(t,0)=- e^{r^2t}\phi(t).$ Note that when $\tilde{z} \geq 0$, we have $z_1^{\epsilon}-z_2^{\epsilon} \geq v_2-v_1$, and so $$ \frac{e^{-(\delta + r) x}}{\epsilon} \arctan(((z_1^{\epsilon}+v_1) \wedge 0)^2)- \frac{e^{-(\delta + r) x}}{\epsilon} \arctan(((z_{2}^{\epsilon}+v_2) \wedge 0)^2) \leq 0.$$ Note also that the last term on the right hand side of (\ref{aa}) is negative. Therefore, testing the equation with $\tilde{z}^+$ we obtain that for $t \in [0,T]$

\begin{equation}\begin{split}
\int_0^t \int_0^{\infty} \frac{\partial \tilde{z}}{\partial t}(s,x) \tilde{z}^+(s,x) \textrm{d}x \textrm{d}s \leq & \int_0^t \int_0^{\infty} \Delta \tilde{z}(s,x) \tilde{z}^+(s,x) \textrm{d}x \textrm{d}s \\ &  + 2(\delta + r)\int_0^t \int_0^{\infty}  \frac{\partial \tilde{z}}{\partial x}(s,x) \tilde{z}^+(s,x) \textrm{d}x\textrm{d}s \\ & + (\delta + r)^2 \int_0^{\infty} \tilde{z}(s,x) \tilde{z}^+(s,x) \textrm{d}x \textrm{d}s.
\end{split}
\end{equation}
By integrating by parts and noting that $\tilde{z}^+$ is zero at time $t=0$ and vanishes at $x=0$ and $x=\infty$, we obtain that
\begin{equation}
\frac{1}{2} \|\tilde{z}_t^+\|_{L^2}^2 \leq (\delta + r)^2 \int_0^t \|\tilde{z}^+_s\|_{L^2} \textrm{d}s.
\end{equation}
It follows by an application of Gronwall's inequality that $z_t^+=0$. Therefore, we have that 
$$z_1^{\epsilon}(t,x)- z_2^{\epsilon}(t,x) \leq w(t,x).$$
Interchanging $z_1^{\epsilon}$ and $z_2^{\epsilon}$, we also have that $$z_2^{\epsilon}(t,x)- z_1^{\epsilon}(t,x) \leq w(t,x).$$It follows that 
$$\|z_1^{\epsilon}-z_2^{\epsilon}\|_{\mathscr{C}_r^T} \leq \|w \|_{\mathscr{C}_r^T} =e^{r^2T}\phi(T).$$ 
We therefore obtain that
\begin{equation}
\|z_1^{\epsilon}-z_2^{\epsilon}\|_{\mathscr{C}_r^T} \leq C_{r,T} \|v_1-v_2\|_{\mathscr{C}_r^T}.
\end{equation}
\end{proof}

We are now in position to argue existence for the obstacle problem on $[0,\infty)$.
\begin{prop}\label{infiniteobstacleexist}
Let $r \in \mathbb{R}$ and $v \in \mathscr{C}_r^T$, with $v(0,\cdot) \leq 0$. Then there exists $(z, \eta)$ solving the heat equation with obstacle $v$ and exponential growth $r$.
\end{prop}

\begin{proof}
Proposition \ref{penalisationobstacleincreasing} gives that the solutions $z^{\epsilon}$ to the equations (\ref{PDE12}) are increasing as $\epsilon \downarrow 0$. For $x \geq 0$ and $t \in [0,T]$, let 
\begin{equation}
z(t,x):= \lim\limits_{\epsilon \downarrow 0} z^{\epsilon}(t,x).
\end{equation}
By Proposition \ref{obstaclepenalisationcontrol}, we have that 
\begin{equation}
\|z^{\epsilon}\|_{\mathscr{C}_r^T} \leq C_{r,T}\|v\|_{\mathscr{C}_r^T}.
\end{equation}
Letting $\epsilon \downarrow 0$, it follows that 
\begin{equation}\label{obstaclecontrol1}
\sup\limits_{t \in [0,T]} \sup\limits_{x \in [0,1]} |e^{-rx}z(t,x)| \leq C_{r,T} \|v\|_{\mathscr{C}_r^T}.
\end{equation}
We also have that $z$ is continuous. The argument for this is as follows. Let $w_n \in C_c^{\infty}([0,T] \times (0,\infty))$ such that 
\begin{equation}
\sup\limits_{t \in [0,T]} \sup\limits_{x \geq 0} |w_n(t,x)- e^{-(r+\delta)x}v(t,x)| \rightarrow 0.
\end{equation}
Let $v_n(t,x):= e^{(r+\delta)x}w_n(t,x),$ so we have that 
\begin{equation}
\|v-v_n\|_{\mathscr{C}_{r+\delta}^T} \rightarrow 0.
\end{equation}
As we did in order to construct $z$, we define the functions $z_n^{\epsilon}$ to be the solutions to the equations 
\begin{equation}\label{PDE_n}
\frac{\partial z_n^{\epsilon}}{\partial t} = \Delta z_n^{\epsilon} + \frac{1}{\epsilon} \arctan(((z_n^{\epsilon}-v_n) \wedge 0)^2).
\end{equation}
We can then argue as in the proof of Theorem \ref{Holder Continuity2}, differentiating the equation in space and time respectively, and applying Proposition \ref{PDE BOund1} to see that 
\begin{equation*}
\left\|\frac{\partial z_n^{\epsilon}}{\partial x} \right\|_{\mathscr{C}_{r+\delta}^T} \leq C_{r+ \delta, T} \left\| \frac{\partial v_n}{\partial x} \right\|_{\mathscr{C}_{r+\delta}^T},
\end{equation*}
and 
\begin{equation*}
\left\|\frac{\partial z_n^{\epsilon}}{\partial t} \right\|_{\mathscr{C}_{r+\delta}^T} \leq C_{r+ \delta, T} \left\| \frac{\partial v_n}{\partial t} \right\|_{\mathscr{C}_{r+\delta}^T}
\end{equation*}
Note that we use the condition that $v(0,\cdot) \leq 0$ for this step, in order to ensure that the initial data for $\frac{\partial z_n^{\epsilon}}{\partial t}$ is zero. Therefore, $\frac{\partial z_n^{\epsilon}}{\partial x}$ and $\frac{\partial z_n^{\epsilon}}{\partial t}$ are uniformly bounded over $\epsilon$ on compact subsets of $[0,T] \times [0,\infty)$. We can now argue that $z$ is continuous. We have that, for $M>0$ and $(t,x), (s,y) \in [0,T] \times [0,M]$
\begin{equation*}
\begin{split}
|z(t,x)-z(s,y)| & =\lim_{\epsilon \downarrow 0} |z^{\epsilon}(t,x)- z^{\epsilon}(s,y)| \\ & \leq \liminf_{\epsilon \downarrow 0} \left[ |z_n^{\epsilon}(t,x)- z_n^{\epsilon}(s,y)| + 2\sup\limits_{t \in [0,T]} \sup\limits_{x \in [0,M]}\left| z^{\epsilon}(t,x)- z^{\epsilon}_n(t,x) \right| \right].
\end{split}
\end{equation*} 
We know that $$\sup\limits_{t \in [0,T]}\sup\limits_{x \in [0,M]}\left| z^{\epsilon}(t,x)- z^{\epsilon}_n(t,x) \right| \leq e^{(r+\delta)M}\| z^{\epsilon} - z^{\epsilon}_n \|_{\mathscr{C}_{r + \delta}} \leq e^{(r+\delta)M}\| v - v_n \|_{\mathscr{C}_{r + \delta}} \rightarrow 0$$ as $n \rightarrow \infty$, using the bound from Proposition \ref{obstaclepenalisationcontrol}. We also have that 
\begin{equation*}
\begin{split}
|z_n^{\epsilon}(t,x)- z_n^{\epsilon}(s,y)| & \leq e^{(r+\delta)M}\left[ \left\|\frac{\partial z_n^{\epsilon}}{\partial x} \right\|_{\mathscr{C}_{r+\delta}^T} +\left\|\frac{\partial z_n^{\epsilon}}{\partial t} \right\|_{\mathscr{C}_{r+\delta}^T}  \right] \left( |t-s| + |x-y| \right) \\ & \leq C_{M,r+\delta,T} \left[ \left\|\frac{\partial v_n}{\partial x} \right\|_{\mathscr{C}_{r+\delta}^T} +\left\|\frac{\partial v_n}{\partial t} \right\|_{\mathscr{C}_{r+\delta}^T}  \right] \left( |t-s| + |x-y| \right) \rightarrow 0
 \end{split}
\end{equation*}
as $(s,y) \rightarrow (t,x)$. It follows that 
\begin{equation*}
|z(t,x)-z(s,y)| \rightarrow 0
\end{equation*}
as $(s,y) \rightarrow (t,x)$, and so $z$ is continuous. Therefore, $z \in \mathscr{C}_r^T$. We now want to verify that $z$ solves the obstacle problem. Clearly, $z(t,0)=0$ and $z(0,x)=0$ for all $x$. Let $\varphi \in C_c^{\infty}([0,T] \times [0, \infty))$ with $\varphi(t,0)=0$ for every $t$. Testing the equation for $z^{\epsilon}$ with $\varphi$, we see that 
\begin{equation}\label{testpenalisedinfinite}\begin{split}
\int_0^{\infty} z^{\epsilon}(T,x) \varphi(T,x) \textrm{d}x  = &   \int_0^T \int_0^{\infty} z^{\epsilon}(t,x) \frac{\partial \varphi}{\partial t}(t,x)   \textrm{d}x \textrm{d}t+ \int_0^T \int_0^{\infty} z^{\epsilon}(t,x)\frac{\partial^2 \varphi}{\partial x^2}(t,x)\textrm{d}x\textrm{d}t \\ & +  \int_0^T \int_0^{\infty} \varphi(t,x) \frac{1}{\epsilon} \arctan(((z^{\epsilon}(t,x)-v(t,x)) \wedge 0)^2) \textrm{d}x\textrm{d}s.
\end{split}
\end{equation} 
Define $$\eta_{\epsilon}(\textrm{d}t,\textrm{d}x):= \frac{1}{\epsilon} \arctan(((z^{\epsilon}-v) \wedge 0)^2) \textrm{dxdt}.$$
Letting $\epsilon \rightarrow 0$, we see that $\eta_{\epsilon} \rightarrow \eta$ in distribution on $[0,T] \times (0,\infty)$, to some positive distribution $\eta$. It follows from it's positivity that $\eta$ is a measure, and we have that for every $\varphi \in C_c^{\infty}([0,T] \times [0,\infty))$ with $\varphi(t,0)=0$ for every $t$, 
\begin{equation}\begin{split}
\int_0^{\infty} z(T,x) \varphi(T,x) \textrm{d}x  = &   \int_0^T \int_0^{\infty} z(t,x) \frac{\partial \varphi}{\partial t}(t,x)   \textrm{d}x \textrm{d}t+ \int_0^T \int_0^{\infty} z(t,x)\frac{\partial^2 \varphi}{\partial x^2}(t,x)\textrm{d}x\textrm{d}t \\ & +  \int_0^T \int_0^{\infty} \varphi(t,x) \eta(\textrm{d}t,\textrm{d}x).
\end{split}
\end{equation} 
It is left to check that $z-v \geq 0$ and that the integral of $z-v$ against the measure $\eta$ is zero (i.e. the reflection measure is supported on the set where $z$ hits the obstacle, $v$). Multiplying (\ref{testpenalisedinfinite}) by $\epsilon$, we obtain that
\begin{equation}
\int_0^T \int_0^{\infty} \varphi(s,x)\arctan(e^{-rx} ((z^{\epsilon}-v) \wedge 0)^2) \textrm{d}x\textrm{d}s=o(\epsilon).
\end{equation}
Letting $\epsilon \rightarrow 0$, we see that 
\begin{equation}
\int_0^T \int_0^{\infty} \varphi(s,x)\arctan(e^{-rx} ((z-v) \wedge 0)^2) \textrm{d}x\textrm{d}s=0.
\end{equation}
It follows that we must have $z-v \geq 0$. Finally, we want to have that $$\int_0^T \int_0^{\infty} (z(t,x)-v(t,x)) \; \eta(\textrm{dx,dt})=0.$$ Since $z^{\epsilon}$ is increasing as $\epsilon$ decreases, we see that $supp(\eta) \subset supp(\eta_{\epsilon})$ for every $\epsilon >0$. Also, we have that $z^{\epsilon}-v \leq 0$ on the support of $\eta_{\epsilon}$, and so on the support of $\eta$. Therefore, for $\varphi \in C_c^{\infty}([0,T] \times (0,\infty))$ with $\varphi \geq 0$, we have that 
\begin{equation}
- \infty < \int_0^T \int_0^{\infty} \varphi(t,x) (z^{\epsilon}(t,x)-v(t,x)) \; \eta(\textrm{d}t,\textrm{d}x) \leq 0
\end{equation}
almost surely. By applying the DCT, noting that $\eta$ assigns finite mass to compact sets in $(0,\infty)$ almost surely, we obtain that $$\int_0^T \int_0^{\infty} \varphi(t,x)(z(t,x)-v(t,x)) \; \eta(\textrm{d}t,\textrm{d}x) \leq 0.$$ Since $z-v \geq 0$, this integral must be zero. So we have the result, and $(z,\eta)$ is a solution to the obstacle problem.
\end{proof}
We now turn to the problem of uniqueness. The following lemma is an adaptation of the result from Section 2.3 in \cite{NP}.
\begin{lem}\label{obstaclelem}
Let $(z_1, \eta_1)$ and $(z_2,\eta_2)$ be two solutions to the obstacle problem with obstacle $v \in \mathscr{C}_r^T$. Set $\psi(t,x):= z_1(t,x)-z_2(t,x)$. Then, for $\phi \in C_c^{\infty}([0,\infty))$ with $\phi(0)=0$, and $t \in [0,T]$, we have that
$$\int_0^{\infty} \psi^2(t,x) \phi^2(x) \textrm{d}x \leq  \int_0^t \int_0^{\infty} \psi(s,x)^2 (\phi^2)^{\prime \prime}(y) \textrm{d}x\textrm{d}s.$$
\end{lem}
\begin{proof}
Fix some $t <T$ and $\phi \in C_c^{\infty}((0,\infty))$. The result would follow if we could test the equation for $\psi$ with the function $\psi(t,x) \phi^2(x)$. Since this isn't possible, as $\psi$ isn't regular enough, we must test with a smooth approximation of this function and take a limit. Let $\epsilon$ be a non-negative function supported on $[-1,1]$ which is symmetric, smooth, positive definite and so that $$\int_{-1}^1 \epsilon(x) \textrm{d}x= 1.$$ We then obtain approximations of the identity, given by $\epsilon_n(x):=n \epsilon(n x)$. We now define the function of two variables, $\epsilon_{n,m}$ to be
\begin{equation}
\epsilon_{n,m}(t,x):= \epsilon_n(t) \epsilon_m(x). 
\end{equation}
Define the function $d_{n,m}$ to be given by
\begin{equation}
d_{n,m}:=((\psi \phi) \ast \epsilon_{n,m}) \phi,
\end{equation}
where $\ast$ here denotes the convolution on $\mathbb{R}^2$. That is, we define  
\begin{equation}
\begin{split}
d_{n,m}(t,x) & = \left(\int_{0}^{\infty} \int_0^{\infty} \psi(s,y)\phi(y) \epsilon_n(t-s) \epsilon_m(x-y) \textrm{d}y \textrm{d}s\right)\phi(x) \\ & = \left(\int_{(t-1/n)^+}^{(t+1/n)} \int_0^{\infty} \psi(s,y)\phi(y) \epsilon_n(t-s) \epsilon_m(x-y) \textrm{d}y \textrm{d}s\right)\phi(x).
\end{split}
\end{equation}
The function $d_{n,m}$ is now a smooth approximation of $\psi(t,x)\phi(x)^2$, so we can test the equation for $\psi$ against this function. Doing so gives
\begin{equation}\begin{split}
\int_0^{\infty} d_{n,m}(t,x) \psi(t,x) \; \textrm{d}x= & \int_0^t \int_0^{\infty} \frac{\partial (d_{n,m})}{\partial t}(s,x) \psi(s,x) \, \textrm{d}x\textrm{d}s + \int_0^t \int_0^{\infty} \psi(s,x) \frac{\partial^2 d_{n,m}}{\partial x^2}(s,x) \textrm{d}x\textrm{d}s \\ & + \int_0^t \int_0^{\infty} d_{n,m}(s,x) \, \eta_1(\textrm{dx,dt})- \int_0^t \int_0^{\infty} d_{n,m}(s,x) \, \eta_2(\textrm{dx,dt}).
\end{split}
\end{equation}
We take the limit for each term separately. The first term is 
\begin{equation}
\int_0^{\infty} d_{n,m}(t,x) \psi(t,x) \; \textrm{d}x.
\end{equation}
As $n,m \rightarrow \infty$, we can apply the DCT to see that this converges to \begin{equation}
\int_0^{\infty} d_{n,m}(t,x) \psi(t,x) \; \textrm{d}x \rightarrow \int_0^{\infty} \psi^2(t,x) \phi^2(x) \textrm{d}x.
\end{equation}
For the second term, we have that
\begin{multline}\label{1}
\int_0^t \int_0^{\infty} \frac{\partial (d_{n,m})}{\partial t}(s,x) \psi(s,x) \, \textrm{d}x\textrm{d}s \\ = \int_0^t \int_{s-1/n}^{s+1/n} \epsilon_n^{\prime}(s-r) \left( \int_0^{\infty} \int_0^{\infty} \psi(r,y) \phi(y) \epsilon_m(x-y) \phi(x) \psi(s,x) \textrm{d}x \textrm{d}y \right) \textrm{d}r\textrm{d}s.
\end{multline}
Define  $$\Gamma_m(r,s):= \int_0^{\infty} \int_0^{\infty} \psi(r,y) \phi(y) \epsilon_m(x-y) \phi(x) \psi(s,x) \textrm{d}x \textrm{d}y.$$ Then, since $\Gamma$ is symmetric in $r,s$ and that $\epsilon^{\prime}(r)=-\epsilon^{\prime}(-r)$, we see that by symmetry, (\ref{1}) is equal to 
\begin{equation}\label{smalldom}
\int_{t-\frac{1}{n}}^t \int_t^{s + \frac{1}{n}} \epsilon_n^{\prime}(s-r) \Gamma_m(r,s) \textrm{d}r \textrm{d}s.
\end{equation}
Let $A_n:= \{(r,s) \; | \; s \in [t-1/n,t], \; r \in [t,s+1/n] \}.$ We can choose $\epsilon$ so that it's derivative is approximately equal to $1$ on the interval $[-1,0]$, and therefore the derivative of $\epsilon_n$ is approximately $n^2$ on $[-1/n,0]$. We also have that $\Gamma_m(r,s)$ is equal to 
\begin{equation*}
\int_0^{\infty} \psi^2(t,x) \phi^2(x) \textrm{d}x + R_{m,n}
\end{equation*}
on $A_n$, where $R_{n,m} \rightarrow 0$ as $n,m \rightarrow \infty.$ It follows from these calculations, and the fact that $\mu_{\textrm{Leb}}(A)= n^2/2$ that (\ref{smalldom}) converges to 
\begin{equation}
\frac{1}{2}\int_0^{\infty} \psi^2(t,x) \phi^2(x) \textrm{d}x
\end{equation}
as $n,m \rightarrow \infty$. The limit of the combination of the reflection terms is at most zero, since 
\begin{equation}
\lim\limits_{n,m \rightarrow \infty} \left[ \int_0^t \int_0^{\infty} d_{n,m}(t,x) \left(\eta_1(\textrm{d}t,\textrm{d}x)- \eta_2(\textrm{d}t,\textrm{d}x) \right) \right] = \int_0^{\infty}  \psi(t,x) \phi(x)^2 \left(\eta_1(\textrm{d}t,\textrm{d}x)- \eta_2(\textrm{d}t,\textrm{d}x) \right).
\end{equation}
Since $z_i+v$ is zero on the support of $\eta_i$, this is equal to 
\begin{equation}
-\int_0^t \int_0^{\infty}  (z_2(s,x)+v(s,x))\phi(x)^2 \eta_1(\textrm{d}s, \textrm{d}x)-\int_0^t \int_0^{\infty}  (z_1(s,x)+v(s,x))\phi(x)^2 \eta_2(\textrm{d}s, \textrm{d}x) \leq 0.
\end{equation}
Finally, we deal with the second derivative term, 
\begin{equation}\label{2}
\int_0^t \int_0^{\infty} \psi(s,x) \frac{\partial^2 d_{n,m}}{\partial x^2}(s,x) \textrm{d}x\textrm{d}s.\end{equation}
Letting $n \rightarrow \infty$, we have that this converges to \begin{equation}\label{3}
\int_0^t \int_0^{\infty} \psi(s,x) \frac{\partial^2 d_{m}}{\partial x^2}(s,x) \textrm{d}x\textrm{d}s 
\end{equation} where $$d_m(t,x)= \left( \int_0^{\infty} \psi(t,y) \phi(y) \epsilon_m(x-y) \textrm{d}y \right) \phi(x).$$ In order to bound (\ref{3}), we first suppose that $\psi$ is smooth. Integrating by parts and making use of the positive definiteness of $\epsilon$, we obtain that (using angle bracket notation to indicate integration over $[0,\infty)$)
\begin{equation}\label{4}\begin{split}
\left\langle \frac{\partial^2}{\partial x^2}d_m(t,\cdot), \psi(t,\cdot) \right\rangle = & \left\langle \frac{\partial^2}{\partial x^2}(((\psi(t,\cdot)\phi(\cdot)) \ast \epsilon_m) \phi(\cdot)), \psi(t,\cdot) \right\rangle \\ = & - \left\langle \frac{\partial}{\partial x}(((\psi(t,\cdot)\phi(\cdot))\ast\epsilon_m)\phi(\cdot)), \frac{\partial}{\partial x}\psi(t,\cdot) \right\rangle \\ = & -\left\langle (\frac{\partial}{\partial x}(\psi(t, \cdot) \phi(\cdot))\ast \epsilon_m ) \phi(\cdot), \frac{\partial}{\partial x}\psi(t,\cdot) \right\rangle \\ & - \left\langle ((\psi(t,\cdot) \phi(\cdot)) \ast \epsilon_m)\frac{\partial}{\partial x}\phi(\cdot), \frac{\partial}{\partial x}\psi(t,\cdot) \right\rangle \\ = & - \left\langle (\frac{\partial}{\partial x} \psi(t,\cdot) \phi(\cdot)) \ast \epsilon_m, \phi(\cdot) \frac{\partial}{\partial x} \psi(t,\cdot) \right\rangle \\ & - \left\langle (\psi(t,\cdot) \frac{\partial}{\partial x}\phi(\cdot)) \ast \epsilon_m, \phi(\cdot) \frac{\partial}{\partial x} \psi(t,\cdot) \right\rangle \\ & - \left\langle ((\psi(t,\cdot) \phi(\cdot)) \ast \epsilon_m), \frac{\partial}{\partial x}\phi(\cdot) \frac{\partial}{\partial x}\psi(t,\cdot) \right\rangle.
 \end{split}
\end{equation}
We bound this last expression. Positive definiteness of $\epsilon$ ensures that the first term is negative. We split the last term into two, using that $$\frac{\partial \phi}{\partial x} \frac{\partial \psi}{\partial x}= \frac{\partial}{\partial x}\left( \frac{\partial \phi}{\partial x} \psi \right) - \frac{\partial^2 \phi}{\partial x^2}\psi,$$
Simple manipulations then give that (\ref{4}) is at most 
\begin{equation}
\left\langle (\psi(t,\cdot)\phi(\cdot)) \ast \epsilon_m, \frac{\partial^2 \phi}{\partial x^2}(\cdot) \psi(t,\cdot) \right\rangle + \left\langle (\psi(t,\cdot)\frac{\partial \phi}{\partial x}(\cdot)) \ast \epsilon_m, \frac{\partial \phi}{\partial x}(\cdot) \psi(t,\cdot) \right\rangle.
\end{equation}
So we have shown that, if $\psi$ were smooth, then 
\begin{equation}
\left\langle \frac{\partial^2 d_m}{\partial x}(t,\cdot), \psi(t,\cdot) \right\rangle \leq \left\langle (\psi(t,\cdot)\phi(\cdot)) \ast \epsilon_m, \frac{\partial^2 \phi}{\partial x^2}(\cdot) \psi(t,\cdot) \right\rangle + \left\langle (\psi(t,\cdot)\frac{\partial \phi}{\partial x}(\cdot)) \ast \epsilon_m, \frac{\partial \phi}{\partial x}(\cdot) \psi(t,\cdot) \right\rangle.
\end{equation}
By taking a sequence of smooth functions which approximate $\psi$ in the infinity norm on the support of $\phi$, we obtain (making use of the compact support of $\phi$ in order to pass to the limit for the integrals) that this formula holds for our original choice of $\psi$, $\psi= z_1-z_2$. Taking the $\limsup$ of both sides gives that 
\begin{equation}\begin{split}
\limsup\limits_{m \rightarrow \infty} \left\langle \frac{\partial^2 d_m}{\partial x}(t,\cdot), \psi(t,\cdot) \right\rangle & \leq \left\langle \psi(t,\cdot)\phi(\cdot), \frac{\partial^2 \phi}{\partial x^2}(\cdot) \psi(t,\cdot) \right\rangle + \left\langle \psi(t,\cdot) \frac{\partial \phi}{\partial x}(\cdot), \frac{\partial \phi}{\partial x}(\cdot) \psi(t,\cdot) \right\rangle \\ & \leq \left\langle \psi^2(t,\cdot), \frac{\partial}{\partial x}(\phi \frac{\partial \phi}{\partial x}) \right\rangle = \frac{1}{2} \left\langle \psi^2(t,\cdot), \frac{\partial^2 (\phi^2)}{\partial x^2} \right\rangle.
\end{split}
\end{equation}
By the reverse Fatou Lemma, we then obtain that
\begin{equation}
\limsup_{m \rightarrow \infty} \int_0^t \int_0^{\infty} \psi(s,x) \frac{\partial^2 d_m}{\partial x^2}(s,x) \textrm{d}x \textrm{d}s  \leq \frac{1}{2} \int_0^t \int_0^{\infty} \psi^2(s,x) (\phi^2)^{\prime \prime}(x) \textrm{d}x.
\end{equation}
Altogether, we have shown that 
\begin{equation}
\int_0^{\infty} \psi^2(t,x) \phi^2(x) \textrm{d}x \leq \frac{1}{2} \int_0^{\infty} \psi^2(t,x) \phi^2(x) \textrm{d}x + \frac{1}{2} \int_0^t \int_0^{\infty} \psi^2(s,x) (\phi^2)^{\prime \prime}(x) \textrm{d}x.
\end{equation}   
A simple rearrangement of this gives the result for $t<T$. The result for $t=T$ then follows by a simple application of the MCT.
\end{proof}

We want for this inequality to hold when we test with functions which are non-zero at zero. This will produce an extra boundary term on the right hand side, but this will be negative, so the previous inequality will still hold.

\begin{cor}
Let $(z_1, \eta_1)$ and $(z_2,\eta_2)$ be two solutions to the obstacle problem with obstacle $v \in \mathscr{C}_r^T$. Set $\psi(t,x):= z_1(t,x)-z_2(t,x)$. Then, for $\phi \in C_c^{\infty}([0,\infty))$, we have that
$$\int_0^{\infty} \psi^2(t,x) \phi^2(x) \textrm{d}x \leq \int_0^t \int_0^{\infty} \psi^2(s,x) (\phi^2)^{\prime \prime}(x) \textrm{d}x\textrm{d}s.$$
\end{cor}
\begin{proof}
We take an approximating sequence for $\phi$, using $\phi_n \in C_c^{\infty}((0,\infty))$. Doing so carefully, we are able to see that $(\phi^2)^{\prime \prime}$ gives an extra negative contribution in the limit, of $$-\lim\limits_{\epsilon \rightarrow 0} \frac{1}{\epsilon}\int_0^t \psi^2(s,\epsilon) \textrm{d}s.$$ In particular, the inequality still holds when testing with such $\phi$.
\end{proof}

\begin{prop}\label{obstacleunique}
Let $r>0$. The parabolic obstacle problem with exponential growth $r$ has a unique solution.
\end{prop}
\begin{proof}
The ideas for this proof borrow from those in Theorem 5.3 of \cite{Otobe}. Once again, $\psi= z_1-z_2$ where $(z_1,\eta_1)$ and $(z_2,\eta_2)$ are two solutions to the obstacle problem with exponential growth $r$. Let $\delta >0$. We want to apply the prevous lemma to the function $e^{-(r+\delta)x}$. Let $\phi_n$ be a sequence of functions in $C_c^{\infty}([0,\infty))$ such that $\phi_n \rightarrow \phi$ in $H^{2,2}((0,\infty))$, where $\phi(x) = e^{-\delta x}$. We then have that 
\begin{equation}
\left| \int_0^{\infty} \psi^2(t,x) e^{-2rx} \left(\phi^2(x)- \phi^2_n(x) \right) \textrm{d}x \right| \leq \sup\limits_{t \in [0,T]} \sup\limits_{x \geq 0} \left[ \psi^2(t,x) e^{-2rx} \right] \times \int_0^{\infty} (\phi^2(x) - \phi^2_n(x)) \textrm{d}x.
\end{equation} 
This converges to zero as $n \rightarrow \infty$. Note that for a $C^2$ function $\varphi$, we have that 
\begin{equation}
((e^{-ry}\varphi)^2)^{\prime \prime}(y)=e^{-2ry}\left( 4r^2 \varphi^2(y)-8r\varphi(y)\varphi^{\prime}(y)+2\varphi^{\prime \prime}(y) \varphi(y) +2 (\varphi^{\prime}(y))^2  \right)
\end{equation}
Since 
\begin{equation}
\sup\limits_{t \in [0,T]} \sup\limits_{x \geq 0} \left[ \psi^2(t,x) e^{-2rx} \right] < \infty 
\end{equation} 
almost surely and $\phi_n \rightarrow \phi$ in $H^{2,2}((0,\infty))$, we have that 
\begin{multline}
\int_0^T \int_0^{\infty} \psi^2(s,y) e^{-2ry}\left( 4r^2 \phi_n^2(y)-8r\phi_n(y)\phi_n^{\prime}(y)+2\phi_n^{\prime \prime}(y) \phi_n(y) +2 (\phi_n^{\prime}(y))^2  \right) \textrm{d}s \textrm{d}y \\ \rightarrow \int_0^T \int_0^{\infty} \psi^2(s,y) e^{-2ry}\left( 4r^2 \phi^2(y)-8r\phi(y)\phi^{\prime}(y)+2\phi^{\prime \prime}(y) \phi(y) +2 (\phi^{\prime}(y))^2  \right) \textrm{d}s \textrm{d}y.
\end{multline} 
Hence, the inequality from Lemma 2 still holds with $\phi(x)= e^{-(r+\delta)x}$. Applying the result with this function, we obtain that, for $t \in [0,T]$, 
\begin{equation}
\int_0^{\infty} \psi^2(t,x) e^{-2(r+\delta)x} \textrm{d}x \leq C_{r, \delta} \int_0^t \int_0^{\infty} \psi^2(s,x) e^{-2(r+\delta)x} \textrm{d}x\textrm{d}s.
\end{equation}
By Gronwall, we see that $\int_0^{\infty} \psi^2(t,x) e^{-2(r+\delta)x} \, \textrm{d}x=0$ for all $t \in [0,T]$. Therefore, $\psi=0$, so we have uniqueness for our problem.
\end{proof}

\begin{proof}[Proof of Theorem \ref{Obstacle INfinite}]
By Proposition \ref{infiniteobstacleexist} and Proposition \ref{obstacleunique}, we have existence and uniqueness of a solution. Suppose now that $v_1$, $v_2 \in \mathscr{C}_r^T$ and $z_1$, $z_2$ are solutions to the associated obstacle problems. Then, as in the proof of Proposition \ref{infiniteobstacleexist} and applying the estimate from Proposition \ref{obstaclepenalisationcontrol}, we have for $\epsilon >0$ functions $z_1^{\epsilon}$ and $z_2^{\epsilon}$ such that for $i=1,2$ 
\begin{equation}
z_i(t,x) = \lim\limits_{\epsilon \downarrow 0} z_i^{\epsilon}(t,x)
\end{equation}
and 
\begin{equation}
\|z_1^{\epsilon}-z_2^{\epsilon}\|_{\mathscr{C}_r^T} \leq C_{r,T} \|v_1-v_2\|_{\mathscr{C}_r^T}.
\end{equation}
Letting $\epsilon \downarrow 0$, we obtain that 
\begin{equation}
\|z_1-z_2\|_{\mathscr{C}_r^T} \leq C_{r,T} \|v_1-v_2\|_{\mathscr{C}_r^T}.
\end{equation}
This concludes the proof.
\end{proof}

\end{document}